\documentclass{amsart}
\usepackage{stmaryrd}
\usepackage{adjustbox}
\usepackage[utf8]{inputenc}
\usepackage{graphicx}
\usepackage{xcolor}
\usepackage{amssymb}
\usepackage{amsthm}
\usepackage{amsfonts}
\usepackage{amsmath}
\usepackage{verbatim}
\usepackage{multirow}
\usepackage{todonotes}
\usepackage{float}
\usepackage{wasysym}
\usepackage{hyperref}
\usepackage{enumerate}
\usepackage{enumitem}

\usetikzlibrary{circuits.ee.IEC}

\usetikzlibrary{shapes}
\usetikzlibrary{intersections}
\usetikzlibrary{svg.path}
\usetikzlibrary{decorations.markings}
\usetikzlibrary{hobby}

\usepackage[square,numbers]{natbib}

\newtheorem{theorem}{Theorem}
\newtheorem{proposition}{Proposition}
\newtheorem{definition}{Definition}
\newtheorem{lemma}{Lemma}
\newtheorem*{remark}{Remark}

\newtheorem*{corollary}{Corollary}
\newtheorem*{example}{Example}
\newtheorem{assumption}{Assumption}

\newcommand{\E}{\mathbb{E}}
\newcommand{\N}{\mathbb{N}}
\newcommand{\bR}{\mathbb{R}}
\newcommand{\bC}{\mathbb{C}}
\newcommand{\bX}{\mathbb{X}}
\newcommand{\bW}{\mathbf{W}}
\newcommand{\bT}{\mathbb{T}}

\newcommand{\cF}{\mathcal{F}}

\newcommand{\bx}{\mathbf{x}}
\newcommand{\by}{\mathbf{y}}

\newcommand{\cA}{\mathcal{A}}
\newcommand{\cP}{\mathcal{P}}

\newcommand{\fs}{\mathfrak{s}}
\newcommand{\fd}{\mathfrak{d}}

\newcommand{\norm}[1]{\left\lVert#1\right\rVert}
\newcommand{\sprod}[2]{\langle #1, #2 \rangle}

\begin{document}
\title[Genus expansion for non-linear random matrix theory]{Genus expansion for non--linear random matrix ensembles with applications to neural networks}

\author{Nicola Muça Cirone}
\address{Department of Mathematics, Imperial College London}
\email{n.muca-cirone22@imperial.ac.uk}

\author{Jad Hamdan}
\address{Mathematical Institute, University of Oxford}
\email{hamdan@maths.ox.ac.uk}

\author{Cristopher Salvi}
\address{Department of Mathematics, Imperial College London}
\email{c.salvi@imperial.ac.uk}

\begin{abstract}
We present a unified approach to studying certain non-linear random matrix ensembles and associated random neural networks at initialization.

This begins with a novel series expansion for neural networks which generalizes Faà di Bruno’s formula to an arbitrary number of compositions. The role of monomials is played by random multilinear maps indexed by directed graphs, whose edges correspond to random matrices. Crucially, this expansion linearises the effect of the activation functions, allowing for the direct application of {Wick’s principle} and the genus expansion technique.

As an application, we prove several results about neural networks with random weights. We first give a new proof of the fact that they converge to Gaussian processes as their width tends to infinity. Secondly, we quantify the rate of convergence of the Neural Tangent Kernel \cite{jacot2018neural} to its deterministic limit in Frobenius norm. Finally, we compute moments of the limiting spectral distribution of the Jacobian (only the first two of which were previously known), expressing them as sums over non-crossing partitions.

All of these results are then generalized to the case of neural networks with sparse and non-Gaussian weights, under moment assumptions. 
\end{abstract}

\maketitle
\tableofcontents

\section{Introduction}
\label{sec:intro}
\subsection{Randomly initialised neural networks and non-linear random matrix theory} 

Deep neural networks (NNs) whose weights' and biases' entries are initialised as appropriately rescaled,  independent and identically distributed (i.i.d.) Gaussian random variables converge to Gaussian processes (GPs) as their width tends to infinity. This well-known fact was originally observed by \citet{neal2012bayesian} for shallow feedforward networks and more recently by \citet{matthews2018gaussian} for multi-layer feedforward networks, by \citet{novak2018bayesian} and \cite{garriga2018deep} for deep convolutional networks, and by \citet{yang2019wide} for more general architectures, including recurrent and attention-based networks. The study of such random initializations is of practical interest, as a better understanding of their properties can help make design choices that improve training. For instance, the concentration of the spectrum of the network's input-output Jacobian can help prevent vanishing and exploding gradients \cite{glorot2010understanding, saxe2013exact, pennington2018emergence}. Other relevant objects which have been studied include the Gram \cite{PenningtonWorah} and Fisher information \cite{pennington2018spectrum} matrices, as well as the \textit{neural tangent kernel} (NTK) \cite{jacot2018neural}. From a theoretical viewpoint, these offer natural non-linear generalizations of the classical questions of random matrix theory (RMT) and free probability, and as such have attracted the attention of both communities \cite{PasturGaussian, PasturIID, BenigniPeche, WangZhu, ZhouFanWang}.

\bigskip 

Albeit of a similar nature, these results have been derived using a wide array of tools which do not always offer a clear path towards universality (meaning initialisations with non-Gaussian random weights). The goal of this paper is to propose a unifying framework that is naturally suited to this task, and is built on the \textit{genus expansion} technique.

This technique has its roots in the connection between matrix integrals and the enumeration of (combinatorial) maps, which was first discovered in the context of quantum field theory (see \cite{feynman,thooft, itzykson}, as well as \cite{zvonkin} for an accessible introduction to the subject). The link to random matrix theory was later established by Harer and Zagier \cite{harerzagier} in a seminal work investigating moduli spaces of curves, and has since been used to study various matrix ensembles and their asymptotic first and second-order freeness (e.g. \cite{mingospeicher, nicaspeicher, RedelmeierI, RedelmeierII}, as well as \cite{Dubach2021} and the references therein).
Roughly speaking, the approach consists of expanding traces of random matrix products and evaluating their expectation using Wick's principle. This yields a sum whose terms are in bijection with a set of graphs, and one determines which terms are of leading order by embedding their corresponding graphs into surfaces and computing their Euler characteristic. 

In a recent work, Dubach and Peled \cite{Dubach2021} used this approach to study products of non-Hermitian random matrices, and in particular showed that it can be adapted to non-Gaussian and sparse matrices with relative ease. This would, at least in principle, make it ideal for the study of multilayer neural networks, but the presence of non-linear activations in that setting prohibits one from being able to apply Wick's formula.

We circumvent this problem by developing a graphical language to express a large class of matrix/vector products, and deriving an expansion for neural networks in terms of this language. This linearises the effect of the activation functions, allowing for the use of Wick’s formula and the connection to the enumeration of maps to be made. To the best of our knowledge, this is the first use of the genus expansion technique in deep learning and non-linear RMT. A high-level overview of our approach is given in section \ref{sec:intro_mainideas} below.

\subsection{Main results}
Fix sequences $(\varphi_{\ell}: \bR \to \bR ~|~ \ell \in \N_{>0})$ of polynomial \emph{activation functions}, $(N_\ell \in \N_{>0} ~|~  \ell \in \N)$ of \textit{layer dimensions} and $(W_{\ell} \in \mathbb{R}^{N_{\ell+1}\times N_{\ell}} ~|~ \ell \in \N)$ of weight matrices.
We define a \emph{feed-forward neural network} $\Phi_L$ of depth $L$ by the recursion
\begin{equation}\label{intro:def:NN}
     \Phi_0(\bx)=W_0\bx,\quad \Phi_{\ell+1}(\bx)=W_{\ell+1}\varphi_{\ell+1}(\Phi_{\ell}(\bx)),
\end{equation}
where each $\varphi_{\ell}$ is applied entry-wise. We omit bias terms and restrict ourselves to polynomial activations for simplicity here (generalizations are discussed in Section \ref{sec:informal_ext}).

To illustrate the proposed method, we prove three results on NNs initialised with sparse random weights, beginning with the following Gaussian process limit.
\begin{theorem}[Gaussian process limit for sparse neural networks]
\label{thm:general_GP}
Let $N_\ell = N$ when $\ell>0$, and assume that each $W_\ell$ has i.i.d. entries drawn from a symmetric, centred distribution with finite moments and variance $\frac{1}{N}\mathbf{1}(\ell>0)+\mathbf{1}(\ell=0)$.

Then for any $M,L \geq 1$ we have 
 \begin{equation}
        ([\Phi_{L}]_1,...,[\Phi_{L}]_M) \xrightarrow[N \to \infty]{d} \mathcal{GP}(0,K_{L} \otimes \mathrm{Id}_M)
    \end{equation}
    where the right-hand side is a Gaussian Process indexed on $\bR^{N_0}$, with diagonal covariance function defined by
        \begin{gather}\label{eq:intro_GPKer}
        K_{0}(\bx,\by) = \sprod{\bx}{\by}_{\bR^{N_0}}, ~
        K_{\ell+1}(\bx,\by) = \E\left[ \varphi_{\ell+1}(X_{\ell}) \varphi_{\ell+1}(Y_{\ell}) \right] \\
        (X_{\ell},Y_{\ell}) \sim \mathcal{N}\left(0,
        \begin{bmatrix}
            K_{\ell}(\bx,\bx) & K_{\ell}(\bx,\by)
            \\
            K_{\ell}(\by,\bx) & K_{\ell}(\by,\by)
        \end{bmatrix}\right),
    \end{gather}
    and $\mathrm{Id}_M$ is the $M\times M$ identity matrix.
Furthermore, the same result holds if the weight matrices in \eqref{intro:def:NN} are replaced by $\tilde W_\ell := W_\ell \odot \frac{1}{\sqrt{p_N}} B_\ell$, where $W_\ell$ are as above and the $B_\ell$ are independent matrices with i.i.d., Bernoulli distributed entries with parameter $p_N$ satisfying $N p_N \to \infty$.   
\end{theorem}
\begin{proof}The claim for $W_\ell$ with Gaussian entries is Theorem \ref{thm:GPlimit}, and the corollaries in Section \ref{sec:extensions} generalize this to the non-Gaussian and sparse cases.
\end{proof}

This adds to the growing list of generalisations of the result of Matthews et al. \cite{matthews2018gaussian} to non-Gaussian settings, such as that of Huang \cite{huang} to orthogonal weights, and Hanin \cite{Hanin_annals} to weights with i.i.d. entries satisfying finite moment assumptions. More recently, Nait-Saada, Naderi and Tanner \cite{saada2024beyond} encompassed both of these results by showing that one can relax the i.i.d. assumption to a class of weights which they call \textsc{Pseudo-IID}. In particular, this class includes \textit{structured sparse} weights, making their work the first to rigorously show that the Gaussian process limit holds in a sparse setting. That said, while their result holds for more general activations than the ones considered here, it only deals with sparsification using a fixed binary mask $B$, whereas Theorem \ref{thm:GPlimit} allows for masks $B_\ell$ whose expected proportion of ones decreases as $N_\ell$ tends to infinity.

\bigskip 

Our second result concerns the NTK (\cite{du2018gradient, jacot2018neural}), which is defined as the matrix representation $\Theta_{L}(\bx,\by)\in\bR^{N_{L+1} \times N_{L+1}}$ of
\begin{equation}\label{def:NTK}
    \sum_{\ell=0}^{L} \lambda_{\ell}^2 (\mathrm{d}\Phi_{L}(\bx))_{W_{\ell}}\circ(\mathrm{d}\Phi_{L}(\by))_{W_{\ell}}^{\top},
\end{equation}
where $(\mathrm{d}\Phi_{L}(\by))_{W_{\ell}}^{\top}$ denotes the adjoint of $(\mathrm{d}\Phi_{L}(\by))_{W_{\ell}}$, and the $(\lambda_\ell)_\ell$ are the so-called \emph{layer-wise learning rates}. We show that at initialisation and under ``NTK parametrisation", $\Theta_L$ converges in $L^2$ to a deterministic kernel matrix (in the space of $N_{L+1}\times N_{L+1}$ matrices equipped with Frobenius norm), at a rate that is inversely proportional to the layer width. Note that the implicit constant in this estimate depends on the depth $L$.

\begin{theorem}[Convergence in $L^2$ of the NTK at intialisation]\label{thm:general_NTK}
For each $0<\ell\leq L$, let $N_\ell = N$ and assume that  $W_\ell$ has i.i.d. entries drawn from a symmetric, centered distribution with finite moments and variance $\sigma_\ell^2=\frac{1}{N}\mathbf{1}(\ell>0)+\mathbf{1}(\ell=0)$.  Moreover, assume that the layer-wise learning rates are $\lambda_\ell=\sigma_\ell^2$.

Then 
\begin{equation} 
    \mathbb{E}\bigg\{\Big\Vert\Theta_{L}(\bx,\by) - \Theta_{L}^{\infty}(\bx,\by) \otimes \mathrm{Id}_{N_{L+1}}\Big\Vert^2\bigg\}= \mathcal{O}\bigg(\frac{N_{L+1}^2}{N}\bigg)
\end{equation}
where $\Vert\cdot\Vert$ denotes the Frobenius norm, and $\Theta_{L}^{\infty}:\mathbb{R}^{N_0}\times\mathbb{R}^{N_0}\to\mathbb{R}$ is a scalar-valued kernel defined recursively as
\begin{equation}
    \Theta_0^{\infty}(\bx,\by) = \sprod{\bx}{\by}_{\bR^{N_0}}, \quad \Theta_{L}^{\infty}(\bx,\by) = K_{L}(\bx,\by) + \dot K_{L}(\bx,\by)\Theta_{L - 1}^{\infty}(\bx,\by).
\end{equation}
($\dot K_{\ell}$ is defined in the same way as $K_{\ell}$ only  substituting $\varphi_{\ell}$ for $\varphi'_{\ell}$ in (\ref{eq:intro_GPKer}).)

The same result holds with a rate of $\mathcal{O}(N_{L+1}^2/(Np_N))$ if the weight matrices in \eqref{intro:def:NN} are replaced by $\tilde W_\ell := W_\ell \odot \frac{1}{\sqrt{p_N}} B_\ell$, where $W_\ell$ are as above and the $B_\ell$ are independent matrices with i.i.d., Bernoulli distributed entries with parameter $p_N$ satisfying $N p_N \to \infty$. 
\end{theorem}
\begin{proof}The claim for $W_\ell$ with Gaussian entries is Theorem \ref{thm:ntk}, and the corollaries in Section \ref{sec:extensions} generalize this to the non-Gaussian and sparse cases.
\end{proof}

Previous results regarding the NTK at initialization have only been shown for Gaussian and orthogonal weights  \cite{jacot2018neural, huang}. With the caveat of only holding for polynomial activations, our result is one of the few quantitative results on the NTK for deep networks, and to the best of our knowledge, the only result to hold with universal and sparse weights.

\bigskip

We then move on to the problem of analysing the Jacobian spectrum of (the post-activations of) $\Phi_L$. Defining $\mathbf{J}_{L,\bx} := \mathrm{d}(\varphi_{L}\circ \Phi_{L-1})_{\bx}$, we’re interested in the macroscopic behaviour of the squared singular values of $\mathbf{J}_{L,\bx}$, and study the empirical spectral distribution of $\mathbf{J}_{L,\bx}\mathbf{J}_{L,\bx}^\top$, defined as
\begin{equation}\label{def:ESDjacobian}
    	\rho_L := \frac{1}{{N}}\sum_{I=1}^N \delta_{\xi_i}
\end{equation}
where $\{\xi_1, \dots, \xi_N\}$ are the eigenvalues of $\mathbf{J}_{L,\bx}\mathbf{J}_{L,\bx}^\top$ and $\delta_{\xi_i}$ denotes a Dirac mass on $\xi_i$.

This measure is known to converge to a deterministic limiting measure as $N\to \infty$, following the works of Collins and Hayase \cite{collins} and Pastur \cite{PasturGaussian} for orthogonal and Gaussian weights, respectively. These works justify an asymptotic freeness assumption, under which Pennington et al. in \cite{pennington2018emergence} first identified the limit of $\rho_L$ using the analytic machinery of free probability. The result of \cite{PasturGaussian} was then extended to non-Gaussian weights in a recent work of Pastur and Slavin \cite{PasturIID}. 

All of these works characterise the limiting measure through an implicit functional equation for its Stieltjes transform. This allows one to compute the first two moments $m_{1,L}$ and $m_{2,L}$ by expanding and solving for coefficients (see Equation (21) in \cite{pennington2018emergence}), but this approach breaks down for higher moments. Our main result (Theorem \ref{thm:general_jacobian}) computes the entire sequence of moments for the first time, by expressing them as weighted sums over non-crossing partitions. As with our previous two results, we also show that these formulae hold for non-Gaussian and sparse weights.

\begin{theorem}[Moments of the limiting spectral distribution of the Jacobian]\label{thm:general_jacobian} 

For each $\ell \geq 0$, assume that $N_\ell = N$ and that  $W_\ell$ has i.i.d. entries drawn from a symmetric, centered distribution with finite moments and variance ${1}/{N}$. Let $\mathbf{x}\in \mathbb{R}^N$ be such that $\frac{1}{N}\langle\mathbf{x},\mathbf{x}\rangle\to x^2\in \mathbb{R}$, and define $\rho_L$ as in Equation \eqref{def:ESDjacobian}.

Then for every $k\geq 0$, 
\[
    \int t^k\mathrm{d}\rho_L(t)\underset{N\to\infty}{\longrightarrow} m_{k,L}(x)
\]
where $\{m_{k,L}(x)\}_{k,L}$ is defined by the recursion in Equation (\ref{eq:recursion}). 

Furthermore, the same result holds if the weight matrices in \eqref{intro:def:NN} are replaced by $\tilde W_\ell := W_\ell \odot \frac{1}{\sqrt{p_N}} B_\ell$, where $W_\ell$ are as above and the $B_\ell$ are independent matrices with i.i.d., Bernoulli distributed entries with parameter $p_N$ satisfying $N p_N \to \infty$.   

\end{theorem}
\begin{proof}The claim for $W_\ell$ with Gaussian entries is Theorem \ref{thm:jacobianmoments}, and the corollaries in Section \ref{sec:extensions} generalize this to the non-Gaussian and sparse cases.
\end{proof}

This recursion sheds light on the relationship between the moments $m_{k,L}$ and the Fuss-Catalan numbers $\mathrm{FC}_{L}^{(k)}$ (see \cite{nicaspeicher}): they are obtained by inserting activation-dependent coefficients in the latter's defining recursion (see Remark \ref{eq:fusscatalan})\footnote{when expressed as a sum over non-crossing partitions of $k$ elements.}. The limit of $\rho_L$ can thus be seen as a non-linear generalisation the Fuss-Catalan or Raney distribution, which is known to be the universal first-order limit of squared singular values for products of Ginibre matrices (in the language of free probability, it is the $L$-fold free multiplicative convolution of the Marchenko-Pastur law). 

\subsection{Overview of our method.} \label{sec:intro_mainideas}

\subsubsection{A graphical language for neural network computations.} 
The idea of using a graphical language to simplify computations involving multilinear maps is not new, dating back to at least the 1970s with the introduction of Penrose diagrams \cite{penrose},
which have more recently been applied in the context of machine learning (see \cite{biamonte2017tensornetworksnutshell, Cichocki_2016}).
 As discussed earlier, graphs have also been used to evaluate expectations of products of Gaussian variables, and this forms the basis of the genus expansion technique.

The graphs that we introduce are novel and accomplish both of these tasks at once. On the one hand, they can be used to express deterministic products and operations involving multilinear maps. On the other, when dealing with tensors with Gaussian entries,  the expectation of these operations can once again be expressed in terms of graphs (in the sense of equation (\ref{eq:wickexpansion})). We explain this briefly below, deferring to Section \ref{sec:graph_dictionary} for more details.

In what we call a \textit{product} graph $G=(V,E)$, edges will correspond to matrices and vertices to vectors, which we call the inputs of their respective edge/vertex. The graph’s structure then dictates a well-defined product involving these inputs, the result of which we call the \textit{value} of the graph and denote by $\mathbf{W}_G$. For instance, a path of length $k$ can be used to express an (ordinary) product of $k$ matrices, while trees can be used to express Hadamard (entrywise) products (this is depicted in Figures \ref{fig:path-graph} and \ref{fig:basic-tree}).
 
If we omit inputs for some vertices and edges of the graph and view them as variables, then the resulting graph can be associated to a (multi)linear map and we call it an \textit{operator graph}. Differentiation, composition and other operations involving these maps then turn out to be expressible using simple manipulations of their corresponding graphs (composition, for instance, reduces to attaching graphs by a vertex), as explained in Section  \ref{subsec:operator_graph} and the figures therein. 

\begin{remark}
    After completing this work, we were made aware of the theory of traffic probability \cite{male2018trafficdistributionsindependencepermutation}, with which the graphical formalism developed here overlaps significantly. We leave the task of reconciling both theories to future work, but point out similarities between the two across Section \ref{sec:graph_dictionary}.
\end{remark}

\subsubsection{Graph expansions of neural networks.}
The connection to neural networks is made by expanding their output at a given input $\bx$ as a linear combination of product graphs
\begin{equation}\label{eqn:intro_nn_graph_expansion}
    \Phi(\bx)= \sum_{G\in \mathcal{F}} c(G)\bW_G
\end{equation}
for some family of graphs $\mathcal{F}$ and combinatorial factors $c(G)$. 

This is achieved in Theorem \ref{thm:tree_exp}, which generalizes Faà di Bruno’s formula (see \cite{Constantine1996AMF, JOHNSTON2022108080}) to the case of an arbitrary number of compositions. In similar tasks, trees have been shown to be a natural combinatorial tool to keep track of terms (see the literature on Butcher series  \cite{mclachlan2017butcherseriesstoryrooted, GubTrees}, and, more generally, on Runge-Kutta methods for ordinary differential equations \cite{HairerBook}), and this is reflected here in the fact that $\mathcal{F}$ (in Eq. (\ref{eqn:intro_nn_graph_expansion})) turns out to be a set of rooted trees.

By applying our previously mentioned graphical rules to each term in this sum, we derive similar expansions for various related quantities, namely the $k$-th coordinate of $\Phi(\bx)$, the neural tangent kernel, and the trace of the input-output Jacobian of $\Phi$ times its transpose, raised to an arbitrary power. The moment method then reduces the study of the asymptotic distribution of these quantities to computing $\E \bW_G$ for various graphs $G$.

\subsubsection{Wick’s principle and genus expansion.}\label{subsec:wick_intro} When $G$ is a product graph whose edge inputs have Gaussian entries, our main tool to evaluate $\E\bW_G$ is Wick's principle, which reduces the expectation of products of Gaussian variables to the sum of their pairwise covariances. Applied to $\bW_G$, it yields the following simple identity
\begin{equation}\label{eq:wickexpansion}
    \E{\bW_G} = \sum_{\phi} \bW_{G_\phi},
\end{equation}
(see Theorem \ref{thm:wickexpansion}), where the sum is taken over \textit{admissible} pairings $\phi$ of the edges of $G$ (see Def. \ref{def:admissible_pairings}), and $G_\phi$ is the graph obtained from $G$ after identifying edges paired by $\phi$ (meaning that we consider such edges to be the same edge in $G_\phi$). Under additional assumptions on $G$ (c.f. Assumption \ref{assumption:genus_graph}), we find that $\bW_{G_\phi}=\sigma_G N^{|V(G_\phi)|}$ for every $\phi$, where $|V(G_\phi)|$ is the number of vertices in $G_\phi$ and $\sigma_G$ is a variance parameter. The leading order of $\bW_G$ is thus determined by the pairings for which $|V(G_\phi)|$ is maximised. 

Instead of counting this quantity directly, it turns out to be much simpler to embed the graph onto a surface $S_\phi$ (as defined in Equation \ref{eq:sphi}) and to then compute $|V(G_\phi)|$ using the \textit{Euler characteristic formula}
\[
    |V(G_\phi)|-|E(G_\phi)|+f(G_\phi:S_\phi) = 2-2g(S_\phi),
\]
where $|E(G_\phi)|$ is the number of edges of $G_\phi$, $f(G_\phi:S_\phi)$ the number of faces of $G_\phi$ in $S_\phi$ and $g(S_\phi)$ the genus of $S_\phi$. This formula allows us to identify which $\phi$ give rise to leading and sub-leading order terms in Eq. (\ref{eq:wickexpansion}), which we call \textit{atomic} and \textit{bi-atomic} pairings, respectively, following \cite{Dubach2021}. We use this to give a more explicit version of equation (\ref{eq:wickexpansion}), and to extend it to centred mixed moments $\E \{\prod_{G} (\bW_G-\E \bW_G)\}$ as well (see Lemma \ref{lemma:mixedmoments_genus}). Lastly, we combine these results to obtain a limit theorem for the joint moments of product graphs (Theorem \ref{thm:joint_gaussian_limit}), reminiscent of a celebrated result of Diaconis and Shahshahani \cite{diaconis} for traces of powers of random unitary matrices and its recent extension generalisation in \cite{Dubach2021} (Theorem 1.2).

When the edge inputs in $G$ are complex, non-Gaussian or sparse matrices (or any combination of the three), we show that the expressions for $\mathbb{E}\mathbf{W}_G$ and $\E \{\prod_{G} (\bW_G-\E \bW_G)\}$ hold up to an $o(1)$ error term (see Sections \ref{sec:complex_case}, \ref{sec:non_gaussian}, \ref{sec:sparse}, respectively). This allows us to extend all of our results to NNs with such weight matrices.
\subsection{Possible generalizations}\label{sec:informal_ext} We imposed two simplifying assumptions in order to better illustrate our method ( polynomial activations and the absence of biases in \eqref{intro:def:NN}). We briefly discuss how these can be lifted below.
\subsubsection{Random biases} Typically, one defines $\Phi_L$ as
\[
    \Phi_0(\mathbf{x})=W_0\mathbf{x},\quad \Phi_{\ell+1}(\mathbf{x})=W_{\ell+1}\varphi_{\ell+1}(\Phi_\ell(\mathbf{x}))+b_{\ell+1}
\]
for some vectors $\{b_\ell\}_{\ell}$ (the \textit{biases}) which can also be initialized at random. 
The method presented here can be adapted to this setting, albeit at the cost of additional combinatorial bookkeeping (namely a more complicated set of graphs $\mathcal{F}$ in \ref{eqn:intro_nn_graph_expansion}). When applying Wick's theorem to compute the resulting $\mathbb{E}\mathbf{W}_G$, one would also have to pair vertices first, before being able to pair edges as explained in \ref{subsec:wick_intro}.

\subsubsection{Non-polynomial activations}  Restricting to polynomial activations allowed us to ignore issues of convergence in Theorem \ref{thm:tree_exp}. That said, the expansion in the latter formally holds for any (reasonably well-behaved) function, and extending the results here to more commonly used activations (e.g. Sigmoid, ReLU) should thus be feasible by an approximation argument as was done by Benigni and Peché in \cite{BenigniPeche}. We present some numerical evidence that Theorem \ref{thm:general_jacobian} holds with ReLU activations in Section \ref{app:numerics}.

\subsection{Notation and nomenclature}

Given a matrix $A \in \bR^{N\times M}$ and a vector $v \in \bR^N$ we write $[A]_{ij}$ and $[v]_j$ for their $i,j$ and $j$-th coordinate, respectively. More generally, we will use square brackets with subscripts to denote coordinates of tensors.
We use the notation $\mathbf{I}_{M,N} \in \bR^{N\times M}, \mathbf{1}_M \in \bR^{N}$ to denote the matrix and vector having all entries equal to $1$ (omitting the subscripts whenever it does not hurt comprehension, and $\mathbf{I}_{M}=\mathbf{I}_{M,M}$), 
 $\mathbf{E}_{ij} \in \bR^{M \times N}$ and $\mathbf{e}_i \in \bR^N$ to denote the canonical basis matrices/vectors in $\mathbb{R}^N$, and $\mathrm{Id}_{N}$ to denote the identity matrix in $\mathbb{R}^{N\times N}$.

$\langle,\rangle$ will denote the standard inner product, with the space in subscript when it is not clear from the context. 
 If $A$ is a matrix with complex entries, we use $\bar{A}$ to denote its conjugate and  $A^*$ its Hermitian transpose. $\mathcal{N}(\mu,\sigma^2)$ will denote a Gaussian with mean $\mu$ and variance $\sigma^2$, and similarly $\mathcal{N}_\mathbb{C}(0,1)$ will denote a standard complex Gaussian (with real and imaginary parts each having variance $1/2$).

We will use standard asymptotic notation, writing $f(T)=o(g(T))$ to mean that $|f(T)/g(T)|\to_{T\to\infty} 0$ and $f(T)=\mathcal{O}(g(T))$ to mean that  $\limsup_{T\to\infty}|f(T)/g(T)|$ is bounded. 


For any positive integer $N$, we will use $[N]$ to denote the set $\{1,...,N\}$. Whenever $G$ is a directed graph, we let $V(G), E(G)$ denote its set of vertices and edges respectively, and omit the dependence on $G$ whenever the graph is clear from the context. For any edge $e=(u,v)$ we call $u$ the \textit{head} and $v$ the \textit{tail} of $e$, and say that $e$ is \textit{adjacent} to $u,v$ and vice versa. 

\bigskip 

A table compiling the notation that we introduce throughout the paper can be found in Appendix \ref{app:sect:table}.

\bigskip 
\textbf{Acknowledgements. } N.C. thanks William Turner for pointing him to \cite{Dubach2021} and J.H. thanks Adam Jones for helpful discussions.  C.S. is supported by Innovate UK (Proj ID 10073285). N.C. and J.H. are supported by the EPSRC Centre for Doctoral Training in Mathematics of Random Systems: Analysis, Modelling and Simulation (EP/S023925/1).


\section{Graphical descriptions of analytic operations}
\label{sec:graph_dictionary}
In this section we explain how various matrix--vector products can be expressed by means of directed graphs. 

Consider a directed graph $G=(V,E)$. We associate a vector $\mathbf{X}_v\in \bR^{\fd(v)}$ to each vertex $v\in V$ and a matrix $\mathbf{X}_e\in \bR^{\fd(u)\times \fd(v)}$ to each edge $e=(u,v)\in E$, calling $\mathbf{X}_e$ and $\mathbf{X}_v$ the \textit{inputs} of the edge $e$ and vertex $v$ respectively. Here, $\fd(v)$ is a positive integer which call the \textit{dimension} of the vertex $v$, and we extend $\fd$ to edges by using the shorthand $\fd(e)=(\fd(u),\fd(v))$ for $e=(u,v)\in E$ (in which case $\bR^{\fd(e)}:=\bR^{\fd(u)\times \fd(v)}$).

As we are often going to talk about vertices and edges simultaneously, it will be useful to have a single term to refer to both types of objects. For this purpose we will use the term \emph{cells}. Every cell in $C:= V\cup E$ will thus have a  dimension assigned to it by $\fd$, and this map is entirely determined by the values that it takes on vertices by definition.

As we will soon make precise, $G$ can be seen as describing a type of product of its inputs. This motivates the naming in the following definition, which summarizes what we have introduced so far. 
\begin{definition}[Product graph]\label{def:product-graph}
    A \emph{product graph} is a triple $(G,\fd, \mathfrak{C})$, where 
    \begin{itemize}
        \item $G=(V,E)$ is a directed graph with vertex set $V$ and edge set $E\subseteq V\times V$.
        \item  $\fd$ assigns a dimension to each vertex in $V$, and thus to each cell of $G$.
        \item $\mathfrak{C}=\{\mathbf{X}_c \in \bR^{\fd(c)} : c\in C =  V\cup E\}$ is a collection of inputs (vectors/matrices) associated to the cells $C$ of $G$.
    \end{itemize}
\end{definition}

When $\fd$, $\mathfrak{C}$ are clear or implied from the context, we will omit them and simply use $G$ to denote the product graph $(G,\fd, \mathfrak{C})$. To a product graph $G$ we can uniquely associate the following \textit{value}.

\begin{definition}\label{def:product}
    The value $\bW_G \in \mathbb R$ of a product graph $(G,\fd,\mathfrak{C})$ is defined as the following scalar
\begin{equation}\label{eq:product}
        \bW_G:= \sum_{(i_v)_{v\in V}}\prod_{(u,v)\in E}[\mathbf{X}_{(u,v)}]_{i_ui_v} \prod_{w\in V} [\mathbf{X}_{w}]_{i_w} =\sum_{\mathbf{i}_{V}} \prod_{c \in C}~
        [\mathbf{X}_c]_{i_c} \in \mathbb{R}
\end{equation}
where the sum on the right hand side is taken over all \textit{indexations} $\mathbf{i}_{V} = (i_v)_{v\in V} \in \prod_{v\in V} \{1,...,\mathfrak{d}(v)\}$ and
where for $c = (u,v) \in E$ we define $i_c := i_ui_v$.
\end{definition}

\begin{remark}
    The constraints imposed by $\mathfrak{d}$ on the dimensions of the elements of $\mathfrak{C}$ in 
    Definition \ref{def:product-graph} ensure that product $(\ref{eq:product})$ is well-defined. Furthermore, if $G$ is disconnected with connected components $\{G_i\}_i$, then $\bW_G=\prod_{i} \bW_{G_i}$. In the event that some of its inputs are random, $\bW_G$ then becomes a random variable.
\end{remark}

\begin{remark}
    Definition \ref{def:product} inherently depends on the bases relative to which the coordinates of the $\mathbf{X}_c$ are being taken. This will always be the canonical basis. In general, taking $A_v\in \mathrm{GL}(\mathbb{R}^{\mathfrak{d}(v)})$ for each $v\in V$ and substituting
    \begin{align}\label{eq:change_of_basis}
    \mathbf{X}_v\mapsto A_v^{-1}\mathbf{X}_v,\,\forall v\in V.\quad \mathbf{X}_{(u,v)}\mapsto A_u^{-1}\mathbf{X}_{(u,v)}A_v,\, \forall (u,v)\in E,
    \end{align}
    can change the value of $\mathbf{W}_G$. On the other hand, the latter is invariant to \eqref{eq:change_of_basis} if each $A_v$ is a permutation matrix: this mimics Lemma 1.4 in \cite{male2018trafficdistributionsindependencepermutation}, with their *-graph monomials corresponding to \eqref{def:product} without vertex inputs.
\end{remark}

We illustrate this definition with some examples. 
\begin{example}
If $G$ is the product graph in Figure \ref{fig:path-graph}, then $\bW_G=\mathbf{x}_a^{\top}W_3W_2W_1\mathbf{x}_d$ since
\[
\mathbf{x}_a^{\top}W_3W_2W_1\mathbf{x}_d = 
\sum_{\alpha=1}^{\fd(a)}
\sum_{\beta=1}^{\fd(b)}
\sum_{\gamma=1}^{\fd(c)}
\sum_{\delta=1}^{\fd(d)}
[W_3]_{\alpha,\beta} [W_2]_{\beta,\gamma} [W_3]_{\gamma,\delta} [\mathbf{x}_a]_{\alpha} [\mathbf{x}_d]_{\delta}.
\]
\begin{figure}[ht]
    \centering
    \tikzset{every picture/.style={line width=0.75pt}} 

\begin{tikzpicture}[x=0.75pt,y=0.75pt,yscale=-1,xscale=1]

\draw    (534.07,4659.91) -- (583.08,4660.16) ;
\draw [shift={(560.98,4660.05)}, rotate = 180.3] [color={rgb, 255:red, 0; green, 0; blue, 0 }  ][line width=0.75]    (4.37,-1.32) .. controls (2.78,-0.56) and (1.32,-0.12) .. (0,0) .. controls (1.32,0.12) and (2.78,0.56) .. (4.37,1.32)   ;
\draw    (485.07,4659.65) -- (534.07,4659.91) ;
\draw [shift={(511.97,4659.79)}, rotate = 180.3] [color={rgb, 255:red, 0; green, 0; blue, 0 }  ][line width=0.75]    (4.37,-1.32) .. controls (2.78,-0.56) and (1.32,-0.12) .. (0,0) .. controls (1.32,0.12) and (2.78,0.56) .. (4.37,1.32)   ;
\draw    (436.06,4659.39) -- (485.07,4659.65) ;
\draw [shift={(462.96,4659.53)}, rotate = 180.3] [color={rgb, 255:red, 0; green, 0; blue, 0 }  ][line width=0.75]    (4.37,-1.32) .. controls (2.78,-0.56) and (1.32,-0.12) .. (0,0) .. controls (1.32,0.12) and (2.78,0.56) .. (4.37,1.32)   ;
\draw  [color={rgb, 255:red, 0; green, 0; blue, 0 }  ,draw opacity=1 ][fill={rgb, 255:red, 0; green, 0; blue, 0 }  ,fill opacity=1 ] (487.69,4662.47) .. controls (489.25,4661.01) and (489.33,4658.58) .. (487.88,4657.02) .. controls (486.43,4655.47) and (483.99,4655.38) .. (482.44,4656.83) .. controls (480.88,4658.29) and (480.8,4660.72) .. (482.25,4662.28) .. controls (483.7,4663.83) and (486.14,4663.92) .. (487.69,4662.47) -- cycle ;
\draw  [color={rgb, 255:red, 0; green, 0; blue, 0 }  ,draw opacity=1 ][fill={rgb, 255:red, 0; green, 0; blue, 0 }  ,fill opacity=1 ] (438.69,4662.21) .. controls (440.24,4660.76) and (440.33,4658.32) .. (438.87,4656.76) .. controls (437.42,4655.21) and (434.99,4655.13) .. (433.43,4656.58) .. controls (431.88,4658.03) and (431.79,4660.47) .. (433.24,4662.02) .. controls (434.69,4663.58) and (437.13,4663.66) .. (438.69,4662.21) -- cycle ;
\draw  [color={rgb, 255:red, 0; green, 0; blue, 0 }  ,draw opacity=1 ][fill={rgb, 255:red, 0; green, 0; blue, 0 }  ,fill opacity=1 ] (536.7,4662.72) .. controls (538.26,4661.27) and (538.34,4658.83) .. (536.89,4657.28) .. controls (535.44,4655.72) and (533,4655.64) .. (531.44,4657.09) .. controls (529.89,4658.54) and (529.81,4660.98) .. (531.26,4662.53) .. controls (532.71,4664.09) and (535.15,4664.17) .. (536.7,4662.72) -- cycle ;
\draw  [color={rgb, 255:red, 0; green, 0; blue, 0 }  ,draw opacity=1 ][fill={rgb, 255:red, 0; green, 0; blue, 0 }  ,fill opacity=1 ] (585.71,4662.98) .. controls (587.26,4661.53) and (587.35,4659.09) .. (585.9,4657.53) .. controls (584.44,4655.98) and (582.01,4655.9) .. (580.45,4657.35) .. controls (578.9,4658.8) and (578.81,4661.24) .. (580.26,4662.79) .. controls (581.72,4664.35) and (584.15,4664.43) .. (585.71,4662.98) -- cycle ;

\draw (473.31,4639.81) node [anchor=north west][inner sep=0.75pt]  [rotate=-359.2]  {$\mathbf{1}_{\fd(b)}$};
\draw (522.05,4639.81) node [anchor=north west][inner sep=0.75pt]  [rotate=-0.29]  {$\mathbf{1}_{\fd(c)}$};
\draw (455.56,4663) node [anchor=north west][inner sep=0.75pt]  [font=\tiny,rotate=-359.58]  {$W_3$};
\draw (504.57,4663) node [anchor=north west][inner sep=0.75pt]  [font=\tiny,rotate=-0.73]  {$W_2$};
\draw (553.58,4663) node [anchor=north west][inner sep=0.75pt]  [font=\tiny,rotate=-0.89]  {$W_1$};
\draw (571.94,4640.25) node [anchor=north west][inner sep=0.75pt]  [rotate=-358.82]  {$\mathbf{x}_{d}$};
\draw (431.66,4640.04) node [anchor=north west][inner sep=0.75pt]  [rotate=-359.8]  {$\mathbf{x}_{a}$};

\end{tikzpicture}
    \caption{A product graph giving rise to a word of matrices multiplied by vectors on either side. $\mathbf{1}_{N}$ is the $N\times 1$ vector of ones.}
    \label{fig:path-graph}
\end{figure}
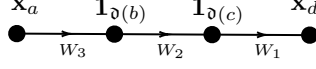
\end{example}

\begin{example}
   If $G$ is a tree, then the corresponding value $\bW_G$ is defined by means of Hadamard products of matrices. 
    For instance, taking $G$ to be the product graph in Figure \ref{fig:basic-tree}
    gives
    \[  
    \bW_G=\mathbf{x}_a^{\top} W_2\big(\mathbf{x}_{b} \odot (W_1\mathbf{x}_d)\astrosun(W_1\mathbf{x}_c)\big)
    \]
    where $A\astrosun B$ denotes the Hadamard product of $A$ and $B$, since
    \[
    \bW_G := \sum_{\alpha=1}^{\fd(a)}
    \sum_{\beta=1}^{\fd(b)}
    \sum_{\gamma=1}^{\fd(c)}
    \sum_{\delta=1}^{\fd(d)}
    [W_2]_{\alpha,\beta}[W_1]_{\beta,\gamma}[W_1]_{\beta,\delta} 
    [\mathbf{x}_{a}]_{\alpha}[\mathbf{x}_{b}]_{\beta}
    [\mathbf{x}_{c}]_{\gamma}[\mathbf{x}_{d}]_{\delta}.
    \]
    \begin{figure}[ht]\label{fig:tree}
        \centering
        \tikzset{every picture/.style={line width=0.75pt}} 

\begin{tikzpicture}[x=0.75pt,y=0.75pt,yscale=-1,xscale=1]

\draw  [color={rgb, 255:red, 0; green, 0; blue, 0 }  ,draw opacity=1 ][fill={rgb, 255:red, 0; green, 0; blue, 0 }  ,fill opacity=1 ] (459.69,4778.21) .. controls (461.24,4776.76) and (461.33,4774.32) .. (459.87,4772.76) .. controls (458.42,4771.21) and (455.99,4771.13) .. (454.43,4772.58) .. controls (452.88,4774.03) and (452.79,4776.47) .. (454.24,4778.02) .. controls (455.69,4779.58) and (458.13,4779.66) .. (459.69,4778.21) -- cycle ;
\draw    (457.06,4775.39) -- (457.06,4737.11) ;
\draw [shift={(457.06,4753.85)}, rotate = 90] [color={rgb, 255:red, 0; green, 0; blue, 0 }  ][line width=0.75]    (4.37,-1.32) .. controls (2.78,-0.56) and (1.32,-0.12) .. (0,0) .. controls (1.32,0.12) and (2.78,0.56) .. (4.37,1.32)   ;
\draw  [color={rgb, 255:red, 0; green, 0; blue, 0 }  ,draw opacity=1 ][fill={rgb, 255:red, 0; green, 0; blue, 0 }  ,fill opacity=1 ] (459.69,4739.92) .. controls (461.24,4738.47) and (461.33,4736.03) .. (459.87,4734.48) .. controls (458.42,4732.92) and (455.99,4732.84) .. (454.43,4734.29) .. controls (452.88,4735.74) and (452.79,4738.18) .. (454.24,4739.74) .. controls (455.69,4741.29) and (458.13,4741.38) .. (459.69,4739.92) -- cycle ;
\draw  [color={rgb, 255:red, 0; green, 0; blue, 0 }  ,draw opacity=1 ][fill={rgb, 255:red, 0; green, 0; blue, 0 }  ,fill opacity=1 ] (484.69,4714.92) .. controls (486.24,4713.47) and (486.33,4711.03) .. (484.87,4709.48) .. controls (483.42,4707.92) and (480.99,4707.84) .. (479.43,4709.29) .. controls (477.88,4710.74) and (477.79,4713.18) .. (479.24,4714.74) .. controls (480.69,4716.29) and (483.13,4716.38) .. (484.69,4714.92) -- cycle ;
\draw  [color={rgb, 255:red, 0; green, 0; blue, 0 }  ,draw opacity=1 ][fill={rgb, 255:red, 0; green, 0; blue, 0 }  ,fill opacity=1 ] (434.69,4714.92) .. controls (436.24,4713.47) and (436.33,4711.03) .. (434.87,4709.48) .. controls (433.42,4707.92) and (430.99,4707.84) .. (429.43,4709.29) .. controls (427.88,4710.74) and (427.79,4713.18) .. (429.24,4714.74) .. controls (430.69,4716.29) and (433.13,4716.38) .. (434.69,4714.92) -- cycle ;
\draw    (457.06,4737.11) -- (432.06,4712.11) ;
\draw [shift={(442.86,4722.91)}, rotate = 45] [color={rgb, 255:red, 0; green, 0; blue, 0 }  ][line width=0.75]    (4.37,-1.32) .. controls (2.78,-0.56) and (1.32,-0.12) .. (0,0) .. controls (1.32,0.12) and (2.78,0.56) .. (4.37,1.32)   ;
\draw    (457.06,4737.11) -- (482.06,4712.11) ;
\draw [shift={(471.26,4722.91)}, rotate = 135] [color={rgb, 255:red, 0; green, 0; blue, 0 }  ][line width=0.75]    (4.37,-1.32) .. controls (2.78,-0.56) and (1.32,-0.12) .. (0,0) .. controls (1.32,0.12) and (2.78,0.56) .. (4.37,1.32)   ;

\draw (469.56,4725.61) node [anchor=north west][inner sep=0.75pt]  [font=\tiny,rotate=-0.89]  {$W_1$};
\draw (440.06,4756.25) node [anchor=north west][inner sep=0.75pt]  [font=\tiny,rotate=-0.73]  {$W_2$};
\draw (430.56,4724.61) node [anchor=north west][inner sep=0.75pt]  [font=\tiny,rotate=-0.89]  {$W_1$};
\draw (415.66,4693) node [anchor=north west][inner sep=0.75pt]  [rotate=-359.8]  {$\mathbf{x}_{c}$};
\draw (479.87,4693) node [anchor=north west][inner sep=0.75pt]  [rotate=-359.8]  {$\mathbf{x}_{d}$};
\draw (463.87,4769.48) node [anchor=north west][inner sep=0.75pt]  [rotate=-359.8]  {$\mathbf{x}_{a}$};
\draw (461.87,4737.48) node [anchor=north west][inner sep=0.75pt]  [rotate=-359.2]  {$\mathbf{x}_{b}$};

\end{tikzpicture}
        \caption{Trees give rise to Hadamard products.}\label{fig:basic-tree}
    \end{figure}
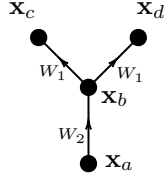
\end{example}
    
    Note that while the previous two examples can be described in terms of ordinary matrix/vector multiplication and entry-wise products, it isn't necessarily the case in general. For instance, taking $G$ to be the product graph in Figure \ref{fig:peculiar_graph}
    gives
    \begin{align*}
        \bW_G 
        &= \mathrm{Tr}\left\{
        \left[ \mathbf{x}_a \mathbf{1}^{\top} \odot W_1 \right]
        \left[ (\mathbf{x}_b \odot W_3\mathbf{x}_c)\mathbf{1}^{\top} \odot W_2 \right] \right\}
    \end{align*}

    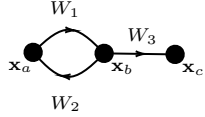
\begin{figure}[ht]
        \centering
        \tikzset{every picture/.style={line width=0.75pt}} 

\begin{tikzpicture}[x=0.75pt,y=0.75pt,yscale=-1,xscale=1]

\draw  [color={rgb, 255:red, 0; green, 0; blue, 0 }  ,draw opacity=1 ][fill={rgb, 255:red, 0; green, 0; blue, 0 }  ,fill opacity=1 ] (337.1,121.61) .. controls (338.72,123.39) and (341.47,123.52) .. (343.25,121.91) .. controls (345.04,120.29) and (345.17,117.54) .. (343.55,115.76) .. controls (341.94,113.98) and (339.19,113.84) .. (337.4,115.46) .. controls (335.62,117.07) and (335.49,119.83) .. (337.1,121.61) -- cycle ;
\draw  [color={rgb, 255:red, 0; green, 0; blue, 0 }  ,draw opacity=1 ][fill={rgb, 255:red, 0; green, 0; blue, 0 }  ,fill opacity=1 ] (372.73,122.12) .. controls (374.34,123.9) and (377.1,124.03) .. (378.88,122.42) .. controls (380.66,120.8) and (380.79,118.05) .. (379.18,116.27) .. controls (377.56,114.49) and (374.81,114.35) .. (373.03,115.97) .. controls (371.25,117.58) and (371.11,120.34) .. (372.73,122.12) -- cycle ;
\draw    (340.33,118.68) -- (375.95,119.19) ;
\draw [shift={(360.54,118.97)}, rotate = 180.82] [color={rgb, 255:red, 0; green, 0; blue, 0 }  ][line width=0.75]    (4.37,-1.32) .. controls (2.78,-0.56) and (1.32,-0.12) .. (0,0) .. controls (1.32,0.12) and (2.78,0.56) .. (4.37,1.32)   ;
\draw  [color={rgb, 255:red, 0; green, 0; blue, 0 }  ,draw opacity=1 ][fill={rgb, 255:red, 0; green, 0; blue, 0 }  ,fill opacity=1 ] (301.48,121.1) .. controls (303.1,122.88) and (305.85,123.01) .. (307.63,121.4) .. controls (309.41,119.78) and (309.55,117.03) .. (307.93,115.25) .. controls (306.31,113.47) and (303.56,113.33) .. (301.78,114.95) .. controls (300,116.56) and (299.86,119.32) .. (301.48,121.1) -- cycle ;
\draw    (304.7,118.17) .. controls (318.95,134.08) and (324.53,134.96) .. (340.33,118.68) ;
\draw [shift={(319,129.96)}, rotate = 7.24] [color={rgb, 255:red, 0; green, 0; blue, 0 }  ][line width=0.75]    (4.37,-1.32) .. controls (2.78,-0.56) and (1.32,-0.12) .. (0,0) .. controls (1.32,0.12) and (2.78,0.56) .. (4.37,1.32)   ;
\draw    (304.7,118.17) .. controls (320.19,102.89) and (325.39,102.97) .. (340.33,118.68) ;
\draw [shift={(325.18,107.14)}, rotate = 183.42] [color={rgb, 255:red, 0; green, 0; blue, 0 }  ][line width=0.75]    (4.37,-1.32) .. controls (2.78,-0.56) and (1.32,-0.12) .. (0,0) .. controls (1.32,0.12) and (2.78,0.56) .. (4.37,1.32)   ;

\draw (311.6,91) node [anchor=north west][inner sep=0.75pt]  [font=\scriptsize]  {$W_{1}$};
\draw (312,138.6) node [anchor=north west][inner sep=0.75pt]  [font=\scriptsize]  {$W_{2}$};
\draw (350.4,104.6) node [anchor=north west][inner sep=0.75pt]  [font=\scriptsize]  {$W_{3}$};
\draw (290.8,121.27) node [anchor=north west][inner sep=0.75pt]  [font=\scriptsize]  {$\mathbf{x}_{a}$};
\draw (377.95,122.59) node [anchor=north west][inner sep=0.75pt]  [font=\scriptsize]  {$\mathbf{x}_{c}$};
\draw (342.33,122.08) node [anchor=north west][inner sep=0.75pt]  [font=\scriptsize]  {$\mathbf{x}_{b}$};

\end{tikzpicture}
        \caption{A simple product graph leading to a more involved analytical expression for its value.}\label{fig:peculiar_graph}
    \end{figure}

\subsection{Operator graphs and their associated linear map} 
\label{subsec:operator_graph}

So far, we have associated a vector/matrix to each cell in our graphs. 
By \emph{freeing} some of these cells--meaning that we consider their inputs as variables--we can use graphs to define more general linear maps between tensors.

\bigskip

Consider for example the product graph $(G,\fd,\mathfrak{C})$ in Figure \ref{fig:basic-tree}, noting how the product defining $\bW_G$ is {linear} in $\bx_a$. If we consider $\bx_a$ to be a variable, this yields a well-defined linear map $\bR^{\fd(a)} \to \bR$, or equivalently a vector in $\bR^{\fd(a)}$.
In this particular case, the linear map in question would be $\bx \mapsto \bx^{\top} W_2\big(\mathbf{x}_{b} \odot (W_1\mathbf{x}_d)\astrosun(W_1\mathbf{x}_c)\big)$, and its vector representation is $W_2\big(\mathbf{x}_{b} \odot (W_1\mathbf{x}_d)\astrosun(W_1\mathbf{x}_c)\big)$. 

\bigskip

More generally, let $G$ be a directed graph with dimensions assigned by $\fd$, and let $\mathcal{F}$ be an (ordered) sequence of {free} cells of $G$ (meaning vertices and edges which are considered as variables), assuming that the remaining cells $c$ which do not belong to $\mathcal{F}$ are each \emph{fixed} to some input $\mathbf{X}_c$. Then for any input sequence $(\mathbf{X}_c)_{c\in \mathcal{F}},$ the product defining $\mathbf{W}_{(G,\fd,\{\mathbf{X}_c\}_{c\in C})}=\mathbf{W}_G$ is $|\mathcal{F}|-$linear in these inputs.

Any map defined this way will be scalar-valued, but we can generalize to vector-valued maps through partial evaluation. Namely, given a partition of $\mathcal{F}$ into two subsequences $\mathcal{F}_\mathrm{in}$ and $\mathcal{F}_{\mathrm{out}}$, we can consider the multilinear map
\begin{equation}\label{eq:associatedmap1}
    (\mathbf{X}_c)_{c\in \mathcal{F}_\mathrm{in}} \mapsto \big(\otimes_{c\in \mathcal{F}_{\mathrm{out}}}\mathbf{X}_c \mapsto \mathbf{W}_G\big).
\end{equation}
This is equivalent to considering the map
\begin{equation}\label{eq:associatedmap2}
       (\mathbf{X}_c)_{c\in \mathcal{F}_\mathrm{in}} \mapsto \mathbf{Y},\quad \langle \mathbf{Y}, \otimes_{c\in \mathcal{F}_\mathrm{out}}  \mathbf{X}_{\mathrm{c}}\rangle= \mathbf{W}_G\, \text{ for any } (\mathbf{X}_c)_{c\in \mathcal{F}_{\mathrm{out}}}    
\end{equation}
where $\mathbf{Y} \in \bigotimes_{c\in \mathcal{F}_{\mathrm{out}}} \bR^{{\fd}(c)}$ is the Riesz representation of $\big(\otimes_{c\in \mathcal{F}_{\mathrm{out}}}\mathbf{X}_c \mapsto \mathbf{W}_G\big)$. Note that the input and output dimensions of this map depend on the partition of $\mathcal{F}$ that is being taken. We refer to it as the \textit{operator associated to} $G$, which we define below. In what follows, we abuse notation slightly and write $c\notin \mathcal{F}$ to mean $c\in C\setminus(\mathcal{F}_{\mathrm{in}}\sqcup\mathcal{F}_{\mathrm{out}}).$

\begin{definition}[Operator associated to a graph]\label{def:operator-graph}
    Let $G$ and $\fd$ be as in Definition \ref{def:product-graph}.
    Let $\mathcal{F}=(\mathcal{F}_{\mathrm{in}},\mathcal{F}_{\mathrm{out}})$ be a sequence of free cells of $C$. Then:
    \begin{itemize}
        \item If $\mathcal{F}\neq ((),())$ the operator associated to $(G,\fd,\mathcal{F},\{\mathbf{X}_c\}_{c\notin \mathcal{F}})$ is defined as
    \begin{equation*}
        \mathbf{W}_G: \prod_{c\in \mathcal{F}_{\mathrm{in}}} \bR^{\fd(c)} \to \bigotimes_{c\in \mathcal{F}_{\mathrm{out}}} \bR^{\fd(c)}, \quad\mathbf{W}_G\big((\mathbf{X}_c)_{c\in \mathcal{F}_\mathrm{in}}\big)=\mathbf{Y}
    \end{equation*}
    where $\mathbf{Y} \in \bigotimes_{c\in \mathcal{F}_{\mathrm{out}}} \bR^{{\fd}(c)}$ is the unique vector satisfying
        $$\sprod{\mathbf{Y}}{\otimes_{c\in \mathcal{F}_{\mathrm{out}}}\mathbf{X}_c} = 
        \bW_{(G,\fd,\{\mathbf{X}_c\}_{c\in C})},$$
    namely the Riesz representation of $\otimes_{c\in\mathcal{F}_{\mathrm{out}}}\mathbf{X}_c\mapsto \bW_{(G,\fd,\{\mathbf{X}_c\}_{c\in C})}$.
    \item Otherwise, $G$ is a product graph and $\mathbf{W}_G$ is defined as in Definition \ref{def:product}.
    \end{itemize}

\end{definition}
 
Note that if $\mathcal{F}_{\mathrm{in}}$ is empty but $\mathcal{F}_{\mathrm{out}}$ isn't, then $\mathbf{W}_G$ is a constant map from $\{0\}$ to $\bigotimes_{c\in\mathcal{F}_{\mathrm{out}}}\mathbb{R}^{\mathfrak{d}(c)}$ which we then identify with its value (see the top right example in Figure \ref{ex:first_example}).

Just as we did for product graphs, we will abuse notation and just use $G$ to denote an  operator graph when $\mathcal{F}, \{\mathbf{X}_c\}_{c\notin \mathcal{F}}$ and $\fd$ are clear from the context. As a result, $\mathbf{W}_G$ can denote both an operator (when $G$ is an operator graph) and a scalar (when $G$ is a product graph).

\begin{figure}[t]
\centering
\tikzset{every picture/.style={line width=0.75pt}} 

\begin{tikzpicture}[x=0.75pt,y=0.75pt,yscale=-1,xscale=1]

\draw    (94.64,435.61) -- (94.64,402.34) ;
\draw [shift={(94.64,416.57)}, rotate = 90] [color={rgb, 255:red, 0; green, 0; blue, 0 }  ][line width=0.75]    (4.37,-1.32) .. controls (2.78,-0.56) and (1.32,-0.12) .. (0,0) .. controls (1.32,0.12) and (2.78,0.56) .. (4.37,1.32)   ;
\draw  [color={rgb, 255:red, 0; green, 0; blue, 0 }  ,draw opacity=1 ][fill={rgb, 255:red, 0; green, 0; blue, 0 }  ,fill opacity=1 ] (97.62,405.52) .. controls (99.37,403.88) and (99.47,401.12) .. (97.83,399.37) .. controls (96.19,397.61) and (93.43,397.51) .. (91.67,399.15) .. controls (89.92,400.8) and (89.82,403.55) .. (91.46,405.31) .. controls (93.1,407.07) and (95.86,407.16) .. (97.62,405.52) -- cycle ;
\draw  [color={rgb, 255:red, 0; green, 0; blue, 0 }  ,draw opacity=1 ][fill={rgb, 255:red, 0; green, 0; blue, 0 }  ,fill opacity=1 ] (97.62,438.79) .. controls (99.37,437.15) and (99.47,434.4) .. (97.83,432.64) .. controls (96.19,430.88) and (93.43,430.78) .. (91.67,432.43) .. controls (89.92,434.07) and (89.82,436.82) .. (91.46,438.58) .. controls (93.1,440.34) and (95.86,440.43) .. (97.62,438.79) -- cycle ;
\draw    (239.64,434.61) -- (239.64,401.34) ;
\draw [shift={(239.64,415.57)}, rotate = 90] [color={rgb, 255:red, 0; green, 0; blue, 0 }  ][line width=0.75]    (4.37,-1.32) .. controls (2.78,-0.56) and (1.32,-0.12) .. (0,0) .. controls (1.32,0.12) and (2.78,0.56) .. (4.37,1.32)   ;
\draw  [color={rgb, 255:red, 0; green, 0; blue, 0 }  ,draw opacity=1 ][fill={rgb, 255:red, 0; green, 0; blue, 0 }  ,fill opacity=1 ] (242.62,404.52) .. controls (244.37,402.88) and (244.47,400.12) .. (242.83,398.37) .. controls (241.19,396.61) and (238.43,396.51) .. (236.67,398.15) .. controls (234.92,399.8) and (234.82,402.55) .. (236.46,404.31) .. controls (238.1,406.07) and (240.86,406.16) .. (242.62,404.52) -- cycle ;
\draw  [color={rgb, 255:red, 74; green, 144; blue, 226 }  ,draw opacity=1 ][fill={rgb, 255:red, 255; green, 255; blue, 255 }  ,fill opacity=1 ][line width=1.5]  (242.62,437.79) .. controls (244.37,436.15) and (244.47,433.4) .. (242.83,431.64) .. controls (241.19,429.88) and (238.43,429.78) .. (236.67,431.43) .. controls (234.92,433.07) and (234.82,435.82) .. (236.46,437.58) .. controls (238.1,439.34) and (240.86,439.43) .. (242.62,437.79) -- cycle ;
\draw    (93.64,498.61) -- (93.64,465.34) ;
\draw [shift={(93.64,479.57)}, rotate = 90] [color={rgb, 255:red, 0; green, 0; blue, 0 }  ][line width=0.75]    (4.37,-1.32) .. controls (2.78,-0.56) and (1.32,-0.12) .. (0,0) .. controls (1.32,0.12) and (2.78,0.56) .. (4.37,1.32)   ;
\draw  [color={rgb, 255:red, 0; green, 0; blue, 0 }  ,draw opacity=1 ][fill={rgb, 255:red, 0; green, 0; blue, 0 }  ,fill opacity=1 ] (96.62,468.52) .. controls (98.37,466.88) and (98.47,464.12) .. (96.83,462.37) .. controls (95.19,460.61) and (92.43,460.51) .. (90.67,462.15) .. controls (88.92,463.8) and (88.82,466.55) .. (90.46,468.31) .. controls (92.1,470.07) and (94.86,470.16) .. (96.62,468.52) -- cycle ;
\draw    (239.64,497.61) -- (239.64,464.34) ;
\draw [shift={(239.64,478.57)}, rotate = 90] [color={rgb, 255:red, 0; green, 0; blue, 0 }  ][line width=0.75]    (4.37,-1.32) .. controls (2.78,-0.56) and (1.32,-0.12) .. (0,0) .. controls (1.32,0.12) and (2.78,0.56) .. (4.37,1.32)   ;
\draw  [color={rgb, 255:red, 74; green, 144; blue, 226 }  ,draw opacity=1 ][fill={rgb, 255:red, 255; green, 255; blue, 255 }  ,fill opacity=1 ][line width=1.5]  (242.62,500.79) .. controls (244.37,499.15) and (244.47,496.4) .. (242.83,494.64) .. controls (241.19,492.88) and (238.43,492.78) .. (236.67,494.43) .. controls (234.92,496.07) and (234.82,498.82) .. (236.46,500.58) .. controls (238.1,502.34) and (240.86,502.43) .. (242.62,500.79) -- cycle ;
\draw  [color={rgb, 255:red, 0; green, 0; blue, 0 }  ,draw opacity=1 ][fill={rgb, 255:red, 255; green, 255; blue, 255 }  ,fill opacity=1 ][line width=1.5]  (96.62,501.79) .. controls (98.37,500.15) and (98.47,497.4) .. (96.83,495.64) .. controls (95.19,493.88) and (92.43,493.78) .. (90.67,495.43) .. controls (88.92,497.07) and (88.82,499.82) .. (90.46,501.58) .. controls (92.1,503.34) and (94.86,503.43) .. (96.62,501.79) -- cycle ;
\draw  [color={rgb, 255:red, 0; green, 0; blue, 0 }  ,draw opacity=1 ][fill={rgb, 255:red, 255; green, 255; blue, 255 }  ,fill opacity=1 ][line width=1.5]  (242.62,467.52) .. controls (244.37,465.88) and (244.47,463.12) .. (242.83,461.37) .. controls (241.19,459.61) and (238.43,459.51) .. (236.67,461.15) .. controls (234.92,462.8) and (234.82,465.55) .. (236.46,467.31) .. controls (238.1,469.07) and (240.86,469.16) .. (242.62,467.52) -- cycle ;

\draw (79.33,412.06) node [anchor=north west][inner sep=0.75pt]  [font=\tiny,rotate=-0.89]  {$M$};
\draw (90.22,389.89) node [anchor=north west][inner sep=0.75pt]  [font=\tiny,rotate=-0.89]  {$\mathbf{x}$};
\draw (103,407) node [anchor=north west][inner sep=0.75pt]  [font=\footnotesize]  {$\mathbf{W}_{G} =\mathbf{y}^{\top } M\mathbf{x} \in \mathbb{R}$};
\draw (89.22,440.89) node [anchor=north west][inner sep=0.75pt]  [font=\tiny,rotate=-0.89]  {$\mathbf{y}$};
\draw (224.33,411.06) node [anchor=north west][inner sep=0.75pt]  [font=\tiny,rotate=-0.89]  {$M$};
\draw (235.22,388.89) node [anchor=north west][inner sep=0.75pt]  [font=\tiny,rotate=-0.89]  {$\mathbf{x}$};
\draw (252,407) node [anchor=north west][inner sep=0.75pt]  [font=\footnotesize]  {$\mathbf{W}_{G} =M\mathbf{x} \in \mathbb{R}^{n}$};
\draw (78.33,475.06) node [anchor=north west][inner sep=0.75pt]  [font=\tiny,rotate=-0.89]  {$M$};
\draw (90.22,452.89) node [anchor=north west][inner sep=0.75pt]  [font=\tiny,rotate=-0.89]  {$\mathbf{x}$};
\draw (224.33,474.06) node [anchor=north west][inner sep=0.75pt]  [font=\tiny,rotate=-0.89]  {$M$};
\draw (252,464) node [anchor=north west][inner sep=0.75pt]  [font=\footnotesize]  {$\mathbf{W}_{G} :\mathbb{R}^{n}\rightarrow \mathbb{R}^{n}$};
\draw (106,462) node [anchor=north west][inner sep=0.75pt]  [font=\footnotesize]  {$\mathbf{W}_{G} :\mathbb{R}^{n}\rightarrow \mathbb{R}$};
\draw (110,475) node [anchor=north west][inner sep=0.75pt]  [font=\footnotesize]  {$=\left(\mathbf{y} \mapsto \mathbf{y}^{\top } M\mathbf{x}\right)$};
\draw (257,476) node [anchor=north west][inner sep=0.75pt]  [font=\footnotesize]  {$=(\mathbf{x} \mapsto M\mathbf{x})$};

\end{tikzpicture}
\caption{
A simple product graph $G$ (top left), from which we can free vertices to make $\mathbf{W}_G$ either a vector, a scalar-valued map, or a vector-valued map (top right, bottom left, and bottom right respectively). }
\label{ex:first_example}
\end{figure}
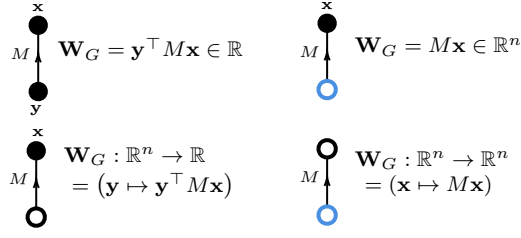
\bigskip

Some examples to illustrate these definitions might be in order, and are given in Figures \ref{ex:first_example} in \ref{fig:basic-tree-free}. When depicting a graph, a free vertex is drawn as a circle (e.g. $\circ$) and is depicted in blue if it is in the set of outputs $\mathcal{F}_{\mathrm{out}}$, while a vertex fixed to $\mathbf{x}$ is drawn as a black dot labeled by said vector (e.g. $\bullet_{\mathbf{x}}$). 
Similarly, we draw free edges with dotted directed lines (e.g. $\dashrightarrow$) and in blue if they are in the set $\mathcal{F}_{\mathrm{out}}$ of output cells, while fixed edges are drawn as solid lines labeled by their input matrix (e.g. $\xrightarrow[]{W}$). 
For simplicity, we will omit labels from vertices (resp. edges) whose input is $\mathbf{1}$ (resp. $\mathbf{I}$) when drawing them. We will also not indicate dimensions given by $\mathfrak{d}$ nor the order of the sequence $\mathcal{F}$ on the graph itself.

\begin{example}
Let $G$ be the graph in Figure \ref{fig:basic-tree-free} (left). To turn this into an operator graph, let $a,b,c,d$ refer to its vertices from bottom to top/ left to right and let $\fd:\{a,b,c,d\}\to\mathbb{N}$ be any fixed dimension map. Then if we consider $\cF_{\mathrm{in}} = (d,(a,b))$, $\cF_{\mathrm{out}} = (a)$ and fix the remaining cells as 
\[
b \mapsto \mathbf{x}_b, ~
c \mapsto \mathbf{1}, ~ 
(b,d) \mapsto \mathbf{I}, ~
(b,c) \mapsto W_1
\]
(c.f. Figure \ref{fig:basic-tree-free}), then $\mathbf{W}_G$ is the map
$(\mathbf{x}_d, W_{(a,b)}) \mapsto 
W_{(a,b)}
\big(\mathbf{x}_{b} \odot (W_1\mathbf{1})\odot(\mathbf{I} ~\mathbf{x}_d)\big)$.

    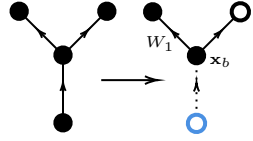
\begin{figure}[ht]
    \centering
    
    \tikzset{every picture/.style={line width=0.75pt}} 

\begin{tikzpicture}[x=0.75pt,y=0.75pt,yscale=-1,xscale=1]

\draw  (448.81,6367.37) -- (426.81,6345.37) ;
\draw [shift={(436.12,6354.68)}, rotate = 45] [color={rgb, 255:red, 0; green, 0; blue, 0 }  ][line width=0.75]    (4.37,-1.32) .. controls (2.78,-0.56) and (1.32,-0.12) .. (0,0) .. controls (1.32,0.12) and (2.78,0.56) .. (4.37,1.32)   ;
\draw  [color={rgb, 255:red, 0; green, 0; blue, 0 }  ,draw opacity=1 ][fill={rgb, 255:red, 0; green, 0; blue, 0 }  ,fill opacity=1 ] (429.78,6348.56) .. controls (431.54,6346.92) and (431.64,6344.16) .. (430,6342.4) .. controls (428.35,6340.64) and (425.6,6340.55) .. (423.84,6342.19) .. controls (422.08,6343.83) and (421.99,6346.59) .. (423.63,6348.35) .. controls (425.27,6350.1) and (428.03,6350.2) .. (429.78,6348.56) -- cycle ;
\draw    (448.81,6367.37) -- (470.81,6345.37) ;
\draw [shift={(461.51,6354.68)}, rotate = 135] [color={rgb, 255:red, 0; green, 0; blue, 0 }  ][line width=0.75]    (4.37,-1.32) .. controls (2.78,-0.56) and (1.32,-0.12) .. (0,0) .. controls (1.32,0.12) and (2.78,0.56) .. (4.37,1.32)   ;
\draw  [color={rgb, 255:red, 0; green, 0; blue, 0 }  ,draw opacity=1 ][fill={rgb, 255:red, 0; green, 0; blue, 0 }  ,fill opacity=1 ] (451.78,6404.54) .. controls (453.54,6402.9) and (453.64,6400.15) .. (452,6398.39) .. controls (450.35,6396.63) and (447.6,6396.53) .. (445.84,6398.18) .. controls (444.08,6399.82) and (443.99,6402.57) .. (445.63,6404.33) .. controls (447.27,6406.09) and (450.03,6406.18) .. (451.78,6404.54) -- cycle ;
\draw  [color={rgb, 255:red, 0; green, 0; blue, 0 }  ,draw opacity=1 ][fill={rgb, 255:red, 0; green, 0; blue, 0 }  ,fill opacity=1 ] (451.78,6370.56) .. controls (453.54,6368.92) and (453.64,6366.16) .. (452,6364.4) .. controls (450.35,6362.64) and (447.6,6362.55) .. (445.84,6364.19) .. controls (444.08,6365.83) and (443.99,6368.59) .. (445.63,6370.35) .. controls (447.27,6372.1) and (450.03,6372.2) .. (451.78,6370.56) -- cycle ;
\draw  [color={rgb, 255:red, 0; green, 0; blue, 0 }  ,draw opacity=1 ][fill={rgb, 255:red, 0; green, 0; blue, 0 }  ,fill opacity=1 ] (473.78,6348.56) .. controls (475.54,6346.92) and (475.64,6344.16) .. (474,6342.4) .. controls (472.35,6340.64) and (469.6,6340.55) .. (467.84,6342.19) .. controls (466.08,6343.83) and (465.99,6346.59) .. (467.63,6348.35) .. controls (469.27,6350.1) and (472.03,6350.2) .. (473.78,6348.56) -- cycle ;
\draw [color={rgb, 255:red, 0; green, 0; blue, 0 }  ,draw opacity=1 ]   (448.81,6367.37) -- (448.81,6401.36) ;
\draw [shift={(448.81,6380.97)}, rotate = 90] [color={rgb, 255:red, 0; green, 0; blue, 0 }  ,draw opacity=1 ][line width=0.75]    (4.37,-1.32) .. controls (2.78,-0.56) and (1.32,-0.12) .. (0,0) .. controls (1.32,0.12) and (2.78,0.56) .. (4.37,1.32)   ;
\draw    (515.81,6367.66) -- (493.81,6345.66) ;
\draw [shift={(503.12,6354.96)}, rotate = 45] [color={rgb, 255:red, 0; green, 0; blue, 0 }  ][line width=0.75]    (4.37,-1.32) .. controls (2.78,-0.56) and (1.32,-0.12) .. (0,0) .. controls (1.32,0.12) and (2.78,0.56) .. (4.37,1.32)   ;
\draw  [color={rgb, 255:red, 0; green, 0; blue, 0 }  ,draw opacity=1 ][fill={rgb, 255:red, 0; green, 0; blue, 0 }  ,fill opacity=1 ] (496.78,6348.84) .. controls (498.54,6347.2) and (498.64,6344.45) .. (497,6342.69) .. controls (495.35,6340.93) and (492.6,6340.84) .. (490.84,6342.48) .. controls (489.08,6344.12) and (488.99,6346.87) .. (490.63,6348.63) .. controls (492.27,6350.39) and (495.03,6350.48) .. (496.78,6348.84) -- cycle ;
\draw    (515.81,6367.66) -- (537.81,6345.66) ;
\draw [shift={(528.51,6354.96)}, rotate = 135] [color={rgb, 255:red, 0; green, 0; blue, 0 }  ][line width=0.75]    (4.37,-1.32) .. controls (2.78,-0.56) and (1.32,-0.12) .. (0,0) .. controls (1.32,0.12) and (2.78,0.56) .. (4.37,1.32)   ;
\draw  [color={rgb, 255:red, 0; green, 0; blue, 0 }  ,draw opacity=1 ][fill={rgb, 255:red, 255; green, 255; blue, 255 }  ,fill opacity=1 ][line width=1.5]  (540.78,6348.84) .. controls (542.54,6347.2) and (542.64,6344.45) .. (541,6342.69) .. controls (539.35,6340.93) and (536.6,6340.84) .. (534.84,6342.48) .. controls (533.08,6344.12) and (532.99,6346.87) .. (534.63,6348.63) .. controls (536.27,6350.39) and (539.03,6350.48) .. (540.78,6348.84) -- cycle ;
\draw [dotted] [color={rgb, 255:red, 0; green, 0; blue, 0 }  ,draw opacity=1 ]   (515.81,6367.66) -- (515.81,6401.64) ;
\draw [shift={(515.81,6381.25)}, rotate = 90] [color={rgb, 255:red, 0; green, 0; blue, 0 }  ,draw opacity=1 ][line width=0.75]    (4.37,-1.32) .. controls (2.78,-0.56) and (1.32,-0.12) .. (0,0) .. controls (1.32,0.12) and (2.78,0.56) .. (4.37,1.32)   ;
\draw  [color={rgb, 255:red, 0; green, 0; blue, 0 }  ,draw opacity=1 ][fill={rgb, 255:red, 0; green, 0; blue, 0 }  ,fill opacity=1 ] (518.78,6370.84) .. controls (520.54,6369.2) and (520.64,6366.45) .. (519,6364.69) .. controls (517.35,6362.93) and (514.6,6362.84) .. (512.84,6364.48) .. controls (511.08,6366.12) and (510.99,6368.87) .. (512.63,6370.63) .. controls (514.27,6372.39) and (517.03,6372.48) .. (518.78,6370.84) -- cycle ;
\draw  [color={rgb, 255:red, 74; green, 144; blue, 226 }  ,draw opacity=1 ][fill={rgb, 255:red, 255; green, 255; blue, 255 }  ,fill opacity=1 ][line width=1.5]  (518.78,6404.83) .. controls (520.54,6403.19) and (520.64,6400.43) .. (519,6398.67) .. controls (517.35,6396.92) and (514.6,6396.82) .. (512.84,6398.46) .. controls (511.08,6400.1) and (510.99,6402.86) .. (512.63,6404.62) .. controls (514.27,6406.37) and (517.03,6406.47) .. (518.78,6404.83) -- cycle ;
\draw   (468,6379.94) -- (494.5,6379.94) ;
\draw [shift={(496.5,6379.94)}, rotate = 180] [color={rgb, 255:red, 0; green, 0; blue, 0 }  ][line width=0.75]    (6.56,-1.97) .. controls (4.17,-0.84) and (1.99,-0.18) .. (0,0) .. controls (1.99,0.18) and (4.17,0.84) .. (6.56,1.97)   ;

\draw (488.81,6356.66) node [anchor=north west][inner sep=0.75pt]  [font=\tiny]  {$W_{1}$};
\draw (521,6366.69) node [anchor=north west][inner sep=0.75pt]  [font=\tiny]  {$\mathbf{x}_{b}$};

\end{tikzpicture}
    \caption{Example of an operator graph (right) obtained from a graph (left) by a choice of $\mathcal{F}_{\mathrm{in}}$, $\mathcal{F}_{\mathrm{out}}$, $\fd$ and $(\mathbf{X}_c)_{c \notin \cF}$. Note how on the right, we omit labels from vertices and edges whose inputs are $\mathbf{1}$ or $\mathbf{I}$ respectively.}
    \label{fig:basic-tree-free}
\end{figure}
\end{example}

\subsection{Operations on graphs.}\label{dictionary} Having outlined the correspondence between graphs and operators, we can now go further and show how various analytic operations can be expressed as binary operations on graphs.

\begin{proposition}[Evaluation] \label{def:fixing}
Let $(G,\fd, (\mathcal{F}_{\mathrm{in}},\mathcal{F}_{\mathrm{out}}), \{\mathbf{X}_c\}_{c \notin \cF})$ be an operator graph, $\mathbf{W}_G$ its associated map and $c_0 \in \cF_{\mathrm{in}}$ be a free cell of $G$. Let $G[\mathbf{X}_{c_0}]$ denote the graph obtained by fixing $c_0$'s input to some $\mathbf{X}_{c_0} \in \bR^{\fd(c_0)}$. Then if $|\mathcal{F}_{\mathrm{in}}|=1$, we have $\mathbf{W}_{G[\mathbf{X}_{c_0}]}=\mathbf{W}_G(\mathbf{X}_{c_0})$. Otherwise,
\[
        \mathbf{W}_{G[\mathbf{X}_{c_0}]}: \prod_{c\in \mathcal{F}_{\mathrm{in}}\setminus c_0} \bR^{\fd(c)} \to \bigotimes_{c\in \mathcal{F}_{\mathrm{out}}} \bR^{\fd(c)}
\]
\[
    \mathbf{W}_{G[\mathbf{X}_{c_0}]}\big((\mathbf{X}_c)_{c\in\mathcal{F}_\mathrm{in}\setminus c_0}\big)=
    \mathbf{W}_{G}\big((\mathbf{X}_c)_{c\in\mathcal{F}_\mathrm{in}}\big).
\]
In particular, $$\mathbf{W}_{G[(\mathbf{X}_c)_{c\in \mathcal{F}_{\mathrm{in}}}]}=\mathbf{W}_G\big((\mathbf{X}_c)_{c\in \mathcal{F}_{\mathrm{in}}}\big). $$
\end{proposition}
\begin{proof}
    This is an immediate consequence of Definition \ref{def:operator-graph}.
\end{proof}

For example, let $G$ be the operator graph on the left in Figure \ref{fig:fixing} below and denote its top-right vertex by $v$. Note that $\mathcal{F}_{\mathrm{in}}=(v)$ and $\mathcal{F}_{\mathrm{out}}=()$. Ignoring dimensions for the sake of this example, $$ \mathbf{W}_G:\mathbf{x}  \mapsto W_{2}\big(( W_{1}\mathbf{1}) \odot ( W_{1}\mathbf{x})\big),$$ 
and fixing the unique in-vertex's input to $\mathbf{X}_{v}$ (c.f. Figure \ref{fig:fixing}, right) gives $W_{2}( (W_{1}\mathbf{1}) \odot (W_{1}\mathbf{X}_v))$.
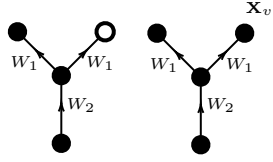
\begin{figure}[ht]
    \centering
    \tikzset{every picture/.style={line width=0.75pt}} 

\begin{tikzpicture}[x=0.75pt,y=0.75pt,yscale=-1,xscale=1]
\draw    (448.81,6287.37) -- (426.81,6265.37) ;
\draw [shift={(436.12,6274.68)}, rotate = 45] [color={rgb, 255:red, 0; green, 0; blue, 0 }  ][line width=0.75]    (4.37,-1.32) .. controls (2.78,-0.56) and (1.32,-0.12) .. (0,0) .. controls (1.32,0.12) and (2.78,0.56) .. (4.37,1.32)   ;
\draw  [color={rgb, 255:red, 0; green, 0; blue, 0 }  ,draw opacity=1 ][fill={rgb, 255:red, 0; green, 0; blue, 0 }  ,fill opacity=1 ] (429.78,6268.56) .. controls (431.54,6266.92) and (431.64,6264.16) .. (430,6262.4) .. controls (428.35,6260.64) and (425.6,6260.55) .. (423.84,6262.19) .. controls (422.08,6263.83) and (421.99,6266.59) .. (423.63,6268.35) .. controls (425.27,6270.1) and (428.03,6270.2) .. (429.78,6268.56) -- cycle ;
\draw    (448.81,6287.37) -- (470.81,6265.37) ;
\draw [shift={(461.51,6274.68)}, rotate = 135] [color={rgb, 255:red, 0; green, 0; blue, 0 }  ][line width=0.75]    (4.37,-1.32) .. controls (2.78,-0.56) and (1.32,-0.12) .. (0,0) .. controls (1.32,0.12) and (2.78,0.56) .. (4.37,1.32)   ;
\draw  [color={rgb, 255:red, 0; green, 0; blue, 0 }  ,draw opacity=1 ][fill={rgb, 255:red, 0; green, 0; blue, 0 }  ,fill opacity=1 ] (451.78,6324.54) .. controls (453.54,6322.9) and (453.64,6320.15) .. (452,6318.39) .. controls (450.35,6316.63) and (447.6,6316.53) .. (445.84,6318.18) .. controls (444.08,6319.82) and (443.99,6322.57) .. (445.63,6324.33) .. controls (447.27,6326.09) and (450.03,6326.18) .. (451.78,6324.54) -- cycle ;
\draw  [color={rgb, 255:red, 0; green, 0; blue, 0 }  ,draw opacity=1 ][fill={rgb, 255:red, 0; green, 0; blue, 0 }  ,fill opacity=1 ] (451.78,6290.56) .. controls (453.54,6288.92) and (453.64,6286.16) .. (452,6284.4) .. controls (450.35,6282.64) and (447.6,6282.55) .. (445.84,6284.19) .. controls (444.08,6285.83) and (443.99,6288.59) .. (445.63,6290.35) .. controls (447.27,6292.1) and (450.03,6292.2) .. (451.78,6290.56) -- cycle ;
\draw    (378.81,6286.66) -- (356.81,6264.66) ;
\draw [shift={(366.12,6273.96)}, rotate = 45] [color={rgb, 255:red, 0; green, 0; blue, 0 }  ][line width=0.75]    (4.37,-1.32) .. controls (2.78,-0.56) and (1.32,-0.12) .. (0,0) .. controls (1.32,0.12) and (2.78,0.56) .. (4.37,1.32)   ;
\draw  [color={rgb, 255:red, 0; green, 0; blue, 0 }  ,draw opacity=1 ][fill={rgb, 255:red, 0; green, 0; blue, 0 }  ,fill opacity=1 ] (359.78,6267.84) .. controls (361.54,6266.2) and (361.64,6263.45) .. (360,6261.69) .. controls (358.35,6259.93) and (355.6,6259.84) .. (353.84,6261.48) .. controls (352.08,6263.12) and (351.99,6265.87) .. (353.63,6267.63) .. controls (355.27,6269.39) and (358.03,6269.48) .. (359.78,6267.84) -- cycle ;
\draw    (378.81,6286.66) -- (400.81,6264.66) ;
\draw [shift={(391.51,6273.96)}, rotate = 135] [color={rgb, 255:red, 0; green, 0; blue, 0 }  ][line width=0.75]    (4.37,-1.32) .. controls (2.78,-0.56) and (1.32,-0.12) .. (0,0) .. controls (1.32,0.12) and (2.78,0.56) .. (4.37,1.32)   ;
\draw  [color={rgb, 255:red, 0; green, 0; blue, 0 }  ,draw opacity=1 ][fill={rgb, 255:red, 255; green, 255; blue, 255 }  ,fill opacity=1 ][line width=1.5]  (403.78,6267.84) .. controls (405.54,6266.2) and (405.64,6263.45) .. (404,6261.69) .. controls (402.35,6259.93) and (399.6,6259.84) .. (397.84,6261.48) .. controls (396.08,6263.12) and (395.99,6265.87) .. (397.63,6267.63) .. controls (399.27,6269.39) and (402.03,6269.48) .. (403.78,6267.84) -- cycle ;
\draw [color={rgb, 255:red, 0; green, 0; blue, 0 }  ,draw opacity=1 ]   (378.81,6286.66) -- (378.81,6320.64) ;
\draw [shift={(378.81,6300.25)}, rotate = 90] [color={rgb, 255:red, 0; green, 0; blue, 0 }  ,draw opacity=1 ][line width=0.75]    (4.37,-1.32) .. controls (2.78,-0.56) and (1.32,-0.12) .. (0,0) .. controls (1.32,0.12) and (2.78,0.56) .. (4.37,1.32)   ;
\draw  [color={rgb, 255:red, 0; green, 0; blue, 0 }  ,draw opacity=1 ][fill={rgb, 255:red, 0; green, 0; blue, 0 }  ,fill opacity=1 ] (381.78,6323.83) .. controls (383.54,6322.19) and (383.64,6319.43) .. (382,6317.67) .. controls (380.35,6315.92) and (377.6,6315.82) .. (375.84,6317.46) .. controls (374.08,6319.1) and (373.99,6321.86) .. (375.63,6323.62) .. controls (377.27,6325.37) and (380.03,6325.47) .. (381.78,6323.83) -- cycle ;
\draw  [color={rgb, 255:red, 0; green, 0; blue, 0 }  ,draw opacity=1 ][fill={rgb, 255:red, 0; green, 0; blue, 0 }  ,fill opacity=1 ] (381.78,6289.84) .. controls (383.54,6288.2) and (383.64,6285.45) .. (382,6283.69) .. controls (380.35,6281.93) and (377.6,6281.84) .. (375.84,6283.48) .. controls (374.08,6285.12) and (373.99,6287.87) .. (375.63,6289.63) .. controls (377.27,6291.39) and (380.03,6291.48) .. (381.78,6289.84) -- cycle ;
\draw  [color={rgb, 255:red, 0; green, 0; blue, 0 }  ,draw opacity=1 ][fill={rgb, 255:red, 0; green, 0; blue, 0 }  ,fill opacity=1 ] (473.78,6268.56) .. controls (475.54,6266.92) and (475.64,6264.16) .. (474,6262.4) .. controls (472.35,6260.64) and (469.6,6260.55) .. (467.84,6262.19) .. controls (466.08,6263.83) and (465.99,6266.59) .. (467.63,6268.35) .. controls (469.27,6270.1) and (472.03,6270.2) .. (473.78,6268.56) -- cycle ;
\draw [color={rgb, 255:red, 0; green, 0; blue, 0 }  ,draw opacity=1 ]   (448.81,6287.37) -- (448.81,6321.36) ;
\draw [shift={(448.81,6300.97)}, rotate = 90] [color={rgb, 255:red, 0; green, 0; blue, 0 }  ,draw opacity=1 ][line width=0.75]    (4.37,-1.32) .. controls (2.78,-0.56) and (1.32,-0.12) .. (0,0) .. controls (1.32,0.12) and (2.78,0.56) .. (4.37,1.32)   ;

\draw (469.85,6248.24) node [anchor=north west][inner sep=0.75pt]  [font=\tiny]  {$\mathbf{X}_{v}$};
\draw (459.81,6276.37) node [anchor=north west][inner sep=0.75pt]  [font=\tiny]  {$W_{1}$};
\draw (421.81,6276.37) node [anchor=north west][inner sep=0.75pt]  [font=\tiny]  {$W_{1}$};
\draw (389.81,6275.66) node [anchor=north west][inner sep=0.75pt]  [font=\tiny]  {$W_{1}$};
\draw (351.81,6275.66) node [anchor=north west][inner sep=0.75pt]  [font=\tiny]  {$W_{1}$};
\draw (380.14,6296.39) node [anchor=north west][inner sep=0.75pt]  [font=\tiny]  {$W_{2}$};
\draw (450.14,6296.39) node [anchor=north west][inner sep=0.75pt]  [font=\tiny]  {$W_{2}$};

\end{tikzpicture}
    \vspace{-25px}
    \caption{Fixing a free in-vertex (middle) or in-edge (right) of a graph.}
    \label{fig:fixing}
\end{figure} 

Recall that by definition, an operator graph $(G,\mathfrak{d},((),\mathcal{F}_{\mathrm{out}}),\{\mathbf{X}_{c}\}_{c\notin\mathcal{F}})$ gives rise to $\mathbf{W}_G\in\bigotimes_{c\in\mathcal{F}_{\mathrm{out}}}\mathbb{R}^{\mathfrak{d}(c)}$ (see Figure \ref{ex:first_example}). For two such graphs $G_1$ and $G_2$, it therefore makes sense to talk about the Hadamard (entrywise) product of $\mathbf{W}_{G_1}$ and $\mathbf{W}_{G_2}$.

\begin{proposition}[Hadamard product] 
\label{def:graph_wedge}
Let $$G_1=(G_1,\fd_1, ((), (v_1)), \{\mathbf{X}_c\}_{c\in C_1\setminus \{v_1\}}),\, G_2=(G_2,\fd_2, ((), (v_2)), \{\mathbf{X}_c\}_{c\in C_2\setminus \{v_2\}})$$ be operator graphs with unique free out-vertices $v_1\in V(G_1)$ and $v_2\in V(G_2)$ respectively. Assume that $\fd_1(v_1) = \fd_2(v_2)=d>0$. 

Define $G_1 \wedge G_2$ to be the operator graph obtained by identifying $v_1$ and $v_2$ (meaning that we take the union of the two graphs and consider $v_1=v_2$ to be the same vertex). Then\[
    \mathbf{W}_{G_1\land G_2}=\mathbf{W}_{G_1}\odot\mathbf{W}_{G_2}\in\mathbb{R}^{\mathfrak{d}_1(v_1)},
\]
where $\odot$ denotes the Hadamard product.
\end{proposition}
\begin{proof}
Let $v_1 \wedge v_2$ be the unified vertex in $G_1 \wedge G_2$, and let $V_i^-=V(G_i)\setminus \{v_i\}$ for $i\in\{1,2\}$. Then for any $ \mathbf{X} \in \bR^{\fd_1(v_1)}$, Proposition \ref{def:fixing} yields $\sprod{\mathbf{W}_{G_1\land G_2}}{\mathbf{X}} = \mathbf{W}_{(G_1\land G_2)[\mathbf{X}]}$, which by definition equals
\begin{align*}
&\sum_{\mathbf{i}_{V(G_1)},\mathbf{i}_{V(G_2)}} 
\delta_{i_{v_1}i_{v_2}}
\bigg(
\prod_{c_1 \in C_1 \setminus\{v_1\}}~[\mathbf{X}_{c_1}]_{i_{c_1}} 
\bigg)
\bigg(
\prod_{c_2 \in C_2 \setminus\{v_2\}}~[\mathbf{X}_{c_2}]_{i_{c_2}} 
\bigg)
[\mathbf{X}]_{i_{v_1}}
\\
&=\sum_{i_{v_1},i_{v_2}=1}^{d}
\bigg(
\sum_{\mathbf{i}_{V_1^-}} 
\prod_{c_1 \in C_1 \setminus\{v_1\}}~[\mathbf{X}_{c_1}]_{i_{c_1}} 
\bigg)
\bigg(
\sum_{\mathbf{i}_{V_2^-}} 
\prod_{c_2 \in C_2 \setminus\{v_2\}}~[\mathbf{X}_{c_2}]_{i_{c_2}} 
\bigg)
[\mathbf{X}]_{i_{v_1}}
\\
&=\sum_{k=1}^{d}
[\mathbf{W}_{G_1}]_{k}
[\mathbf{W}_{G_2}]_{k}
[\mathbf{X}]_{k}
= \sum_{k=1}^{d}
[\mathbf{W}_{G_1} \odot \mathbf{W}_{G_2}]_{k}
[\mathbf{X}]_{k}
\\
&= \sprod{\mathbf{W}_{G_1} \odot \mathbf{W}_{G_2}}{\mathbf{X}},
\end{align*}
where the first sum is over indexations $\mathbf{i}_{V(G_j)}\in \prod_{v\in V(G_j)} [\mathfrak{d}(v)]$, $j\in \{1,2\}$ and $\delta_{ij}$ is the Kronecker delta.
The proposition then follows by uniqueness of the Riesz representation.
\end{proof}
\begin{figure}[h]
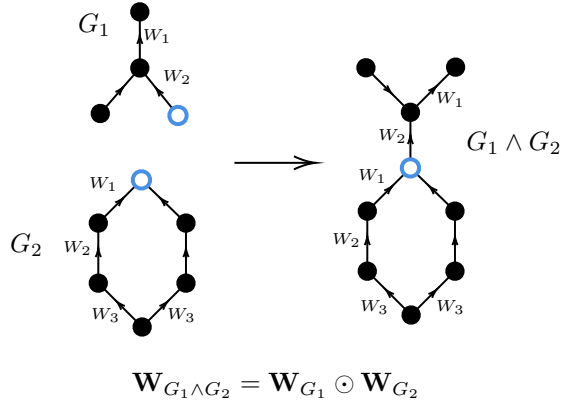

    \centering
    \include{Figures/identifying}
    \vspace{-20pt}
    \caption{Identifying two graphs by their unique free (out-)vertex gives rise to a Hadamard product.}
    \label{fig:identifying}
\end{figure}
Note that the graph operation $\wedge$ is associative and symmetric much like its analytic counterpart $\odot$, meaning that there is no confusion in writing $G_1 \wedge \cdots \wedge G_n = \bigwedge_{i \in [n]} G_i$.    

\begin{proposition}[Transpose]\label{def:swap_vertices}
Let $G=(G,\fd, (\mathcal{F}_{\mathrm{in}},\mathcal{F}_{\mathrm{out}}, \{\mathbf{X}_c\}_{c \notin \mathcal{F}})$ be such that $\mathcal{F}_{\mathrm{in}}$ and $\mathcal{F}_{\mathrm{out}}$ have at most one element, but are not both empty. Define
\[
    G^\top:=(G,\fd, (\mathcal{F}_{\mathrm{out}},\mathcal{F}_{\mathrm{in}}, \{\mathbf{X}_c\}_{c \notin \mathcal{F}}).
\]
Then $\mathbf{W}_{G^\top}$ is the adjoint of $\mathbf{W}_G$. 
\end{proposition}
\begin{proof}
Assume that $\mathcal{F}_{\mathrm{in}}=(c_1)$ and $\mathcal{F}_{\mathrm{out}}=(c_2)$. Then by definition, for any fixed $\mathbf{X}_{c_1},\mathbf{X}_{c_2}$,
\begin{align*}
\Big\langle{\mathbf{W}_G\big(\mathbf{X}_{c_1}\big),}{\mathbf{X}_{c_2}}\Big\rangle &=\bW_{(G,\fd,(\mathbf{X}_c)_{c\in C})}= \Big\langle{\mathbf{X}_{c_1}},{\mathbf{W}_{G^\top }\big(\mathbf{X}_{c_2}\big)\Big\rangle}.
\end{align*}
The proof is essentially identical if one of $\mathcal{F}_{\mathrm{in}}$ or $\mathcal{F}_{\mathrm{out}}$ is $(())$.
\end{proof}

\begin{proposition}[Composition] \label{prop:composition}
Let $G_1=(G_1,\fd_1, \mathcal{F}_1=((a_1),(b_1)), \{\mathbf{X}_c\}_{c\notin \mathcal{F}_1})$ and $G_2=(G_2,\fd_2, \mathcal{F}_2=((b_2),(c_2)), \{\mathbf{X}_c\}_{c\notin \mathcal{F}_2})$ be such that $\fd_1(b_1)=\fd_2(b_2)$. 
Define $G_2 \circ G_1$ to be the graph obtained by identifying $b_1$ and $b_2$ and fixing the resulting cell's input to $\mathbf{1}$ if the $b_i$ are vertices, and $\mathbf{I}$ if they are edges.
In the event that vertices $v_1, v_2$ get identified as a result (e.g. when $b_1,b_2$ are edges), the resulting vertex is given input 
 $\mathbf{X}_{v_1}\odot \mathbf{X}_{v_2}$.
Then we have
\[  
    \mathbf{W}_{G_2 \circ G_1} = \mathbf{W}_{G_2} \circ \mathbf{W}_{G_1}
\]
where on the right hand side, $\circ$ denotes composition.
\end{proposition}
\begin{proof}
Assume that $b_1,b_2$ are vertices (the proof is essentially identical if they are edges) for which $\mathfrak{d}(b_1)=\mathfrak{d}(b_2)=d>0$. Fix a pair of inputs $\mathbf{X}_{a_1},\mathbf{X}_{d_2}$ for $a_1$ and $c_2$. Then on the one hand, Proposition \ref{def:fixing} yields
\begin{align}
    \Big\langle \mathbf{W}_{G_2\circ G_1}\big(\mathbf{X}_{a_1}\big),\mathbf{X}_{c_2}\Big\rangle \nonumber &=\Big\langle \mathbf{W}_{G_2\circ G_1[\mathbf{X}_{a_1}]},\mathbf{X}_{c_2}\Big\rangle,
\end{align}
which is by definition equal to the value of the product graph obtained from $G_{2}\circ G_1$ by fixing $a_1$ and $c_2$ to have input $\mathbf{X}_{a_1}$ and $\mathbf{X}_{c_2}$ respectively. Letting  $V_i^-=V(G_i)\setminus \{b_i\}$ and $C_i^-=C_i\setminus \{b_i\}$ for simplicity, this equals
\begin{align}
    &\sum_{\mathbf{i}_{V(G_1)},\mathbf{i}_{V(G_2)}}\delta_{i_{b_1}i_{b_2}}\prod_{s_1\in C_1^-}[\mathbf{X}_{s_1}]_{i_{s_1}}\prod_{s_1\in C_2^-}[\mathbf{X}_{s_2}]_{i_{s_2}}\nonumber\\
    &=\sum_{i_{b_1},i_{b_2}=1}^d\bigg(\sum_{\mathbf{i}_{V_1^-}}\prod_{s_1\in C_1^{-}}[\mathbf{X}_{s_1}]_{i_{s_1}}\bigg)\bigg(\sum_{\mathbf{i}_{V_2^-}}\prod_{s_2\in C_2^{-}}[\mathbf{X}_{s_2}]_{i_{s_2}}\bigg).\label{eq:composition}
\end{align}

On the other hand, Propositions \ref{def:swap_vertices} and \ref{def:graph_wedge} yield
\begin{align*}
    \Big\langle\mathbf{W}_{G_2}\big(\mathbf{W}_{G_1}(\mathbf{X}_{a_1})\big),\mathbf{X}_{c_2}\Big\rangle&=\Big\langle \mathbf{W}_{G_1}\big(\mathbf{X}_{a_1}\big),\mathbf{W}_{G_2^\top}\big(\mathbf{X}_{c_2}\big)\Big\rangle \\
    &=\Big\langle \mathbf{W}_{G_1}\big(\mathbf{X}_{a_1}\big)\odot\mathbf{W}_{G_2^\top}\big(\mathbf{X}_{c_2}\big), \mathbf{1}\Big\rangle_{\mathbb{R}^d}
\end{align*}
which is equal to the right-hand side of \eqref{eq:composition}.
\end{proof}
\begin{figure}[t]
        \centering
        \input{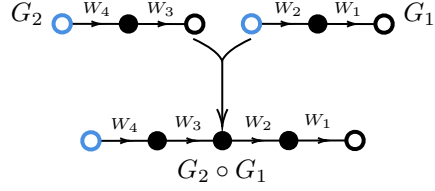}
        \caption{Two graphs $G_1, G_2$ for which $\mathbf{W}_{G_1}(\mathbf{x})=W_2W_1\mathbf{x}$ and $\mathbf{W}_{G_2}(\mathbf{x})=W_4W_3\mathbf{x}$, some matrices $W_i$. Composing them yields the map $\mathbf{W}_{G_2\circ G_1}(\mathbf{x})=W_4W_3W_2W_1\mathbf{x}$.}
        \label{fig:composition_example_specific}
\end{figure}

 This operation is depicted in Figure \ref{fig:composition_example_specific} when $c_1,c_2$ are vertices. A concrete example of composition by identifying edges  will be given in Section \ref{sec:NTK} when computing the neural tangent kernel.

\begin{proposition}[Trace]\label{prop:trace}
Let $G=(G,\fd,((c_1), (c_2)), \{\mathbf{X}_c\}_{c\in C\setminus\{c_1,c_2\}})$ be such that $\fd(c_1)=\fd(c_2)$.
Define the \emph{trace} $\mathrm{Tr}(G)$ of $G$ to be the graph obtained by identifying $c_1$ with $c_2$ and fixing the resulting cell to $\mathbf{1}$ or $\mathbf{I}$ accordingly. Again, if any vertices $v_1, v_2$ are identified as a result (e.g., when $c_1,c_2$ are edges), the resulting vertex is given input $\mathbf{X}_1\odot \mathbf{X}_2$ where $\mathbf{X}_1,\mathbf{X}_2$ are the respective inputs of $v_1$ and $v_2$.
Then we have 
\[
\mathrm{Tr}(\bW_G) = \bW_{\mathrm{Tr}(G)}.
\]
\end{proposition}
\begin{proof}
Assume that $c_1,c_2$ are vertices (the argument is essentially identical if they are edges). Let $c_1\land c_2$ denote the unified vertex in $\mathrm{Tr}(G)$, and $C'$ be the set of cells of $\mathrm{Tr}(G)$. Then there is a canonical map $\phi: C \to C'$ which restricts to a bijection between $C\setminus\{c_1, c_2\}$ and $C'\setminus\{c_1 \wedge c_2\}$, and such that $\phi(c_1)=\phi(c_2)=c_1 \wedge c_2$.
Moreover, this map induces a bijection between the indexations of $V(G)$ for which $i_{c_1} = i_{c_2}$ and the indexations $\mathbf{i}_{V(\mathrm{Tr}(G))}$ of $V(\mathrm{Tr}(G))$, which is given by $i_v \mapsto i_{\phi(v)}$ and under which $[\mathbf{X}_{c}]_{i_c} = [\mathbf{X}_{\phi(c)}]_{i_{\phi(c)}}$ for all $c \in C \setminus \{c_1, c_2\}$. It follows that
\begin{align*}
    \mathrm{Tr}(\bW_G) &= \sum_{\mathbf{i}_{V(G)}} \delta_{i_{c_1}i_{c_2}} 
    \prod_{c \in C\setminus\{c_1, c_2\}} [\mathbf{X}_{\phi(c)}]_{i_{\phi(c)}}= \sum_{\mathbf{i}_{V(\mathrm{Tr}(G))}} 
    \prod_{c' \in C'\setminus\{c_1 \wedge c_2\}} [\mathbf{X}_{c'}]_{i_{c'}}
\\
    &
    = \sum_{\mathbf{i}_{V(\mathrm{Tr}(G))}} 
    \prod_{c' \in C'\setminus\{c_1 \wedge c_2\}} [\mathbf{X}_{c'}]_{i_{c'}}
    [\mathbf{1}]_{i_{c_1 \wedge c_2}} = \bW_{\mathrm{Tr}(G)}.
\end{align*}
\end{proof}

Note that the cyclic property of the trace translates into the formation of a cycle in $\mathrm{Tr}(G)$, as depicted in  Figure \ref{fig:traces}).
\begin{figure}[ht]
        \centering
        \tikzset{every picture/.style={line width=0.75pt}} 

\begin{tikzpicture}[x=0.75pt,y=0.75pt,yscale=-1,xscale=1]

\draw    (150.74,960.8) -- (184.01,960.95) ;
\draw [shift={(169.77,960.89)}, rotate = 180.27] [color={rgb, 255:red, 0; green, 0; blue, 0 }  ][line width=0.75]    (4.37,-1.32) .. controls (2.78,-0.56) and (1.32,-0.12) .. (0,0) .. controls (1.32,0.12) and (2.78,0.56) .. (4.37,1.32)   ;
\draw  [color={rgb, 255:red, 0; green, 0; blue, 0 }  ,draw opacity=1 ][fill={rgb, 255:red, 0; green, 0; blue, 0 }  ,fill opacity=1 ] (180.81,963.91) .. controls (182.44,965.67) and (185.2,965.78) .. (186.97,964.15) .. controls (188.73,962.52) and (188.84,959.76) .. (187.21,958) .. controls (185.57,956.23) and (182.82,956.12) .. (181.05,957.76) .. controls (179.29,959.39) and (179.18,962.14) .. (180.81,963.91) -- cycle ;
\draw  [color={rgb, 255:red, 74; green, 144; blue, 226 }  ,draw opacity=1 ][fill={rgb, 255:red, 255; green, 255; blue, 255 }  ,fill opacity=1 ][line width=1.5]  (147.54,963.75) .. controls (149.17,965.52) and (151.93,965.63) .. (153.69,963.99) .. controls (155.46,962.36) and (155.57,959.6) .. (153.93,957.84) .. controls (152.3,956.07) and (149.55,955.97) .. (147.78,957.6) .. controls (146.02,959.23) and (145.91,961.99) .. (147.54,963.75) -- cycle ;
\draw    (184.01,960.95) -- (217.28,961.11) ;
\draw [shift={(203.04,961.04)}, rotate = 180.27] [color={rgb, 255:red, 0; green, 0; blue, 0 }  ][line width=0.75]    (4.37,-1.32) .. controls (2.78,-0.56) and (1.32,-0.12) .. (0,0) .. controls (1.32,0.12) and (2.78,0.56) .. (4.37,1.32)   ;
\draw  [color={rgb, 255:red, 0; green, 0; blue, 0 }  ,draw opacity=1 ][fill={rgb, 255:red, 0; green, 0; blue, 0 }  ,fill opacity=1 ] (214.08,964.07) .. controls (215.72,965.83) and (218.47,965.94) .. (220.24,964.31) .. controls (222,962.67) and (222.11,959.92) .. (220.48,958.15) .. controls (218.84,956.39) and (216.09,956.28) .. (214.32,957.91) .. controls (212.56,959.54) and (212.45,962.3) .. (214.08,964.07) -- cycle ;
\draw    (217.28,961.11) -- (250.55,961.27) ;
\draw [shift={(236.32,961.2)}, rotate = 180.27] [color={rgb, 255:red, 0; green, 0; blue, 0 }  ][line width=0.75]    (4.37,-1.32) .. controls (2.78,-0.56) and (1.32,-0.12) .. (0,0) .. controls (1.32,0.12) and (2.78,0.56) .. (4.37,1.32)   ;
\draw  [color={rgb, 255:red, 0; green, 0; blue, 0 }  ,draw opacity=1 ][fill={rgb, 255:red, 255; green, 255; blue, 255 }  ,fill opacity=1 ][line width=1.5]  (247.35,964.22) .. controls (248.99,965.99) and (251.74,966.1) .. (253.51,964.46) .. controls (255.27,962.83) and (255.38,960.08) .. (253.75,958.31) .. controls (252.12,956.54) and (249.36,956.44) .. (247.59,958.07) .. controls (245.83,959.7) and (245.72,962.46) .. (247.35,964.22) -- cycle ;
\draw    (376.13,970.11) -- (330.13,970.11) ;
\draw [shift={(350.73,970.11)}, rotate = 360] [color={rgb, 255:red, 0; green, 0; blue, 0 }  ][line width=0.75]    (4.37,-1.32) .. controls (2.78,-0.56) and (1.32,-0.12) .. (0,0) .. controls (1.32,0.12) and (2.78,0.56) .. (4.37,1.32)   ;
\draw  [color={rgb, 255:red, 0; green, 0; blue, 0 }  ,draw opacity=1 ][fill={rgb, 255:red, 0; green, 0; blue, 0 }  ,fill opacity=1 ] (356.1,945.3) .. controls (357.85,943.66) and (357.95,940.9) .. (356.31,939.14) .. controls (354.67,937.38) and (351.91,937.29) .. (350.15,938.93) .. controls (348.4,940.57) and (348.3,943.33) .. (349.94,945.08) .. controls (351.58,946.84) and (354.34,946.94) .. (356.1,945.3) -- cycle ;
\draw    (376.13,970.11) -- (353.13,942.11) ;
\draw [shift={(366.78,958.74)}, rotate = 230.6] [color={rgb, 255:red, 0; green, 0; blue, 0 }  ][line width=0.75]    (4.37,-1.32) .. controls (2.78,-0.56) and (1.32,-0.12) .. (0,0) .. controls (1.32,0.12) and (2.78,0.56) .. (4.37,1.32)   ;
\draw  [color={rgb, 255:red, 0; green, 0; blue, 0 }  ,draw opacity=1 ][fill={rgb, 255:red, 0; green, 0; blue, 0 }  ,fill opacity=1 ] (333.1,973.3) .. controls (334.85,971.66) and (334.95,968.9) .. (333.31,967.14) .. controls (331.67,965.38) and (328.91,965.29) .. (327.15,966.93) .. controls (325.4,968.57) and (325.3,971.33) .. (326.94,973.08) .. controls (328.58,974.84) and (331.34,974.94) .. (333.1,973.3) -- cycle ;
\draw    (330.13,970.11) -- (353.13,942.11) ;
\draw [shift={(343.15,954.26)}, rotate = 129.4] [color={rgb, 255:red, 0; green, 0; blue, 0 }  ][line width=0.75]    (4.37,-1.32) .. controls (2.78,-0.56) and (1.32,-0.12) .. (0,0) .. controls (1.32,0.12) and (2.78,0.56) .. (4.37,1.32)   ;
\draw  [color={rgb, 255:red, 0; green, 0; blue, 0 }  ,draw opacity=1 ][fill={rgb, 255:red, 0; green, 0; blue, 0 }  ,fill opacity=1 ] (379.1,973.3) .. controls (380.85,971.66) and (380.95,968.9) .. (379.31,967.14) .. controls (377.67,965.38) and (374.91,965.29) .. (373.15,966.93) .. controls (371.4,968.57) and (371.3,971.33) .. (372.94,973.08) .. controls (374.58,974.84) and (377.34,974.94) .. (379.1,973.3) -- cycle ;
\draw    (270.5,959.55) -- (308.5,959.55) ;
\draw [shift={(310.5,959.55)}, rotate = 180] [color={rgb, 255:red, 0; green, 0; blue, 0 }  ][line width=0.75]    (7.65,-2.3) .. controls (4.86,-0.97) and (2.31,-0.21) .. (0,0) .. controls (2.31,0.21) and (4.86,0.98) .. (7.65,2.3)   ;

\draw (190.33,947.39) node [anchor=north west][inner sep=0.75pt]  [font=\tiny,rotate=-0.89]  {$W_{1}{}$};
\draw (157.33,947.39) node [anchor=north west][inner sep=0.75pt]  [font=\tiny,rotate=-0.89]  {$W_{2}{}$};
\draw (223.33,947.39) node [anchor=north west][inner sep=0.75pt]  [font=\tiny,rotate=-0.89]  {$W_{0}{}$};
\draw (142.32,969.14) node [anchor=north west][inner sep=0.75pt]  [font=\scriptsize]  {$\mathbf{W}_{G}(\mathbf{x}) =W_{2} W_{1} W_{0}\mathbf{x}$};
\draw (323.33,941.39) node [anchor=north west][inner sep=0.75pt]  [font=\tiny,rotate=-0.89]  {$W_{1}{}$};
\draw (344.33,975.39) node [anchor=north west][inner sep=0.75pt]  [font=\tiny,rotate=-0.89]  {$W_{2}{}$};
\draw (366.33,942.39) node [anchor=north west][inner sep=0.75pt]  [font=\tiny,rotate=-0.89]  {$W_{0}{}$};
\draw (295.32,988.53) node [anchor=north west][inner sep=0.75pt]  [font=\scriptsize]  {$\mathbf{W}_{\mathrm{Tr}( G)} =\mathrm{Tr}( W_{2} W_{1} W_{0})$};

\end{tikzpicture}
        \caption{The trace of a path graph.}
        
        \label{fig:traces}
\end{figure}
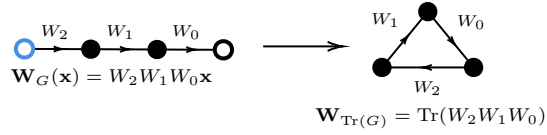

\begin{remark}Propositions \ref{def:graph_wedge}, \ref{def:swap_vertices}, and \ref{prop:composition} are analogous to the rules for entry-wise products, transposition and composition in \cite{male2018trafficdistributionsindependencepermutation} (namely (7), (4) and (3) in Example 1.3, respectively).
\end{remark}

We close this section by showing how differentiation can be performed graphically. To that end, it will be convenient to introduce the notion of sums of graphs.

Let $\{G_i\}_i$ be a set operator graphs for which the associated $\{
\mathbf{W}_{G_i}\}_{i}$ all belong to the same space. Let $\mathfrak{F}=(\mathfrak{F}(\{G_i\}_i),+)$ denote the free abelian group generated by the $G_i$. Then the map $G\mapsto \mathbf{W}_G$ can be uniquely extended to $\mathfrak{F}$ by
\[
    \mathbf{W}_{aG_1+bG_2}:=a\mathbf{W}_{G_1}+b\mathbf{W}_{G_2},\quad \forall a,b\in \mathbb{N}, aG_1+bG_2\in \mathfrak{F}.
\]
In what follows, linear combinations of operator graphs are always understood as elements of an underlying free abelian group (which we will not specify).

\begin{definition}[Differential]\label{def:graphdiff}
    Let $G=(G,\fd, (\mathcal{F}_{\mathrm{in}},\mathcal{F}_{\mathrm{out}}), \{\mathbf{X}_c\}_{c\notin \mathcal{F}})$ for which $\mathcal{F}_\mathrm{in}=(c_1,...,c_d)$ for some $d>0$. Fix a sequence $(\mathbf{X}_{c_j})_{j=1}^d$ of possible inputs, and define
    \[ 
        (\mathrm{d}_iG)_{(\mathbf{X}_{c_j})_{j\in [d]}}:=(G,\fd, ((c_i),\mathcal{F}_{\mathrm{out}}), \{\mathbf{X}_c\}_{c\in C \setminus (\mathcal{F}_{\mathrm{out}} \cup \{c_i\})}),
    \]
    namely the graph obtained from $G$ by inserting $\mathbf{X}_{c_j}$ in $c_j$ for all $j\neq i$.
    Then the \emph{differential} of $G$ at $(\mathbf{X}_{c_j})_{j\in [d]}$ is defined as
    \[
    (\mathrm{d}G)_{(\mathbf{X}_{c_j})_{j\in [d]}} := 
    \sum_{i=1}^{d}
    (\mathrm{d}_iG)_{(\mathbf{X}_{c_j})_{j\in [d]}}.
    \]
    \end{definition}
    \begin{proposition}\label{prop:differentiate}
    Let $(G,\fd, (\mathcal{F}_{\mathrm{in}},\mathcal{F}_{\mathrm{out}}), \{\mathbf{X}_c\}_{c\notin \mathcal{F}})$ be such that $|\mathcal{F}_{\mathrm{in}}|=d>0$. Denote its $i$-th partial derivative by $\mathrm{d}_i\mathbf{W}_G$ and its total differential by $\mathrm{d}\mathbf{W}_G$. Then
    \[
    (\mathrm{d}\mathbf{W}_G)_{(\mathbf{X}_{c_j})_{j\in [d]}} 
    = \sum_{i=1}^{d}
    (\mathrm{d}_i \mathbf{W}_G)_{(\mathbf{X}_{c_j})_{j\in [d]}} = \sum_{i=1}^{d}
    \bW_{(\mathrm{d}_iG)_{(\mathbf{X}_{c_j})_{j\in [d]}}}= 
    \bW_{(\mathrm{d}G)_{(\mathbf{X}_{c_j})_{j\in [d]}}} 
    \]
    \end{proposition}
    \begin{proof} 
        It suffices to show that $(\mathrm{d}_1 \bW_{G})_{(\mathbf{X}_{c_j})_{j\in [d]}} = \bW_{(\mathrm{d}_1 G)_{(\mathbf{X}_{c_j})_{j\in [d]}}}$. The claim then follows by definition of the total differential.

        Let $\mathcal{F}_{\mathrm{in}}=(c_1,...,c_d)$, so that $\mathrm{d}_1\mathbf{W}_{G}$ denotes the partial derivative with respect to the cell $c_1$'s input. Then since
        \[
            \Big\langle \mathbf{W}_G\big((\mathbf{X}_{c_i})_{i\in [d]}\big),\otimes_{c\in \mathcal{F}_{\mathrm{out}}}\mathbf{X}_{c}\Big\rangle=\sum_{\mathbf{i}_{V(G)}}\bigg(\prod_{c\in C\setminus\mathcal{F}_{\mathrm{in}}}[\mathbf{X}_{c}]_{i_c}\bigg)\bigg(\prod_{j=2}^{d}[\mathbf{X}_{c_j}]_{i_{c_j}}\bigg)[\mathbf{X}_{c_1}]_{i_{c_1}},
        \]
        and $(\mathrm{d}_1\mathbf{W}_G)_{(\mathbf{X}_{c_j})_{j\in [d]}}$ is linear (by definition), we that 
        \begin{align*}
        \Big\langle {(\mathrm{d}_1 \bW_{G})_{(\mathbf{X}_{c_j})_{j\in [d]}}(\mathbf{H}_{c_1})},
        {\otimes_{c \in \mathcal{F}_{\mathrm{out}}} \mathbf{X}_{c}}\Big\rangle
        &= 
        \sum_{\mathbf{i}_{V(G)}} \bigg(
        \prod_{c \in C \setminus \mathcal{F}_{\mathrm{in}}} [\mathbf{X}_{c}]_{i_c}
        \bigg)
        \bigg(
        \prod_{j=2}^{d}[\mathbf{X}_{c_{j}}]_{i_{c_{j}}}
        \bigg)
        [\mathbf{H}_{c_{1}}]_{i_{c_{1}}}
        \\
        &= 
        \sum_{\mathbf{i}_{V(G)}} \bigg(
        \prod_{c \in C \setminus \{c_1\}} [\mathbf{X}_{c}]_{i_c}
        \bigg)
        [\mathbf{H}_{c_{1}}]_{i_{c_{1}}}
        \\
        &= 
        \sum_{\mathbf{i}_{V((\mathrm{d}_1 G)_{(\mathbf{X}_{c_j})_{j\in [d]}})}} \bigg(
        \prod_{c \in C \setminus \{c_1\}} [\mathbf{X}_{c}]_{i_c}
        \bigg)
        [\mathbf{H}_{c_{1}}]_{i_{c_{1}}}
        \\
        &=: \Big\langle \bW_{(\mathrm{d}_1 G)_{{(\mathbf{X}_{c_j})_{j\in [d]}}}}(\mathbf{H}_{c_{1}}) 
        ,{\otimes_{c \in \mathcal{F}_{\mathrm{out}}} \mathbf{X}_{c}}\Big\rangle
        \end{align*}
        for any increment $\mathbf{H}_{c_1}$, where $C=C(G)$ denotes the set of cells of $G$.
    \end{proof}

\begin{figure}[ht]
        \centering
        \input{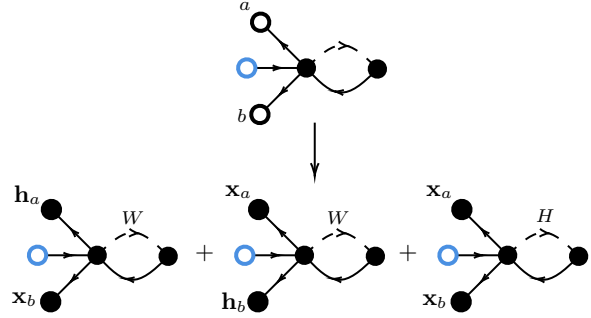}
        \caption{An operator graph (top) and its total differential at $(\mathbf{x}_a,\mathbf{x}_b,W)$ (right) evaluated in the increment vector \( (\mathbf{h}_a, \mathbf{h}_b, H) \).}
        \label{fig:derivative}
\end{figure}

\section{Neural network expansions}
\label{sec:trees}

Having outlined our correspondence between analytic and graphical operations, we now show how to express neural networks and related quantities using these graphs.

\subsection{Tree expansion for feed-forward neural networks.}
\label{sec:tree_exp_defs}

Fix sequences $(\varphi_{\ell}: \bR \to \bR ~|~ \ell \in \N_{>0})$ of polynomial \emph{activation functions}, $(N_\ell \in \N_{>0} ~|~  \ell \in \N)$ of \textit{layer dimensions} and $(W_{\ell} \in \mathbb{R}^{N_{\ell+1}\times N_{\ell}} ~|~ \ell \in \N)$ of weight matrices.
Recall that a \emph{feed-forward neural network} $\Phi_L$ of depth $L$ with no bias term is defined by the recursion
\begin{equation}\label{def:NN}
     \Phi_0(\bx)=W_0x,\quad \Phi_{\ell+1}(\bx)=W_{\ell+1}\varphi_{\ell+1}(\Phi_{\ell}(\bx)).
\end{equation}
\noindent As explained in the introduction, we aim to derive an expansion for $\Phi_L(\bx)$ that linearizes the effect of the activation functions.

\bigskip 

This is achieved in the main result of the section (Theorem \ref{thm:tree_exp}), in which we expand $\Phi_L$ as a sum of operator graphs which are rooted trees (see Definition \ref{eq:gubspace}). Some examples of the trees arising in this expansion are plotted in Figure \ref{fig:some_trees}: they all have weight matrices as edge inputs, basis vectors as leaf inputs and their internal nodes all have input $\mathbf{1}$. Each of these trees has a unique free vertex (its root) which is an out-vertex, and the latter always has out-degree one. 
\begin{figure}[ht]
    \centering
    \adjustbox{scale=0.9}{\input{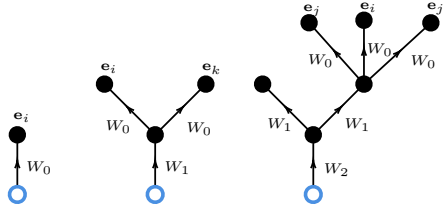}}
    \caption{Examples of trees arising in the expansion of $\Phi_{3}$}
    \label{fig:some_trees}
\end{figure}

 Consider $k$ such trees $\tau_1,...,\tau_k$ (along with the data required to make them operator graphs), assuming that their respective root edges are all fixed to the same $W_\ell$ for some $\ell>0$ (which implies that the dimension assigned to each of the root vertices is $N_{\ell +1}$). 
 We define $[\tau_1\cdots\tau_k]_{\ell}$ to be the tree (
\begin{tikzpicture}[x=0.75pt,y=0.75pt,yscale=-.5,xscale=.5]

\draw    (79.64,35) -- (120.14,35) ;
\draw [shift={(102.29,35)}, rotate = 180] [color={rgb, 255:red, 0; green, 0; blue, 0 }  ][line width=0.75]    (4.37,-1.32) .. controls (2.78,-0.56) and (1.32,-0.12) .. (0,0) .. controls (1.32,0.12) and (2.78,0.56) .. (4.37,1.32)   ;
\draw  [color={rgb, 255:red, 74; green, 144; blue, 226 }  ,draw opacity=1 ][fill={rgb, 255:red, 255; green, 255; blue, 255 }  ,fill opacity=1 ][line width=1.5]  (82.62,38.18) .. controls (84.37,36.54) and (84.47,33.79) .. (82.83,32.03) .. controls (81.19,30.27) and (78.43,30.18) .. (76.67,31.82) .. controls (74.92,33.46) and (74.82,36.21) .. (76.46,37.97) .. controls (78.1,39.73) and (80.86,39.82) .. (82.62,38.18) -- cycle ;
\draw  [color={rgb, 255:red, 0; green, 0; blue, 0 }  ,draw opacity=1 ][fill={rgb, 255:red, 255; green, 255; blue, 255 }  ,fill opacity=1 ][line width=1.5]  (123.11,38.18) .. controls (124.87,36.54) and (124.97,33.79) .. (123.33,32.03) .. controls (121.68,30.27) and (118.93,30.18) .. (117.17,31.82) .. controls (115.41,33.46) and (115.32,36.21) .. (116.96,37.97) .. controls (118.6,39.73) and (121.36,39.82) .. (123.11,38.18) -- cycle ;

\draw (83,16) node [anchor=north west][inner sep=0.75pt]  [font=\tiny,rotate=-0.89]  {$W_{\ell}$};

\end{tikzpicture})$\circ (\wedge_{i=1}^k \tau_i)$, where $\circ$ denotes composition as defined in Section \ref{dictionary} and where for 
\begin{tikzpicture}[x=0.75pt,y=0.75pt,yscale=-.5,xscale=.5]

\draw    (79.64,35) -- (120.14,35) ;
\draw [shift={(102.29,35)}, rotate = 180] [color={rgb, 255:red, 0; green, 0; blue, 0 }  ][line width=0.75]    (4.37,-1.32) .. controls (2.78,-0.56) and (1.32,-0.12) .. (0,0) .. controls (1.32,0.12) and (2.78,0.56) .. (4.37,1.32)   ;
\draw  [color={rgb, 255:red, 74; green, 144; blue, 226 }  ,draw opacity=1 ][fill={rgb, 255:red, 255; green, 255; blue, 255 }  ,fill opacity=1 ][line width=1.5]  (82.62,38.18) .. controls (84.37,36.54) and (84.47,33.79) .. (82.83,32.03) .. controls (81.19,30.27) and (78.43,30.18) .. (76.67,31.82) .. controls (74.92,33.46) and (74.82,36.21) .. (76.46,37.97) .. controls (78.1,39.73) and (80.86,39.82) .. (82.62,38.18) -- cycle ;
\draw  [color={rgb, 255:red, 0; green, 0; blue, 0 }  ,draw opacity=1 ][fill={rgb, 255:red, 255; green, 255; blue, 255 }  ,fill opacity=1 ][line width=1.5]  (123.11,38.18) .. controls (124.87,36.54) and (124.97,33.79) .. (123.33,32.03) .. controls (121.68,30.27) and (118.93,30.18) .. (117.17,31.82) .. controls (115.41,33.46) and (115.32,36.21) .. (116.96,37.97) .. controls (118.6,39.73) and (121.36,39.82) .. (123.11,38.18) -- cycle ;

\draw (83,16) node [anchor=north west][inner sep=0.75pt]  [font=\tiny,rotate=-0.89]  {$W_{\ell}$};

\end{tikzpicture} the dimension of the vertices is given by $\fd(\text{\begin{tikzpicture}[x=0.75pt,y=0.75pt,yscale=-.5,xscale=.5]

\draw  [color={rgb, 255:red, 0; green, 0; blue, 0 }  ,draw opacity=1 ][fill={rgb, 255:red, 255; green, 255; blue, 255 }  ,fill opacity=1 ][line width=1]  (15.62,19.02) .. controls (17.37,17.38) and (17.47,14.63) .. (15.83,12.87) .. controls (14.19,11.11) and (11.43,11.01) .. (9.67,12.66) .. controls (7.92,14.3) and (7.82,17.05) .. (9.46,18.81) .. controls (11.1,20.57) and (13.86,20.66) .. (15.62,19.02) -- cycle ;

\end{tikzpicture}}\hspace{-3pt})=N_{\ell}, \fd(\text{\begin{tikzpicture}[x=0.75pt,y=0.75pt,yscale=-.5,xscale=.5]

\draw  [color={rgb, 255:red, 74; green, 144; blue, 226 }  ,draw opacity=1 ][fill={rgb, 255:red, 255; green, 255; blue, 255 }  ,fill opacity=1 ][line width=1]  (15.62,19.02) .. controls (17.37,17.38) and (17.47,14.63) .. (15.83,12.87) .. controls (14.19,11.11) and (11.43,11.01) .. (9.67,12.66) .. controls (7.92,14.3) and (7.82,17.05) .. (9.46,18.81) .. controls (11.1,20.57) and (13.86,20.66) .. (15.62,19.02) -- cycle ;

\end{tikzpicture}}\hspace{-3pt})=N_{\ell+1}$
 (recalling that empty products in $\bR^N$ are valued as $\mathbf{1}_N$, we have $[~]_{\ell} = 
\begin{tikzpicture}[x=0.75pt,y=0.75pt,yscale=-.5,xscale=.5]

\draw    (79.64,35) -- (120.14,35) ;
\draw [shift={(102.29,35)}, rotate = 180] [color={rgb, 255:red, 0; green, 0; blue, 0 }  ][line width=0.75]    (4.37,-1.32) .. controls (2.78,-0.56) and (1.32,-0.12) .. (0,0) .. controls (1.32,0.12) and (2.78,0.56) .. (4.37,1.32)   ;
\draw  [color={rgb, 255:red, 74; green, 144; blue, 226 }  ,draw opacity=1 ][fill={rgb, 255:red, 255; green, 255; blue, 255 }  ,fill opacity=1 ][line width=1.5]  (82.62,38.18) .. controls (84.37,36.54) and (84.47,33.79) .. (82.83,32.03) .. controls (81.19,30.27) and (78.43,30.18) .. (76.67,31.82) .. controls (74.92,33.46) and (74.82,36.21) .. (76.46,37.97) .. controls (78.1,39.73) and (80.86,39.82) .. (82.62,38.18) -- cycle ;
\draw  [color={rgb, 255:red, 0; green, 0; blue, 0 }  ,draw opacity=1 ][fill={rgb, 255:red, 0; green, 0; blue, 0 }  ,fill opacity=1 ][line width=1.5]  (123.11,38.18) .. controls (124.87,36.54) and (124.97,33.79) .. (123.33,32.03) .. controls (121.68,30.27) and (118.93,30.18) .. (117.17,31.82) .. controls (115.41,33.46) and (115.32,36.21) .. (116.96,37.97) .. controls (118.6,39.73) and (121.36,39.82) .. (123.11,38.18) -- cycle ;

\draw (83,16) node [anchor=north west][inner sep=0.75pt]  [font=\tiny,rotate=-0.89]  {$W_{\ell}$};

\end{tikzpicture}$).
 Then using this bracket operation, any tree like those in Figure \ref{fig:some_trees} can be constructed from the following family of operator graphs:
\[
\{[\bullet_{\mathbf{e}_i}]_0 \}_{i\in [N_0]}, \text{ where }[\bullet_{\mathbf{e}_i}]_0=\text{\begin{tikzpicture}[x=0.75pt,y=0.75pt,yscale=-.4,xscale=.4]

\draw    (79.64,35) -- (120.14,35) ;
\draw [shift={(102.29,35)}, rotate = 180] [color={rgb, 255:red, 0; green, 0; blue, 0 }  ][line width=0.75]    (4.37,-1.32) .. controls (2.78,-0.56) and (1.32,-0.12) .. (0,0) .. controls (1.32,0.12) and (2.78,0.56) .. (4.37,1.32)   ;
\draw  [color={rgb, 255:red, 74; green, 144; blue, 226 }  ,draw opacity=1 ][fill={rgb, 255:red, 255; green, 255; blue, 255 }  ,fill opacity=1 ][line width=1.5]  
(82.62,38.18) .. controls (84.37,36.54) and (84.47,33.79) .. (82.83,32.03) .. controls (81.19,30.27) and (78.43,30.18) .. (76.67,31.82) .. controls (74.92,33.46) and (74.82,36.21) .. (76.46,37.97) .. controls (78.1,39.73) and (80.86,39.82) .. (82.62,38.18) -- cycle ;
\draw [color={rgb, 255:red, 0; green, 0; blue, 0 }, draw opacity=1 ]
[fill={rgb, 255:red, 0; green, 0; blue, 0 }  ,fill opacity=1 ]
[line width=1.5]  
(123.11,38.18) .. controls (124.87,36.54) and (124.97,33.79) .. (123.33,32.03) .. controls (121.68,30.27) and (118.93,30.18) .. (117.17,31.82) .. controls (115.41,33.46) and (115.32,36.21) .. (116.96,37.97) .. controls (118.6,39.73) and (121.36,39.82) .. (123.11,38.18) -- cycle ;

\draw (78,9) node [anchor=north west][inner sep=0.75pt]  [font=\tiny,rotate=-0.89]  {$W_{0}$};
\draw (127,26.82) node [anchor=north west][inner sep=0.75pt]  [font=\tiny,rotate=-0.89]  {$\mathbf{e}_{i}{}{}$};

\end{tikzpicture}}
\text{so that } \fd(\bullet) = N_0, \fd(\text{\begin{tikzpicture}[x=0.75pt,y=0.75pt,yscale=-.5,xscale=.5]

\draw  [color={rgb, 255:red, 74; green, 144; blue, 226 }  ,draw opacity=1 ][fill={rgb, 255:red, 255; green, 255; blue, 255 }  ,fill opacity=1 ][line width=1]  (15.62,19.02) .. controls (17.37,17.38) and (17.47,14.63) .. (15.83,12.87) .. controls (14.19,11.11) and (11.43,11.01) .. (9.67,12.66) .. controls (7.92,14.3) and (7.82,17.05) .. (9.46,18.81) .. controls (11.1,20.57) and (13.86,20.66) .. (15.62,19.02) -- cycle ;

\end{tikzpicture}}\hspace{-3pt}) = N_1,
\]
where $\fd$ is the dimension function for $[\bullet_{\mathbf{e}_i}]_0$.
For instance, the first and second tree in Figure \ref{fig:some_trees} can be written as $[\bullet_{\mathbf{e}_i}]_0$ and $[[\bullet_{\mathbf{e}_i}]_0[\bullet_{\mathbf{e}_j}]_0]_1$. 

\begin{definition}\label{eq:gubspace}
Fix $L>0$, $\mathbb{T}_L$ is the set of operator graphs defined inductively by
\begin{equation*}
         \bT_{0}= \{[\bullet_{\mathbf{e}_1}]_0,\cdots,[\bullet_{\mathbf{e}_{N_0}}]_0\},
         \quad
         \bT_{\ell+1} := \{[ \tau_1  \cdots \tau_M ]_{\ell+1} ~|~ M \geq 0, ~ \forall i \in [M]~\tau_i \in \bT_{\ell}\}.
\end{equation*}
\end{definition}

\begin{remark}
    Note that for any $\tau \in \bT_{\ell}$, there is a unique dimension function $\fd$ that is compatible with its inputs:
if an edge $(u,v)$ is labeled by $W_{k}$ then we must have $\fd(u) = N_{k+1}$ and  $\fd(v) = N_{k}$. In what follows, we will implicitly assume this choice of dimension function for all of the trees being considered.
\end{remark}

By definition of the bracket operation $[~ \cdot ~]$ and the graph operations $\wedge$ and $\circ$ (see Propositions \ref{def:graph_wedge} and \ref{prop:composition}), we have that 
\[\bW_{[ \tau_1  \cdots \tau_M ]_{\ell+1}} = W_{\ell+1} ( \odot_{i=1}^M \bW_{\tau_i})
\]
for every $\tau = [ \tau_1  \cdots \tau_M ]_{\ell+1} \in \bT_{\ell+1}$.

Note that by construction, the trees in $\mathbb{T}_L$ are \textit{non-plane}, meaning that the order of the edges coming out of each vertex is ignored. For any $\tau\in \mathbb{T}_\ell$, it will thus be useful to keep track of how many ways we it can be obtained from plane trees (where the order of edges is not ignored).
Following Gubinelli \cite{GubTrees}, we call this quantity the \textit{symmetry factor} $s(\tau)$ of $\tau$, and define it recursively as
\begin{equation}\label{def:symmetry-factor}
    s([\bullet_{\mathbf{e}_i}]_0) = 1, \quad s([(\tau_1)^{k_1} \cdots (\tau_m)^{k_m}]_{\ell}) =  \prod_i (k_i!) s(\tau_i)^{k_i},
\end{equation}
where $\tau_1,...,\tau_m$ are assumed to be distinct and $\tau_1^{k_1}$ denotes the repetition of $\tau_1$, $k_1$ times. Lastly, we introduce the following bit of notation: for any vector $\bx$, we let
\begin{equation}\label{def:x_tree}
\bx_{[\bullet_{\mathbf{e}_i}]_0} = [\bx]_i, \quad \bx_{[\tau_1 \cdots \tau_M]_{\ell}} =  \prod_{i=1}^M \bx_{\tau_i} \in \bR,
\end{equation}
and for any activation function $\varphi$, \begin{equation}\label{def:varphi}
    \varphi_{[{\bullet_{\mathbf{e}_i}}]_0} = 1, \quad \varphi_{[\tau_1 \cdots \tau_M]_{\ell+1}} = \varphi^{(M)}_{\ell+1}(0) \cdot \prod_{i=1}^M \varphi_{\tau_i} \in \bR,
\end{equation}
where $\varphi^{(M)}$ is the $M$-th derivative of $\varphi$.

\begin{definition}\label{def:bT_L(x)}
    Given $\bx \in \bR^{N_0}$ we define the sets $\bT_{\ell}(\bx)$ of operator graphs obtained from the trees in $\bT_{\ell}$ by changing all the leaves' inputs from $\bullet_{\mathbf{e}_i}$ to $\bx$. We will henceforth denote this operation by  $\bullet_{\mathbf{e}_i} \mapsto \bullet_\bx$.
\end{definition}


If we extend $s(\eta)$ and $\varphi_\eta$ to $\bT_{\ell}(\bx)$ setting the base cases as $s([\bullet_{\bx}]_0) = \varphi_{[\bullet_{\bx}]_0} = 1$, then clearly $\varphi_{\tau}=\varphi_{\tau(\bx)}$ for all $\tau \in \bT_\ell$ since the definition of $\varphi$ only depends on the out-degree of the vertices, which is left unchanged.
While it is not generally true that $s(\tau) = s(\tau(\bx))$, we nonetheless have the useful equality
\begin{equation}\label{eqn:symmetric_Tx}
    \frac{1}{s(\eta)}\bW_{\eta} = \sum_{\tau \in \bT_\ell : \tau(\bx) = \eta} \frac{\bx_{\tau}}{s(\tau)}\bW_{\tau}.
\end{equation}
This follows from the fact that $s(\tau(\bx))/s(\tau)$ counts the number of ways in which the {non-plane} tree $\tau$ is obtained from $\tau(\bx)$ by applying the operation $\bullet_{\bx} \mapsto \bullet_{\mathbf{e}_i}$.

\begin{example}
    Consider the trees $\eta := [[\bullet_\bx]_0 ~ [\bullet_\bx]_0]_1$ and $\tau := [[\bullet_1]_0 ~ [\bullet_2]_0]_1$. 
    Then $\eta = \tau(\bx)$, but $s(\eta) = 2! s([\bullet_{\bx}]_0) = 2$ while $s(\tau) = s([\bullet_1]_0)s([\bullet_2]_0)=1$.
    Notice that $\tau$ can be obtained both by $[[\bullet_\bx]_0 ~ [\bullet_\bx]_0]_1 \mapsto [[\bullet_1]_0 ~ [\bullet_2]_0]_1$ and $[[\bullet_\bx]_0 ~ [\bullet_\bx]_0]_1 \mapsto [[\bullet_2]_0 ~ [\bullet_1]_0]_1$, so in $2 = s(\eta)/s(\tau)$ ways.
\end{example}

We can now state the main result of this section.
        
\begin{theorem}[Tree expansion]
\label{thm:tree_exp}
Let $\Phi_L$ be as defined in (\ref{def:NN}). Then for any $\mathbf{x}\in \mathbb{R}^{N_0}$, $\Phi_L(\bx)$ admits the following expansion over $\bT_{\ell}$:
\begin{equation}\label{eqn:tree_exp}
     \Phi_{\ell}(\bx) = \sum_{\tau \in \bT_{\ell}} \frac{\varphi_{\tau}}{s(\tau)} \bx_{\tau} \bW_{\tau}  
     =\sum_{\eta \in \bT_{\ell}(\bx)} \frac{\varphi_{\eta}}{s(\eta)} \bW_{\eta}
     \in \bR^{N_{\ell+1}}.
\end{equation}
\end{theorem}

\begin{remark}
    For $\ell = 1$ this expansion reduces to the classical (vector valued) Taylor expansion since $s([[\bullet_{\mathbf{e}_1}]_0^{k_1} \cdots [\bullet_{\mathbf{e}_d}]_0^{k_d}]_1) = \prod_i k_i!$.
\end{remark}

\begin{remark}
    Despite the space $\bT_{\ell}$ being infinite, this sum is always well-defined since only finitely many $\varphi_\tau$ are non-zero (having assumed that the $\varphi$ are polynomials).
\end{remark}

\begin{proof}
     We proceed by induction on depth $\ell$ for the first equality. 
     If $\ell=0$ then $\Phi_0(\bx) = W_0\bx$ and 
     \[
      \sum_{\tau \in \bT_{0}} \frac{\varphi_{\tau} \bx_{\tau}}{s(\tau)} \bW_{\tau}
      =
       \sum_{i=1}^{N_0} \frac{1\cdot[\bx]_i}{1} W_0\mathbf{e}_i
        = W_0 \left( \sum_{i=1}^{N_0} [\bx]_i \mathbf{e}_i \right)
        = W_0 \bx \in \bR^{N_1}     
    \]
    as needed. 
    
    Now assume then the claim holds for $\ell$.  For $\mathbf{y} \in \bR^{N_{\ell}}$, we write $\varphi_{\ell+1}(\mathbf{y})$ to mean $\varphi_{\ell+1}$ applied componentwise to $\mathbf{y}$. Then by Taylor expansion around the origin, 
    \[
    \varphi_{\ell+1}(\mathbf{y}) = \sum_{M=0}^{\infty} \frac{\varphi_{\ell +1}^{(M)}(0)}{M!} \mathbf{y}^{\odot M}
    \]
    where $\odot$ denotes the Hadamard product. Note that this sum is actually finite, since $\varphi_{\ell+1}$ is assumed to be a polynomial.
    
    If $y = \sum_{\tau \in \bT_{\ell}} \frac{\varphi_{\tau} }{s(\tau)} \bW_{\tau}\bx_{\tau}$ then Proposition \ref{prop:symmetric_expansion} yields
    \[
    y^{\odot M} = \sum_{\bar\tau = \llbracket\tau_1 \cdots \tau_M\rrbracket \in {\bX^M_{\bT_{\ell}}}} \frac{M! }{{\fs(\bar\tau)}} \odot_{i=1}^M \left( \frac{\varphi_{\tau_i} }{s(\tau_i)} \bW_{\tau_i} \bx_{\tau} \right).
    \]
    Here, $\mathbb{X}_{\bT_{\ell}}^M$ denotes the set of cardinality-$M$ multisets of elements of $\mathbb{T}_\ell$, and the $\mathfrak{s}(\bar{\tau})$ are symmetry factors defined as
    \[
        \mathfrak{s}(\llbracket \rrbracket)=1,\quad \mathfrak{s}(\llbracket (\eta_1)^{k_1}\cdots(\eta_m)^{k_m}\rrbracket):=\prod_{i=1}^m(k_i).
    \]
    ($\llbracket (\eta_1)^{k_1}\cdots(\eta_m)^{k_m}\rrbracket$ is the multiset containing $\eta_i\in \mathbb{T}_\ell$ with multiplicity $k_i$.) Letting $\mathbb{X}_{\mathbb{T}_\ell}=\cup_{M\geq 0}\mathbb{X}_{\mathbb{T}_\ell}^{M}$, we have a natural bijection between $\mathbb{T}_{\ell+1}$ and $\mathbb{X}_{\mathbb{T}_\ell}$ given by 
    \[
        \tau:=[\tau_1\cdots\tau_m]_{\ell+1}\mapsto \bar{\tau}:=\llbracket \tau_1\cdots \tau_m\rrbracket.
    \]
    Put simply, we identify a rooted non-plane tree in $\mathbb{T}_{\ell+1}$ with the multiset of subtrees ($\in \mathbb{T}_\ell$) rooted at its depth-$1$ vertices. Under this identification, $s(\tau) = {\fs(\bar\tau)} \prod_i s(\tau_i)$ and we conclude that
    \begin{align*}
        W_{\ell+1} \varphi_{\ell + 1}(\mathbf{y}) &= 
        W_{\ell+1} \sum_{M=0}^{\infty} \frac{\varphi^{(M)}(0)}{M!} \mathbf{y}^{\odot M} 
        \\ &=
        \sum_{M=0}^{\infty} \varphi^{(M)}(0) \sum_{\bar\tau = \llbracket\tau_1 \cdots \tau_M\rrbracket \in \bX^M_{\bT_{\ell}}} \frac{1}{\fs(\bar\tau)} W_{\ell+1} \left[ \odot_{i=1}^M \left( \frac{\varphi_{\tau_i}}{s(\tau_i)} \bW_{\tau_i} \bx_{\tau_i} \right) 
        \right]   
        \\ &=
        \sum_{M=0}^{\infty}\sum_{\bar\tau = \llbracket\tau_1 \cdots \tau_M\rrbracket \in \bX^M_{\bT_{\ell}}}
        \frac{ \varphi^{(M)}(0) \prod_i \varphi_{\tau_i}}{\fs(\bar\tau) \prod_i s(\tau_i)}
        \left(  W_{\ell+1} \left[ \odot_{i=1}^M \bW_{\tau_i} \right]    \right)
        \left( \prod_i \bx_{\tau_i} \right)
        \\ &= 
        \sum_{\tau \in \bT_{\ell+1}} \frac{\varphi_{\tau}\bx_{\tau}}{s(\tau)} \bW_{\tau}\in \bR^{N_{\ell+1}}.
    \end{align*}

    The second equality in the theorem statement readily follows from
    \begin{align*}
        \sum_{\eta \in \bT_{\ell}(\bx)} \frac{\varphi_{\eta}}{s(\eta)} \bW_{\eta}
        &=
        \sum_{\eta \in \bT_{\ell}(\bx)} \sum_{\tau \in \bT_\ell : \tau(\bx) = \eta} 
        \frac{\varphi_{\tau}}{s(\tau)} \bW_{\tau} \bx_{\tau}
        =
        \sum_{\tau \in \bT_{\ell}} 
        \frac{\varphi_{\tau}}{s(\tau)} \bW_{\tau} \bx_{\tau}.
    \end{align*}
\end{proof}


\subsection{Expansions for related quantities}
\label{sec:related_expansions}
Combined with the dictionary in Section \ref{dictionary}, this expansion allows us to express the coordinates of $\Phi_L$ and various related quantities as linear combinations of values of product graphs. Note that in the figures below, we label an edge with input $W_\ell$ by $\ell$ for simplicity.

\subsubsection{Entry-wise expansion}

 We use Theorem \ref{thm:tree_exp} to derive an expansion for the $k$-th entry of $\Phi_L(\mathbf{x})$ (in the standard basis).
To that end, let $\mathbb{T}_{\ell,k}$ be the set of trees obtained from $\bT_{\ell}$ by fixing the root to $\mathbf{e}_k$ (\emph{i.e.} applying $\text{\begin{tikzpicture}[x=0.75pt,y=0.75pt,yscale=-.5,xscale=.5]

\draw  [color={rgb, 255:red, 74; green, 144; blue, 226 }  ,draw opacity=1 ][fill={rgb, 255:red, 255; green, 255; blue, 255 }  ,fill opacity=1 ][line width=1]  (15.62,19.02) .. controls (17.37,17.38) and (17.47,14.63) .. (15.83,12.87) .. controls (14.19,11.11) and (11.43,11.01) .. (9.67,12.66) .. controls (7.92,14.3) and (7.82,17.05) .. (9.46,18.81) .. controls (11.1,20.57) and (13.86,20.66) .. (15.62,19.02) -- cycle ;

\end{tikzpicture}}\hspace{-3pt} \mapsto \bullet_{\mathbf{e}_k}$), noting that this makes the tree a product graph.
Then we can write the $k$-th entry of $\Phi_L$ a
\begin{equation}\label{eqn:tree_exp_coords}
    [\Phi_{L}(\bx)]_k = \sum_{\tau \in \mathbb{T}_{L,k}} \frac{\varphi_{\tau}}{s(\tau)} \bW_{\tau} \bx_{\tau} 
= \sum_{\eta \in \mathbb{T}_{L,k}(\bx)} \frac{\varphi_{\eta}}{s(\eta)} \bW_{\eta}   \in \bR.
\end{equation}
where $s(\tau), \varphi_\tau$ and $\bx_{\tau}$ are defined as in equations (\ref{def:symmetry-factor}),(\ref{def:x_tree}) and (\ref{def:varphi}), as these definitions only depend on $\tau$ through its set of vertices and edges.

\subsubsection{Input-output Jacobian} 
Consider the (post-activation) Jacobian of $\Phi_L$ with respect to an input $\bx$, which we define as the matrix representation $\mathbf{J}_{L,\mathbf{x}}\in \mathbb{R}^{N_L\times N_0}$ of
\begin{equation}\label{def:Jacobian}
    \mathrm{d}(\varphi_{L}\circ \Phi_{L-1})_{\bx}.
\end{equation}
Note that, following the same arguments as in the proof of Theorem \ref{thm:tree_exp}, we can write 
\begin{equation}
     (\varphi_{L}\circ \Phi_{L-1})(\bx) 
     = \sum_{\eta \in \bT_{L}(\bx)} \frac{\varphi_{\eta^*}}{s(\eta^*)} \bW_{\eta^*}
     \in \bR^{N_{\ell+1}}
\end{equation}
where $\eta^*$ denotes the tree obtained from $\eta$ by deleting the root and the edge which stems from it, and where $\varphi_{\eta^*}$ and $s(\eta^*)$ are equal to $\varphi_{\eta}$ and $s(\eta)$, respectively.
This operation is depicted in Figure \ref{fig:post-activation-trees}, and let $(\mathbb{T}_L(\mathbf{x}))^*=\{\eta^*: \eta\in \mathbb{T}_L(\mathbf{x})\}$.
\begin{figure}[ht]
    \centering
    \input{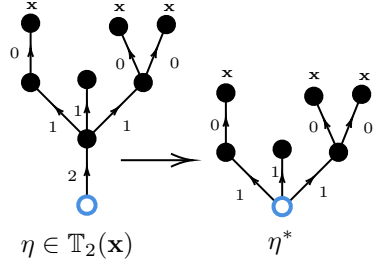}
    \caption{A tree $\eta\in \mathbb{T}_{2}(\mathbf{x})$ (left) and $\eta^*$ (right).}
    \label{fig:post-activation-trees}
\end{figure}

We now take the derivative. For any tree $\eta\in (\mathbb{T}(\mathbf{x}))^*$, we want to see it as a multi-linear map in the $\mathbf{x}$'s and compute the total derivative of said map at $(\mathbf{x},...,\mathbf{x})$. To be precise, let $k$ denote the number of vertices in $\eta$ with input $\mathbf{x}$, and $\eta_i$ denote the tree obtained from $\eta$ by applying $\bullet_{\bx}\mapsto \text{\begin{tikzpicture}[x=0.75pt,y=0.75pt,yscale=-.5,xscale=.5]

\draw  [color={rgb, 255:red, 0; green, 0; blue, 0 }  ,draw opacity=1 ][fill={rgb, 255:red, 255; green, 255; blue, 255 }  ,fill opacity=1 ][line width=1]  (15.62,19.02) .. controls (17.37,17.38) and (17.47,14.63) .. (15.83,12.87) .. controls (14.19,11.11) and (11.43,11.01) .. (9.67,12.66) .. controls (7.92,14.3) and (7.82,17.05) .. (9.46,18.81) .. controls (11.1,20.57) and (13.86,20.66) .. (15.62,19.02) -- cycle ;

\end{tikzpicture}}\hspace{-3pt}$ to the $i$-th such vertex (respectively). Then by Proposition \ref{prop:differentiate},
\[
    (\mathrm{d}\eta)_{\mathbf{x}} = \sum_{i=1}^k \eta_i
\]
and, arguing similarly as we did to get Equation (\ref{eqn:symmetric_Tx}) yields
\begin{equation}\label{eqn:symmetrix_dTx}
    \frac{1}{s(\eta)}\bW_{(\mathrm{d}\eta)_{\mathbf{x}}} = 
\sum_{\eta^*: \eta^*(\bx) = \eta} \frac{1}{s(\eta^*)}\bW_{\eta^*}.
\end{equation}
This equality holds since $s(\eta^*)/s(\eta)$ is equal to the number of ways in which $\eta^*$ can be obtained from a tree in $(\partial_{\bx} \mathbb{T}_L(\bx))^*$ as defined below, which is equal to the number of times it appears in the total derivative of $\eta$.

We let $\partial_\mathbf{x} \bT_{L}(\bx)$ denote the space of trees which can be obtained from trees in $\mathbb{T}_L(\bx)$ by applying $\bullet_{\bx}\mapsto \text{\begin{tikzpicture}[x=0.75pt,y=0.75pt,yscale=-.5,xscale=.5]

\draw  [color={rgb, 255:red, 0; green, 0; blue, 0 }  ,draw opacity=1 ][fill={rgb, 255:red, 255; green, 255; blue, 255 }  ,fill opacity=1 ][line width=1]  (15.62,19.02) .. controls (17.37,17.38) and (17.47,14.63) .. (15.83,12.87) .. controls (14.19,11.11) and (11.43,11.01) .. (9.67,12.66) .. controls (7.92,14.3) and (7.82,17.05) .. (9.46,18.81) .. controls (11.1,20.57) and (13.86,20.66) .. (15.62,19.02) -- cycle ;

\end{tikzpicture}}\hspace{-3pt}$ to exactly one of their leaves. Summarizing what we have said to far, we have
\begin{equation}\label{eqn:jacobian_graph_exp}
   \mathrm{d}(\varphi_{L}\circ \Phi_{L-1})_{\bx}=\sum_{\eta\in (\partial_{\bx} \mathbb{T}_L(\bx))^*} \frac{\varphi_\eta}{s(\eta)}\mathbf{W}_\eta,
\end{equation}
noting that $\varphi_\eta$ is invariant under $\bullet_{\bx}\mapsto \text{\begin{tikzpicture}[x=0.75pt,y=0.75pt,yscale=-.5,xscale=.5]

\draw  [color={rgb, 255:red, 0; green, 0; blue, 0 }  ,draw opacity=1 ][fill={rgb, 255:red, 255; green, 255; blue, 255 }  ,fill opacity=1 ][line width=1]  (15.62,19.02) .. controls (17.37,17.38) and (17.47,14.63) .. (15.83,12.87) .. controls (14.19,11.11) and (11.43,11.01) .. (9.67,12.66) .. controls (7.92,14.3) and (7.82,17.05) .. (9.46,18.81) .. controls (11.1,20.57) and (13.86,20.66) .. (15.62,19.02) -- cycle ;

\end{tikzpicture}}\hspace{-3pt}$. 

Recall that $\mathbf{J}_{L,\mathbf{x}}$ denotes the matrix representation of $\mathrm{d}(\varphi_{L}\circ \Phi_{L-1})_{\bx}$.
In Section \ref{sec:application_NNs}, we will study the distribution of the singular values of $\mathbf{J}_{L,\mathbf{x}}$ when the weight matrices of $\Phi_L$ are chosen at random. This will amount to estimating the elements of the sequence 
\[  \mathrm{Tr}\big((\mathbf{J}_{L,\mathbf{x}}\mathbf{J}_{L,\mathbf{x}}^\top)^k\big), \quad k \in \mathbb N.
\]
We note that every $\eta$ in the expansion of $\mathbf{J}_{L,\bx}$ has exactly two free vertices: its root ($\text{\begin{tikzpicture}[x=0.75pt,y=0.75pt,yscale=-.5,xscale=.5]

\draw  [color={rgb, 255:red, 74; green, 144; blue, 226 }  ,draw opacity=1 ][fill={rgb, 255:red, 255; green, 255; blue, 255 }  ,fill opacity=1 ][line width=1]  (15.62,19.02) .. controls (17.37,17.38) and (17.47,14.63) .. (15.83,12.87) .. controls (14.19,11.11) and (11.43,11.01) .. (9.67,12.66) .. controls (7.92,14.3) and (7.82,17.05) .. (9.46,18.81) .. controls (11.1,20.57) and (13.86,20.66) .. (15.62,19.02) -- cycle ;

\end{tikzpicture}}\hspace{-3pt}$) and one of its leaves ($\text{\begin{tikzpicture}[x=0.75pt,y=0.75pt,yscale=-.5,xscale=.5]

\draw  [color={rgb, 255:red, 0; green, 0; blue, 0 }  ,draw opacity=1 ][fill={rgb, 255:red, 255; green, 255; blue, 255 }  ,fill opacity=1 ][line width=1]  (15.62,19.02) .. controls (17.37,17.38) and (17.47,14.63) .. (15.83,12.87) .. controls (14.19,11.11) and (11.43,11.01) .. (9.67,12.66) .. controls (7.92,14.3) and (7.82,17.05) .. (9.46,18.81) .. controls (11.1,20.57) and (13.86,20.66) .. (15.62,19.02) -- cycle ;

\end{tikzpicture}}\hspace{-3pt}$). The transpose of $\mathbf{J}_{L,\mathbf{x}}$ can thus be obtained from the same expansion by taking $\mathbf{W}_{\eta^\top}$ instead of $\mathbf{W}_{\eta}$ (noting that $\varphi_\eta$ and $s(\eta)$ are invariant under $\eta\mapsto \eta^\top$). 
We can thus write
\begin{equation}
    \mathbf{J}_{L,\bx}\mathbf{J}_{L,\bx}^{\top} = \sum_{\eta_1, \eta_2 \in (\partial_\mathbf{x} \bT_{L}(\bx))^*}
    \frac{\varphi_{\eta_1}\varphi_{\eta_2}}{s(\eta_1)s(\eta_2)}\mathbf{W}_{\eta_1\circ\eta_2^{\top}}
\end{equation}
which by linearity of the trace gives
\begin{equation}\label{eq:k=1trace}
    \mathrm{Tr}(\mathbf{J}_{L,\bx}\mathbf{J}_{L,\bx}^{\top}) = \sum_{\eta_1, \eta_2 \in (\partial_\mathbf{x} \bT_{L}(\bx))^*}
    \frac{\varphi_{\eta_1}\varphi_{\eta_2}}{s(\eta_1)s(\eta_2)}\mathbf{W}_{\mathrm{Tr}(\eta_1\circ\eta_2^{\top})}.
\end{equation}

\begin{figure}[t]
    \centering
    \input{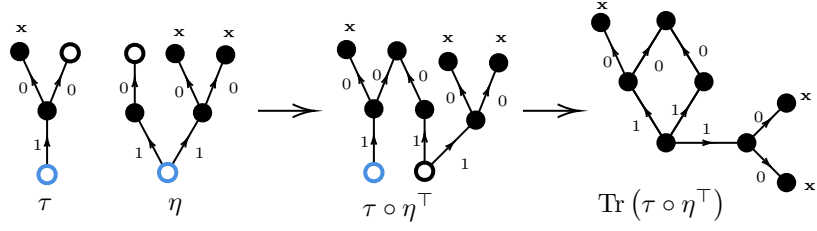}
    \caption{Operator graphs arising in the expansion of $\mathrm{Tr}(\mathbf{J}_{2,\mathbf{x}}\mathbf{J}_{2,\mathbf{x}}^\top)$.}
    \label{fig:decoratedcycle_base}
\end{figure}

Similarly, for $k=2$ we can repeat the above procedure to obtain
\begin{equation}\label{eq:k=2trace}
    \sum_{\eta^{(1)}_1, \eta^{(1)}_2, \eta^{(2)}_1, \eta^{(2)}_2 \in (\partial_\mathbf{x} \bT_{L}(\bx))^*}
    \left( 
    \prod_{i=1}^2\prod_{j=1}^2\frac{\varphi_{\eta^{(i)}_j}}{s(\eta^{(i)}_j)}
    \right) \mathbf{W}_{\mathrm{Tr}\left(\eta^{(1)}_1\circ (\eta^{(1)}_2)^{\top} \circ \eta^{(2)}_1\circ (\eta^{(2)}_2)^{\top}\right)},
\end{equation}
and the case for general $k$ case follows analogously, with the final expression involving product graphs of the form $\mathrm{Tr}\left(\eta^{(1)}_1\circ (\eta^{(1)}_2)^{\top} \circ \cdots \circ \eta^{(k)}_1\circ (\eta^{(k)}_2)^{\top}\right)$. 
This will be discussed in detail in Section \ref{sec:jacobian}, and is depicted in Figure \ref{fig:decoratedcycleconcrete}.

\subsubsection{Neural tangent kernel} 
The same procedure can be used to compute the NTK $\Theta_{L}(\bx,\by)\in \bR^{N_{L+1} \times N_{L+1}} $, which we recall is the matrix representation of
\begin{equation} 
    \sum_{\ell=0}^{L} \lambda_{\ell}^2 (\mathrm{d}\Phi_{L}(\bx))_{W_{\ell}}\circ(\mathrm{d}\Phi_{L}(\by))_{W_{\ell}}^{\top},
\end{equation}
for layer-wise learning rates $(\lambda_{\ell})_\ell \in \bR$ (see Equation \eqref{def:NTK}).

To find an expansion for $\Theta_L$, let $\partial_\ell \mathbb{T}_L(\mathbf{x})$ denote the space of trees which can be obtained from trees in $\mathbb{T}_L(\mathbf{x})$ by freeing exactly one of its height $\ell$ edges (i.e. one edge with input $W_\ell$), making it an in-edge. We can then write
\begin{equation}
    (\mathrm{d}\Phi_{L}(\bx))_{W_{\ell}} = \sum_{\tau \in \partial_{\ell}\bT_{L}(\mathbf{x})} \frac{\varphi_{\tau}}{s(\tau)}\mathbf{W}_{\tau},
\end{equation}
where $\mathbf{W}_{\tau}: \bR^{N_{\ell+1} \times N_{\ell}} \to \bR^{N_{L+1}}$, noting that the $s(\tau)$ take care of any overcounting. Arguing as we did above, we then find that 
\begin{equation}\label{eqn:NTK_matrix_components}
    [(\mathrm{d}\Phi_{L}(\bx))_{W_{\ell}}\circ(\mathrm{d}\Phi_{L}(\by))_{W_{\ell}}^{\top}]_{ij}
    = \sum_{\tau \in \partial_{\ell}\bT_{L,i}(\mathbf{x})}\sum_{\eta \in \partial_{\ell}\bT_{L,j}(\mathbf{y})}
    \frac{\varphi_{\tau}\varphi_{\eta}}
    {s(\tau)s(\eta)}
    \mathbf{W}_{\tau\circ\eta^{\top}}.
\end{equation}

\section{Genus expansion and limit theorems}
\label{sec:genus}
In this section, we develop tools to compute quantities of the form $\E\{\bW_G\}$, where $G$ is a product graph whose edge inputs are matrices with Gaussian entries. We will extend our methods to non-Gaussian and sparse inputs in Section \ref{sec:non_gaussian}. These results are all consequences of the classical Wick's Theorem, which we now recall. 

For any set $S$ of even size $k$, define a \emph{pairing} $\phi$ of $S$ to be a partition of $S$ into pairs $S_j=\{\phi_{j,1},\phi_{j,2}\}$ for $1\leq j\leq k/2$. Then Wick's theorem reduces the task of computing expectations of a product of Gaussian random variables to that of computing pairwise covariances.

\begin{theorem}[Wick's theorem] 
\label{thm:Wicks_theo}
Let $(Z_i)_{i\geq 1}$ be a real centered Gaussian vector.
Then for any sequence of integers $i_1,...,i_k$, 
\[
\mathbb{E}\Big\{\prod_{j=1}^k Z_{i_j}\Big\} = \sum_{\phi} \prod_{j=1}^{k/2}\mathbb{E}\big\{Z_{i_{\phi_{j,1}}}Z_{i_{\phi_{j,2}}}\big\}.
\]
where the sum is taken over all pairings of $\{1,...,k\}$.
\end{theorem}

The arguments below are inspired by those of Dubach and Peled \cite{Dubach2021}, albeit with substantial modifications to accommodate our more general setting.

\subsection{Wick expansion for product graphs with random inputs} Throughout this section, we will let $\mathcal{W}$ denote the following sequence of random matrices. 
\begin{definition}
Let $\mathcal{W}:=(W_i)_{i\in \mathbb{N}}$, where $W_i\in \bR^{N_{r_i}\times N_{c_i}}$ are independent Gaussian random matrices with i.i.d. entries distributed as $\mathcal{N}_{\bR}(0,\sigma_i^2)$.
\end{definition}

\begin{assumption}\label{assumption1} We assume that any graph $G$ considered throughout the section is a product graph whose inputs are either deterministic or matrices in $\mathcal{W}$. 
\end{assumption}

For any product graph $G$ satisfying Assumption \ref{assumption1} and edge $e\in E$, we will say that $e$ is an ``$\ell$-edge" if its input is fixed to $W_\ell\in \mathcal{W}$.
For simplicity, we label such edges by $\ell$ instead of $W_\ell$ when depicting the graph, as the latter can quickly become cumbersome.

We will consider pairings of edges of $G$ whose inputs are matrices in $\mathcal{W}$, denoting the set of such edges by $E_\mathcal{W}$. 
Such a pairing is said to be \emph{admissible} if it only pairs $\ell-$edges with $\ell-$edges.

\begin{definition}[Admissible pairing] 
\label{def:admissible_pairings}
    Let $G$ be a product graph, and $E_\mathcal{W}=\{e\in E: \mathbf{X}_e = W_\ell\text{ for some $\ell(e):=\ell \in \N$}\}$.
    Then a pairing $\phi$ of $E_\mathcal{W}$ is said to be \emph{admissible} if any two paired edges $\{e,e'\} \in \phi$ satisfy $\ell(e)=\ell(e')$. 
    We denote the set of all admissible pairings of $E_\mathcal{W}$ by $\mathcal{P}(G)$.
\end{definition}

\begin{figure}[t]
    \centering
    \vspace{-5pt}
    \scalebox{0.80}{\input{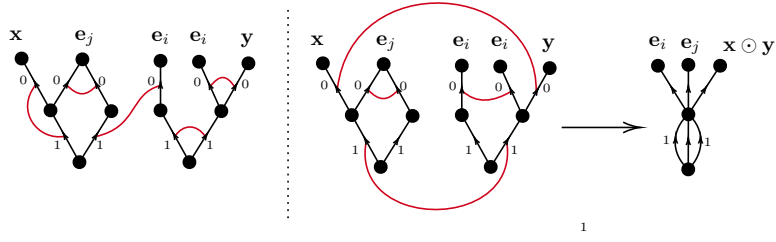}}
    \caption{Left: A pairing which is \emph{not} admissible (in red), since a $1$-edge is paired to a $0$-edge. Right: An admissible pairing (in red) and the corresponding $G_{\phi}$ (note that not all edges are necessarily paired).}
    \label{fig:admissible_partition}
\end{figure}

For any product graph $G$ and  admissible pairing $\phi\in \mathcal{P}(G)$, let $G_{\phi}$ be the product graph constructed by identifying the edges of $G$ paired by $\phi$ coherently with their orientation (see, e.g., Figure \ref{fig:admissible_partition}). 
The inputs of the cells in $G_\phi$ are assigned as follows:
\begin{itemize}
    \item When identifying two edges fixed to  $W_\ell$, the resulting edge will be fixed to $\E[W_\ell\odot W_{\ell}] = \sigma_{\ell}^2 \mathbf{I}$.
    \item When identifying two vertices $u,v$ (e.g. by pairing edges which have $u$ and $v$ as their respective endpoints), the resulting vertex will have input $\mathbf{x}_u\odot \mathbf{x}_v$. 
\end{itemize} 
Examples of pairings which are and aren't admissible are depicted in Figure \ref{fig:admissible_partition}, along with the graph $G_\phi$ that arises when admissibility holds. 

In the event that the graph $G$ is not connected, we write it as a disjoint union $G_{\phi} = \sqcup_{i=1}^{c(G_{\phi})} (G_i)_{\phi}$ of its connected components $(G_i)_{\phi}$, where $c(G_{\phi})$ denotes the number of connected components of $G_\phi$.

\begin{theorem}[Wick expansion]\label{thm:wickexpansion}
 Let $G$ be a (possibly disconnected) product graph satisfying Assumption \ref{assumption1}. Then
    \begin{equation}\label{eq:genusexpansion}
        \E \left\{\bW_G \right\}=  \sum_{\phi \in \cP(G)} \bW_{G_{\phi}}.
    \end{equation}
where $\phi$ runs over all of $G$'s \emph{admissible} pairings.
\end{theorem}

\begin{proof}
    Let $G=(G,\fd, \{\mathbf{X}\}_{c\in C})$ be a product graph and consider the set of indexations $\mathcal{I}^G := \prod_{v \in V(G)} [N_{v}]$. 
    Recall that for each $e=(u,v)\in E$, $i_e=i_ui_v$. 
    Then for any $\mathbf{i}_V\in \mathcal{I}^G$,
    \[
    \E\Big\{\prod_{c \in C} [\mathbf{X}_c]_{i_c}\Big\} 
        = \prod_{c \in C\setminus E_{\mathcal{W}}} [\mathbf{X}_c]_{i_c}
        \E\Big\{\prod_{e \in E_{\mathcal{W}}} [\mathbf{X}_e]_{i_e}\Big\}.
    \]
    By Wick's theorem, we have
     \begin{align*}
        \E\Big\{\prod_{e \in E_{\mathcal{W}}} [\mathbf{X}_e]_{i_e}\Big\} =& 
         \sum_{\phi} \prod_{\{e,e'\} \in \phi } \E\Big\{ [\mathbf{X}_e]_{i_e} [\mathbf{X}_{e'}]_{i_{e'}}\Big\}
    \end{align*}
    where the sum is over all pairings $\phi$ of $E_{\mathcal{W}}$. Note that if $\phi\notin \cP(G)$ then 
    $$\prod_{\{e,e'\} \in \phi } \E\Big\{ [\mathbf{X}_e]_{i_e} [\mathbf{X}_{e'}]_{i_{e'}}\Big\} = 0$$ 
    and we can therefore write
    \begin{align}\label{eq:W_Glemma}
        \E\{\mathbf{W}_G\} =& 
        \sum_{\mathbf{i}_V} 
        \bigg(
        \prod_{c \in C\setminus E_{\mathcal{W}}} [\mathbf{X}_c]_{i_c}
        \bigg)
        \sum_{\phi \in \cP(G)} \prod_{\{e,e'\} \in \phi } 
        \E\Big\{ [\mathbf{X}_e]_{i_e} [\mathbf{X}_{e'}]_{i_{e'}}\Big\}
        \nonumber \\ =&
        \sum_{\phi \in \cP(G)} \sum_{\mathbf{i}_V} 
        \bigg(
        \prod_{c \in C\setminus E_{\mathcal{W}}} [\mathbf{X}_c]_{i_c}
        \bigg)
        \prod_{\{e,e'\} \in \phi } \sigma_{\ell(e)}^{2} \delta_{i_ei_{e'}},
    \end{align}
    where $\delta_{ij} := \mathbf{1}(i=j)$ is the Kronecker delta function.
 
    For any vertices $u,v\in V$, we write $u\sim_\phi v$ if they are identified after identifying edges according to $\phi$. Similarly, we write $e\sim_\phi e'$ if $\{e,e'\}\in \phi$.
    If we let $V_\mathcal{W}$ be the set of vertices adjacent to edges in $E_\mathcal{W}$, we then have canonical bijections
    \[
        V(G_\phi)\simeq V_\mathcal{W}^c \sqcup (V_\mathcal{W}/\sim_\phi),\quad E(G_\phi)\simeq E^c_\mathcal{W} \sqcup (E_\mathcal{W}/\sim_\phi)\,
    \]
    Note that each element in $\varepsilon\in E_\mathcal{W}/\sim_\phi$ corresponds to a block $\varphi(\varepsilon)\in \phi$, while each element of $\nu\in V_\mathcal{W}/\sim_\phi$ corresponds to a subset $\varphi(\nu)\subset V$ (these are the pre-images under the quotient maps).  
     Then for any $\phi\in \mathcal{P}(G)$ we can write the product of $G_{\phi}$ in terms of indexations of $G$ as
    \begin{align*}
        \bW_{G_{\phi}} =& 
        \sum_{\mathbf{i}\in \mathcal{I}^{G_{\phi}}} 
        \prod_{c \in E_{\mathcal{W}}^c \cup V_{\mathcal{W}}^c} [\mathbf{X}_c]_{i_c}
        \bigg( 
        \prod_{\varepsilon\in (E_\mathcal{W}/\sim_\phi)} \sigma_{\ell(\varepsilon)}^{2} 
        \bigg) 
        \bigg(
        \prod_{\nu \in (V_\mathcal{W}/\sim_\phi)} \prod_{u \in \varphi(\nu)}[\mathbf{X}_u]_{i_\nu}
        \bigg)
        \\
        =& 
        \sum_{\mathbf{i}\in \mathcal{I}^{G}} 
        \prod_{c \in E_{\mathcal{W}}^c \cup V_{\mathcal{W}}^c} [\mathbf{X}_c]_{i_c}
        \bigg( 
        \prod_{\{e,e'\} \in \varphi(\epsilon) } \sigma_{\ell(e)}^{2}
        \delta_{i_ei_{e'}}
        \bigg) 
        \bigg(
        \prod_{\nu \in (V_\mathcal{W}/\sim_\phi)} 
        \delta_{\mathbf{i}|_\nu}
        \prod_{u \in \varphi(\nu)}[\mathbf{X}_u]_{i_u} 
        \bigg)
    \end{align*}
    Here $\delta_{\mathbf{i}|_\nu} := \mathbf{1}\left(\forall u,u' \in \varphi(\nu), ~i_u=i_{u'}\right)$ and $\sigma_{\ell(\varepsilon)} = \sigma_{\ell(e)}$ for any choice of $e \in \varphi(\varepsilon)$, since they all yield the same $\sigma_{\ell(e)}$.
     By definition, there eixts a $\nu \in (V_\mathcal{W}/\sim_\phi)$ such that $u,u' \in \varphi(\nu)$ if and only if we have $\{e,e'\}\in\phi$ such that $u,u'$ are either the respective heads or or respective tails of $e,e'$, it follows that
     \[
     \prod_{\{e,e'\} \in \varphi(\epsilon)} \delta_{i_ei_{e'}} =
     \prod_{\nu \in (V_\mathcal{W}/\sim_\phi)} \delta_{\mathbf{i}|_\nu}.
     \]
     Using this, the claim then follows from the equality
    \[
    \prod_{c \in C\setminus E_{\mathcal{W}}} [\mathbf{X}_c]_{i_c} = 
 \prod_{c \in E_{\mathcal{W}}^c \cup V_{\mathcal{W}}^c} [\mathbf{X}_c]_{i_c} 
    \prod_{\nu \in (V_\mathcal{W}/\sim_\phi)} 
    \prod_{u \in \varphi(\nu)}[\mathbf{X}_u]_{i_u}.
    \]
\end{proof}

We now derive a centered version of Theorem \ref{thm:wickexpansion} (Proposition \ref{prop:mixedmoments_}), which we will need to derive limit theorems for $\mathbf{W}_{G}$.
To that end, we need to single out the admissible pairings which give rise to leading and subleading order terms in the limit. The nomenclature below is borrowed from \cite{Dubach2021}, noting that the following is essentially their  Definition 2.3.

\begin{definition}
\label{def:various_partitions}
    Let $G = \sqcup_{i=1}^{c(G)} G_{i}$ be a product graph, given as a disjoint union of its connected components, and fix an admissible pairing $\phi \in \cP(G)$ of its edges.
    \begin{enumerate}
        \item $\phi$ is said to be \emph{atomic} if and only if for every $1\leq i\leq c(G)$, the edges in $E(G_i)$ are paired between themselves and
        $(G_i)_{\phi|_{G_i}}$ is a tree, where $\phi|_{G_i}$ denotes the pairing induced by $\phi$ on the edges of $G_i$.
        \item $\phi$ is \emph{bi-atomic} if and only if for every $1\leq i\leq c(G)$, there exists a $j\neq i$ such that $\phi$ pairs edges of $E(G_i)\sqcup E(G_j)$ between themselves and $(G_i \sqcup G_j)_{\phi|_{G_i \sqcup G_j}}$ is a tree.
        \item $\phi$ is \emph{atom-free} if and only if there is no $G_i$ for which $\phi$ pairs the edges in $E(G_i)$ between themselves and $(G_i)_\phi$ is a tree.
    \end{enumerate}
\end{definition}
We denote the set of (admissible) atomic, bi-atomic and atom-free pairings of $G$ by $\mathcal{P}_{\mathrm{A}}(G)$, 
$\mathcal{P}_{\mathrm{B}}(G)$, and $\mathcal{P}_{\mathrm{AF}}(G)$  respectively. 
With these definitions, we can now state and prove the following centered version of Theorem \ref{thm:wickexpansion}.

\begin{proposition}\label{prop:mixedmoments_}
Let $G=\sqcup_{i=1}^{c(G)}G_i$ be a product graph satisfying Assumption \ref{assumption1}, with disjoint connected components $G_i$. Then
\begin{equation}\label{eqn:lemma:fix_vertices_new}
     \E\bigg\{
     \prod_{i=1}^{c(G)} \Big( 
     \bW_{G_i} - 
    \sum_{\psi \in \mathcal{P}_{\mathrm{A}}(G_{i})}
      \bW_{(G_i)_\psi}
     \Big)
     \bigg\} = 
     \sum_{\phi \in \mathcal{P}_{\mathrm{AF}}(G)} \bW_{G_\phi}
\end{equation}
\end{proposition}

\begin{proof}
Let $[c(G)]:=\{1,...,c(G)\}$, and for any $T \subseteq [c(G)]$ define $G_{T} := \sqcup_{i \in T} G_i$. Then by expanding the expectation in (\ref{eqn:lemma:fix_vertices_new}), we get
    \begin{align*}
        \mathbb{E}\bigg\{
        \prod_i &\Big(\bW_{G_i} -
        \sum_{\psi \in \mathcal{P}_{\mathrm{A}}(G_{i})}
         \bW_{(G_i)_\psi}\Big)\bigg\}
         \\
        &=
        \sum_{T \subseteq [c(G)]} (-1)^{|T^c|} \E\left\{ \prod_{i \in T} \bW_{G_i} \right\} 
        \left( \prod_{i \notin T} 
       \sum_{\psi \in \mathcal{P}_{\mathrm{A}}(G_{i})}
        \bW_{(G_i)_{\psi}}
        \right)
    \end{align*}
 Using Theorem \ref{thm:wickexpansion}, we have
\[
\E\bigg\{\prod_{i \in T} \bW_{G_i} \bigg\} = \sum_{\phi \in \mathcal{P}(G_T)}
     \bW_{(G_T)_{\phi}}
\]
and in turn
\begin{align*}
    &=  
\sum_{T \subseteq [c(G)]} \sum_{\phi \in \cP(G_T)}
(-1)^{|T^c|}
\bW_{(G_T)_{\phi}}
\prod_{i \notin T} 
       \sum_{\psi \in \mathcal{P}_{\mathrm{A}}(G_{i})}
     \bW_{(G_i)_{\psi}}  
     \\
     &= 
        \sum_{T \subseteq [c(G)]} 
        (-1)^{|T^c|}
        \sum_{\phi \in \cP(G_T) \times (\prod_{i \in T^c} \mathcal{P}_{\mathrm{A}}(G_i))}
        \bW_{G_{\phi}}
\end{align*}
and the conclusion follows from the inclusion-exclusion principle since 
\[
\mathcal{P}_{\mathrm{AF}}(G) = \bigcap_{i=1}^{c(G)} (\mathcal{P}(G_{[c(G)]\setminus \{i\}}) \times \mathcal{P}_{\mathrm{A}}(G_i))^c = \mathcal{P}(G) \setminus \bigcup_{i=1}^{c(G)} \big(\mathcal{P}(G_{[c(G)]\setminus \{i\}}) \times \mathcal{P}_{\mathrm{A}}(G_i)\big)
\]
and 
\[
\bigcap_{i \in T^c} \big( \mathcal{P}(G_{[c(G)]\setminus \{i\}}) \times \mathcal{P}_{\mathrm{A}}(G_i)\big) = \cP(G_T) \times\prod_{i \in T^c} \mathcal{P}_{\mathrm{A}}(G_i).
\]
\end{proof}

More generally, the same proof gives the following.
\begin{proposition}
    Let $G=\sqcup_{i=1}^{c(G)}G_i$ be a product graph satisfying Assumption \ref{assumption1}, with disjoint connected components $G_i$. 
    Given any sequence $\cP_i \subseteq \cP(G_i)$ and setting $\cP_{\star} := \bigcap_i [\cP(G_{[c(G)]\setminus\{i\}}) \sqcup \cP_i]^c$ one has
\begin{equation}
     \E\bigg\{
     \prod_{i=1}^{c(G)} \Big( 
     \bW_{G_i} - 
    \sum_{\psi \in \cP_i}
      \bW_{(G_i)_\psi}
     \Big)
     \bigg\} = 
     \sum_{\phi \in \cP_{\star}} \bW_{G_\phi}.
\end{equation}
\end{proposition}

\subsection{Genus expansion}
\label{sec:genus_actual}

We now study a particular setting in which it will be possible to explicitly express $\mathbf{W}_{G_{\phi}}$ for each $\phi$. Recalling that $\mathcal{W}$ is the sequence of random matrices serving as inputs to edges in $E_\mathcal{W}(G)\subset E(G)$, we assume the following.

\begin{assumption}\label{assumption:genus_graph}
      $\mathcal{W}$ is such that all $N_{c_i} = N_{r_i} = N$ for some integer $N>0$, and $G$ is such that $\fd(v)=N$, $\mathbf{X}_v = \mathbf{1}_N$ for all $v\in V$ and $\mathbf{X}_e = \mathbf{I}$ for all $e \in E\setminus E_{\mathcal{W}}(G)$.
\end{assumption}

Under Assumption \ref{assumption:genus_graph}, the definition of $\bW_{G_\phi}$ then directly yields
\begin{equation}\label{eq:allones}
    \bW_{G_\phi} = \sigma_G N^{|V(G_\phi)|}, \quad \sigma_G = \prod_{e\in E_{\mathcal{W}}}\sigma_{\ell(e)},
\end{equation}
given that all the edges of $G_\phi$ are labeled by $\mathbf{I}_{N \times N}$.
The order of $\bW_{G_\phi}$ is thus determined by the number of vertices in $G_\phi$. 

Instead of counting the number of vertices $|V(G_\phi)|$ directly, it will be simpler to associate a surface to $G_\phi$ and use its Euler characteristic to express $|V(G_\phi)|$ in terms of this surface's genus, number of edges and connected components. This forms the basis of the \textit{genus expansion} technique.

To use it, we will need some basic notions from topological graph theory (see \cite{GrossTucker} for a more in-depth introduction). Any connected graph $G=(V,E)$ can be viewed as a topological space, by viewing vertices as distinct points {and edges as subspaces homeomorphic to $[0,1]$ joining their ends, such that edges only meet at their endpoints.
An \textit{embedding} $\rho$ of $G$ on some orientable surface $S$ is then defined as a homeomorphism from $G$ (viewed as a space) to $S$}, and is said to be \textit{cellular} if the resulting faces are all homeomorphic to the open disk. For our purposes, all embeddings are assumed to be cellular.

The connected components of $S\setminus\rho(G)$
are referred to as the \textit{faces} of the embedding, and we denote the total number of faces by $f(G:S)$. Lastly, we denote the genus of $S$ is denoted by $g(S)$. For any cellular embedding, we have the following formula.
\begin{proposition}[Euler Characteristic formula] Let $\rho$ be a (cellular) embedding of a connected $G=(V,E)$ on an orientable surface $S$. Then we have
\[
    |V|-|E|+f(G:S)=2-2g(S).
\]
\end{proposition}

For instance, if $G$ is a planar graph with $F$ faces, then its drawing on the plane is a cellular embedding and this reduces to the famous identity $|V|-|E|+F=2$. 

\noindent Further, it is well known that all orientable surfaces of genus $g$ are equivalent up to homeomorphism (as a representative of their equivalence class, one can for instance take the surface obtained by adding $g$ ``handles" to a $2$-sphere, an idea going back to \cite{Brahana}). For any $g\geq 0$, it thus makes sense to talk about \textit{the} surface of genus $g$, which we henceforth denote by $S_g$. 

Lastly, we can define embeddings of disconnected graphs by considering each connected component separately. Letting $c(G)$ denote the number of components and $G=\sqcup_{i=1}^{c(G)} G_i$, an embedding of $G$ is given by embedding each $G_i$ on an oriented surface $S_i$. We can also extend the definitions of $f$ and $g$ by writing  $S=\sqcup_{i=1}^{c(G)}S_i$,\[
    g(S):=\sum_{i=1}^{c(G)}g(S_i),\quad \text{ and }f(G:S):=\sum_{i=1}^{c(G)}f(G_i:S_i),
\]
and we note that the Euler characteristic formula still holds with these definitions.

\bigskip

Going back the setting of Assumption \ref{assumption:genus_graph}, let $G$ be such a product graph and $\phi\in \mathcal{P}(G)$ be an admissible pairing of its edges. Recalling that $G_\phi$ is the graph obtained from $G$ after identifying edges paired by $\phi$, we let
\begin{equation}\label{eq:sphi}
        S_\phi := \min_{g\geq 0} \{S_g:\ \text{there exists an embedding } \rho:G_\phi\to S_g\}.
\end{equation}
if $G$ is connected. Otherwise, let $(S_i)_{\phi}$ be the surface obtained this way from each connected component $G_i$ of $G$, and define $S_\phi = \sqcup_{i=1}^{c(G)} (S_i)_{\phi}$.
\begin{remark}
    This minimum in (\ref{eq:sphi}) is always attained, and the value of $g$ for which $S_g=S_\phi$ is often referred to as the \textit{minimum orientable genus} of $G$ in the literature (see \cite{GrossTucker}).  
\end{remark} 
We do not dwell on the details of these concepts and definitions, as we essentially only care about the case when $G_\phi$ is a tree in our arguments. In that case, $g(S_\phi)=0$ and we concern ourselves with $f(G_\phi:S_\phi)$ and $c(G_\phi)$ instead.

Note also that we obtained $S_\phi$ from $G_\phi$ in a different way than in \cite{Dubach2021}: in their case, $G$ is always a cycle, which naturally gives rise to a closed surface after identifying edges. No such construction seems to exist in our case.

\bigskip

For any $G$ satisfying Assumption \ref{assumption:genus_graph} and $\phi \in \cP(G)$ we note that $G_{\phi}$ always has the same number of edges, which we denote by
\begin{equation}
        \check{e}(G) := |E(G_{\phi})| = |E_\mathcal{W}(G)|/2+|E\setminus E_{\mathcal{W}}(G)|,
\end{equation}
since the random edges are paired up by $\phi$ and the rest are left untouched. Theorem \ref{thm:wickexpansion} then yields the following.
\begin{corollary}[Genus expansion]
\label{thm:wickexpansion_ones}
Let $G$ be a product graph satisfying Assumptions \ref{assumption1} and \ref{assumption:genus_graph}.
Then for any $\phi\in \mathcal{P}(G)$,
\begin{equation}
    \bW_{G_\phi} = \sigma_{G}N^{2(c(G_\phi)-g(S_\phi))-f(G_\phi:S_\phi)+\check{e}(G)},
\end{equation}
and it follows that
\begin{equation}
    \E \left\{ \bW_G \right\} = \sigma_G N^{\check{e}(G)} \sum_{\phi \in \cP(G)} N^{2(c(G_\phi)-g(S_\phi))-f(G_\phi:S_{\phi})}.
\end{equation}
\end{corollary}

\begin{proof}
This follows directly from Equation (\ref{eq:allones}) and the Euler characteristic formula, which gives
\begin{equation}
|V(G_{\phi})| =  2(c(G_{\phi})-g(S_{\phi})) + |E(G_\phi)| - f(G_{\phi}: S_{\phi}).
\end{equation}
\end{proof}

Let us demonstrate how this result can be used to study words of random matrices, as defined in \cite{Dubach2021}. As explained in the following example, this amounts to computing the values of polygonal product graphs. 

\begin{example}
    Under Assumption \ref{assumption:genus_graph}, let $w$ be a word of matrices in $\mathcal{W}$, say $w=W_1W_4^TW_2^TW_5W_3^T$. Then taking $G$ to be the product graph in Figure \ref{fig:polygon-word} below gives $\mathbf{W}_G=\mathrm{Tr}(w)$.
    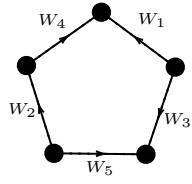
\begin{figure}[t]
        \centering
        \tikzset{every picture/.style={line width=0.75pt}} 

\begin{tikzpicture}[x=0.75pt,y=0.75pt,yscale=-1,xscale=1]

\draw  [dash pattern={on 0.84pt off 2.51pt}] (110.98,6155.34) -- (148.05,6182.87) -- (133.33,6226.63) -- (87.15,6226.15) -- (73.35,6182.09) -- cycle ;
\draw  [color={rgb, 255:red, 0; green, 0; blue, 0 }  ,draw opacity=1 ][fill={rgb, 255:red, 0; green, 0; blue, 0 }  ,fill opacity=1 ] (113.96,6158.52) .. controls (115.71,6156.88) and (115.81,6154.13) .. (114.17,6152.37) .. controls (112.53,6150.61) and (109.77,6150.51) .. (108.01,6152.16) .. controls (106.26,6153.8) and (106.16,6156.55) .. (107.8,6158.31) .. controls (109.44,6160.07) and (112.2,6160.16) .. (113.96,6158.52) -- cycle ;
\draw  [color={rgb, 255:red, 0; green, 0; blue, 0 }  ,draw opacity=1 ][fill={rgb, 255:red, 0; green, 0; blue, 0 }  ,fill opacity=1 ] (151.03,6186.05) .. controls (152.78,6184.41) and (152.88,6181.66) .. (151.24,6179.9) .. controls (149.6,6178.14) and (146.84,6178.04) .. (145.08,6179.69) .. controls (143.32,6181.33) and (143.23,6184.08) .. (144.87,6185.84) .. controls (146.51,6187.6) and (149.27,6187.69) .. (151.03,6186.05) -- cycle ;
\draw  [color={rgb, 255:red, 0; green, 0; blue, 0 }  ,draw opacity=1 ][fill={rgb, 255:red, 0; green, 0; blue, 0 }  ,fill opacity=1 ] (136.3,6229.81) .. controls (138.06,6228.17) and (138.15,6225.42) .. (136.51,6223.66) .. controls (134.87,6221.9) and (132.11,6221.81) .. (130.36,6223.45) .. controls (128.6,6225.09) and (128.5,6227.84) .. (130.14,6229.6) .. controls (131.78,6231.36) and (134.54,6231.45) .. (136.3,6229.81) -- cycle ;
\draw  [color={rgb, 255:red, 0; green, 0; blue, 0 }  ,draw opacity=1 ][fill={rgb, 255:red, 0; green, 0; blue, 0 }  ,fill opacity=1 ] (90.13,6229.33) .. controls (91.88,6227.69) and (91.98,6224.93) .. (90.34,6223.18) .. controls (88.7,6221.42) and (85.94,6221.32) .. (84.18,6222.96) .. controls (82.43,6224.61) and (82.33,6227.36) .. (83.97,6229.12) .. controls (85.61,6230.88) and (88.37,6230.97) .. (90.13,6229.33) -- cycle ;
\draw  [color={rgb, 255:red, 0; green, 0; blue, 0 }  ,draw opacity=1 ][fill={rgb, 255:red, 0; green, 0; blue, 0 }  ,fill opacity=1 ] (76.32,6185.27) .. controls (78.08,6183.63) and (78.17,6180.87) .. (76.53,6179.11) .. controls (74.89,6177.36) and (72.13,6177.26) .. (70.38,6178.9) .. controls (68.62,6180.54) and (68.52,6183.3) .. (70.16,6185.06) .. controls (71.81,6186.82) and (74.56,6186.91) .. (76.32,6185.27) -- cycle ;
\draw    (73.35,6182.09) -- (110.98,6155.34) ;
\draw [shift={(94.12,6167.32)}, rotate = 144.6] [color={rgb, 255:red, 0; green, 0; blue, 0 }  ][line width=0.75]    (4.37,-1.32) .. controls (2.78,-0.56) and (1.32,-0.12) .. (0,0) .. controls (1.32,0.12) and (2.78,0.56) .. (4.37,1.32)   ;
\draw    (148.05,6182.87) -- (110.98,6155.34) ;
\draw [shift={(127.59,6167.67)}, rotate = 36.6] [color={rgb, 255:red, 0; green, 0; blue, 0 }  ][line width=0.75]    (4.37,-1.32) .. controls (2.78,-0.56) and (1.32,-0.12) .. (0,0) .. controls (1.32,0.12) and (2.78,0.56) .. (4.37,1.32)   ;
\draw    (133.33,6226.63) -- (148.05,6182.87) ;
\draw [shift={(139.61,6207.97)}, rotate = 288.6] [color={rgb, 255:red, 0; green, 0; blue, 0 }  ][line width=0.75]    (4.37,-1.32) .. controls (2.78,-0.56) and (1.32,-0.12) .. (0,0) .. controls (1.32,0.12) and (2.78,0.56) .. (4.37,1.32)   ;
\draw    (87.15,6226.15) -- (73.35,6182.09) ;
\draw [shift={(79.53,6201.83)}, rotate = 72.6] [color={rgb, 255:red, 0; green, 0; blue, 0 }  ][line width=0.75]    (4.37,-1.32) .. controls (2.78,-0.56) and (1.32,-0.12) .. (0,0) .. controls (1.32,0.12) and (2.78,0.56) .. (4.37,1.32)   ;
\draw    (87.15,6226.15) -- (133.33,6226.63) ;
\draw [shift={(112.64,6226.41)}, rotate = 180.6] [color={rgb, 255:red, 0; green, 0; blue, 0 }  ][line width=0.75]    (4.37,-1.32) .. controls (2.78,-0.56) and (1.32,-0.12) .. (0,0) .. controls (1.32,0.12) and (2.78,0.56) .. (4.37,1.32)   ;

\draw (127.67,6154.52) node [anchor=north west][inner sep=0.75pt]  [font=\tiny]  {$W_{1}$};
\draw (62.67,6199.52) node [anchor=north west][inner sep=0.75pt]  [font=\tiny]  {$W_{2}$};
\draw (140.69,6204.75) node [anchor=north west][inner sep=0.75pt]  [font=\tiny]  {$W_{3}$};
\draw (77.69,6154.75) node [anchor=north west][inner sep=0.75pt]  [font=\tiny]  {$W_{4}$};
\draw (101.67,6228.52) node [anchor=north west][inner sep=0.75pt]  [font=\tiny]  {$W_{5}$};

\end{tikzpicture}
        \caption{A product graph whose value is the trace of a word of Ginibre random matrices.}
        \label{fig:polygon-word}
    \end{figure}
    Products of such traces (the subject of Theorem 2.2 in \cite{Dubach2021}) can then be evaluated using Corollary \ref{thm:wickexpansion_ones} by considering disjoint unions of such graphs.
\end{example}

In this specific setting, we can also refine our result on centered mixed moments.
\begin{lemma}\label{lemma:mixedmoments_genus}
Let $G$ satisfy Assumptions \ref{assumption1} and \ref{assumption:genus_graph}. Then for every $\phi\in \cP_{\mathrm{AF}}(G)$, we have $$2(c(G_\phi)-g(S_\phi))-f(G_\phi:S_{\phi})\leq  \frac{c(G)}{2}. $$
Further, equality holds if and only if $\phi$ is bi-atomic, giving
\begin{equation*}
    \E\left\{\prod_{i=1}^{c(G)} 
    \left(\bW_{G_i}-\sigma_{G_i} N^{\check{e}(G_i)+1} |\mathcal{P}_\mathrm{A}(G_i)| \right)\right\} = \sigma_G N^{ \check{e}(G) + \frac{c(G)}{2} } \left(|\mathcal{P}_\mathrm{B}(G)|+\mathcal{O}\left(\frac{1}{N}\right) \right).
\end{equation*}
\end{lemma}

\begin{proof}
    Assume that $
    c(G_{\phi}) \leq {c(G)}/{2}
    $.
    Then it suffices to prove that for each connected component $G_i$ of $G$ and its corresponding surface $S_i$, 
    \[
    2(1-g((S_i)_\phi))-f\big((G_i)_\phi:(S_i)_{\phi}\big) \leq 1
    \]
    and that the equality holds if and only if $(G_i)_\phi$ is a tree.

    Since $f\big((G_i)_\phi:(S_i)_{\phi}\big) \geq 1$ and $g\big((S_i)_\phi\big) \geq 0$, we have 
    \[
    2(1-g\big((S_i)_\phi)\big)-f\big((G_i)_\phi:(S_i)_\phi\big) 
    \leq 
    2(1-0)-1 = 1
    \]
    with equality if and only if $f\big((G_i)_\phi : (S_i)_\phi\big) = 1$ and $g\big((S_i)_\phi\big) = 0$ (which is equivalent to $(G_i)_\phi$ being a tree).
    Now if $c(G_{\phi}) > {c(G)}/{2}$ instead, then there must be at least     
    \[
    2\bigg(c(G_{\phi}) - \frac{c(G)}{2}\bigg)
    \]
    components of $G_\phi$ of the form $(G_j)\big|_{\phi}$ for some $1\leq j\leq c(G)$. Letting $k$ be the number of components obtained this way, then we must have that $k \geq 2(c(G_{\phi}) - {c(G)}/{2})$: assuming otherwise, the inequality 
    \[
    c(G) \geq k + 2(c(G_{\phi}) - k) = 2\bigg(c(G_{\phi}) - \frac{k}{2}\bigg)
    \]
    quickly gives a contradiction. Since $\phi$ is atom-free, 
    \[
    f((G_i)_\phi:(S_i)_\phi) = 1 \implies g((S_i)_\phi) > 0
    \]
    for these components, and it follows that
    \[
    2(1-g\big((S_i)_\phi\big)-f\big((G_i)_\phi:(S_i)_\phi\big) \leq 0.
    \]
    As for the remaining components (of which there are strictly less than $c(G)/2$) this quantity is $\leq 1$ as proved above, and the lemma follows.
\end{proof}

\begin{corollary}
Let $G$ be a product graph satisfying Assumptions \ref{assumption1} and \ref{assumption:genus_graph}, and $\lambda_G := ({\sigma_{G} N^{\check{e}(G)+1}})^{-1}$ 
    \begin{equation}
    \E\left\{ 
    \left(\lambda_G \mathbf{W}_{G}- |\mathcal{P}_\mathrm{A}(G)| \right)^2 \right\} = \frac{1}{N} \left(|\mathcal{P}_\mathrm{B}(G)|+\mathcal{O}\left(\frac{1}{N}\right) \right).
\end{equation}
\end{corollary}

Lastly, we note that by combining the previous two results, we can derive the following limit theorem for the joint distribution of product graphs. 
\begin{theorem}\label{thm:joint_gaussian_limit}
    Let $(G_i)_{i\geq 1}$ be a sequence of connected product graphs satisfying Assumptions \ref{assumption1} and \ref{assumption:genus_graph}. Then for any $m\geq 1$,
    \[
        \big((\sigma_{G_i}N^{\check{e}(G_i) + \frac{1}{2}}\big)^{-1}\bW_{G_i} - |\mathcal{P}_\mathrm{A}(G_i)| \sqrt{N})_{i=1}^m \overset{d}{\longrightarrow} \mathbf{Z} = \big(Z_1,...,Z_m\big)^{\top}
    \]
    where the limit is a centered Gaussian vector with covariance function  $C := \E[\mathbf{Z}\mathbf{Z}^{\top}]$
    having entries $[C]_{i,j} = |\mathcal{P}_\mathrm{B}(G_i \sqcup G_j)|$.
\end{theorem}

\begin{proof}
    We use the method of moments. Setting 
    \[
    \tilde Z_i := (\sigma_{G_i}N^{\check{e}(G_i) + \frac{1}{2}})^{-1}\bW_{G_i} - |\mathcal{P}_\mathrm{A}(G_i)| \sqrt{N}
    \]we want to compute the expectations
    \[
    \E\left\{ \prod_{i=1}^m (\tilde Z_i)^{k_i}\right\}
    \]
    for non-negative integers $k_i$.
    We are in the setting of Lemma \ref{lemma:mixedmoments_genus} and the leading term of the expectation is the number of bi-atomic pairings of the induced graph.
    Given a possible pairing of its connected components, for each of its pairs $(\tilde G, \tilde G')$ this contributes a factor \(|\mathcal{P}_\mathrm{A}(\tilde G \sqcup \tilde G')|\) to the total sum. 
    But this is exactly the result of computing 
    \[
    \E\left\{ \prod_{i=1}^m (Z_i)^{k_i}\right\}
    \]
    by applying Wick's theorem. 
\end{proof}

\section{Applications to feed-forward networks}
\label{sec:application_NNs}

Recall that a neural network $\Phi_L$ is defined by
\begin{equation}
     \Phi_0(\bx)=W_0\bx,\quad \Phi_{\ell+1}(\bx)=W_{\ell+1}\varphi_{\ell+1}(\Phi_{\ell}(\bx)),
\end{equation}
for a choice of activation functions $(\varphi_\ell)_{\ell \geq 1}$, layer dimensions $(N_\ell)_{\ell\geq 0}$ and weight matrices $(W_{\ell})_{\ell\geq 0}$ (see Equation (\ref{def:NN})). 
In this section, the weight matrices $W_\ell$ are assumed to have i.i.d. Gaussian entries with standard deviation $\sigma_\ell$. Each of the results will depend on a different choice of  $(N_\ell)_{\ell\geq 0}$ and $(\sigma_\ell)_{\ell}$. Extensions to non-Gaussian/sparse matrices are handled in Sections \ref{sec:sparse} and \ref{sec:non_gaussian}. 

Note that the proofs of some auxiliary lemmas stated here are deferred to Section \ref{sec:aux} to help with readability.


\subsection{Gaussian process limit}\label{sec:GP}
We begin by proving a special case of the Gaussian process limit (Theorem \ref{thm:general_GP}), for $\Phi$ with Gaussian weights. 

\begin{assumption}\label{assumption:GP_lim}
    The input dimension $N_0$ of $\Phi_L$ is fixed, and $N_{\ell}=N$ for all other layers $\ell > 0$. Assume also that $\sigma_0=1$ and $\sigma_{\ell} = N_{\ell}^{-1/2}$  for $\ell\geq 1$,
\end{assumption}

\begin{theorem}\label{thm:GPlimit} 
    Under Assumption \ref{assumption:GP_lim}, for any $M,L \geq 1$ we have 
    \begin{equation}\label{eqn:GP_limit}
        ([\Phi_{L}]_1,...,[\Phi_{L}]_M) \xrightarrow[N \to \infty]{d} \mathcal{GP}(0,K_{L} \otimes \mathrm{Id}_M)
    \end{equation}
    where the right hand side is a Gaussian Process indexed on $\bR^{N_0}$, with diagonal covariance function defined by
        \begin{gather}\label{def:kernel}
        K_{0}(\bx,\by) = \sprod{\bx}{\by}_{\bR^{N_0}}, ~
        K_{\ell+1}(\bx,\by) = \E\left[ \varphi_{\ell+1}(X_{\ell}) \varphi_{\ell+1}(Y_{\ell}) \right] \\
        (X_{\ell},Y_{\ell}) \sim \mathcal{N}\left(0,
        \begin{bmatrix}
            K_{\ell}(\bx,\bx) & K_{\ell}(\bx,\by)
            \\
            K_{\ell}(\by,\bx) & K_{\ell}(\by,\by)
        \end{bmatrix}\right).
    \end{gather}
\end{theorem}

To prove this, we first consider the coordinate-wise tree expansion of $\Phi_{L}$ given in Equation (\ref{eqn:tree_exp_coords}):
\[
[\Phi_{L}(\bx)]_k = \sum_{\eta \in \mathbb{T}_{L,k}(\bx)} \frac{\varphi_{\eta}}{s(\eta)} \bW_{\eta}  \in \bR.
\]
Finding a distributional limit of $\Phi_{L}$ using this expansion thus amounts to finding the limiting distribution of
$$\left(\bW_{\eta} ~|~ \eta \in \cup_{i=1}^M \mathbb{T}_{L,k_i}(\bx_i)\right)$$ 
as $N \to \infty$, for any fixed sequences $(\bx_i \in \bR^{N_0} ~|~ i \in [M])$ and  $(k_i \in [N_{L+1}] ~|~ i \in [M])$.

The mixed moments of the $\bW_{\eta}$ can be computed using Proposition \ref{prop:mixedmoments_}, which yields
\begin{align}\label{eq:mixedmoments}
    \E\bigg\{\prod_{i=1}^{m}\big(\bW_{\eta_i}-\sum_{\phi\in \mathcal{P}_{\mathrm{A}}(\eta_i)}\mathbf{W}_{(\eta_i)_{\phi}}\big)\bigg\}=\sum_{\phi \in \mathcal{P}_{\mathrm{AF}}(\sqcup_i \eta_i)} \mathbf{W}_{(\sqcup_i \eta_i)_{\phi}}.
\end{align}
To identify the leading order terms on the right-hand side, we must therefore identify the pairings $\phi\in \mathcal{P}_{\mathrm{AF}}(\sqcup_i \eta_i)$ which maximize $\mathbf{W}_{(\sqcup_i \eta_i)_\phi}$. 
This is done in the following lemma.
\begin{lemma}\label{lemma:GP_scaling}
    Under Assumption \ref{assumption:GP_lim}, for any $L \geq 0$ and fixed sequences $(\bx_i \in \bR^{N_0} ~|~ i \in [M])$ , $(k_i \in [N_{L+1}] ~|~ i \in [M])$ and $(\eta_i \in \bT_{L,k_i}(\bx_i) ~|~ i \in [M])$, we have
    \begin{equation}
    \forall i \in [M], \quad 
    \mathcal{P}_{\mathrm{A}}(\eta_i) = \emptyset \text{, and thus} 
    \sum_{\psi \in \mathcal{P}_{\mathrm{A}}(\eta_i)} \mathbf{W}_{(\eta_i)_{\psi}}
    = 0,
    \end{equation} 
    \begin{equation}
        \sum_{\phi \in \mathcal{P}_{\mathrm{AF}}(\sqcup_i \eta_i)} \mathbf{W}_{(\sqcup_i \eta_i)_{\phi}}
        = \sigma_{\sqcup_i \eta_i} N^{\tfrac{1}{2}\left( |E(\sqcup_i \eta_i)| - |\mathcal{L}_0(\sqcup_i \eta_i)|\right)}
        \left(\alpha(\sqcup_i \eta_i)
        + \mathcal{O}\Big(\frac{1}{N}\Big)\right),
    \end{equation}
    where $\mathcal{L}_0(\sqcup_i\eta_i)$ denotes the set of \textit{leaves} of $\sqcup_i\eta_i$ which are endpoints of $0$-labeled edges, 
    \[
    \alpha(\sqcup_i \eta_i) := 
    \sum_{\phi \in \cP_{\mathrm{B}}(\sqcup_i \eta_i)} 
    \delta_{\phi}
    \sprod{\bx}{\bx}_{\phi}
    \]
    where $\delta_{\phi}$ is equal to one if roots paired by $\phi$  are fixed to the same $\mathbf{e}_i$ and zero otherwise, and
    \[
    \sprod{\bx}{\bx}_{\phi} := \prod_{\{\bullet_{\bx_i}, \bullet_{\bx_j}\} \in (\mathcal{L}_0(\sqcup_i\eta_i)/\sim_{\phi})} \sprod{\bx_i}{\bx_j}_{\bR^{N_0}}.
    \]
\end{lemma}
\begin{proof}
    See Section \ref{sec:proofs_lem2&3}.
\end{proof}
Here, $\alpha(\sqcup_i\eta_i)$ is essentially a weighted version of the cardinality of $\mathcal{P}_\mathrm{B}(\sqcup_i\eta_i)$.
Applying this lemma to (\ref{eq:mixedmoments}) gives
\begin{equation}\label{eqn:GP_tree_scaling}
     \E\left\{\prod_{i=1}^{m} \bW_{\eta_i}\right\}
     = \sigma_{\sqcup_i \eta_i} N^{\tfrac{1}{2}({|E(\sqcup_i \eta_i)| - |\mathcal{L}_0(\sqcup_i \eta_i)|})}
        \left(
        \alpha(\sqcup_i \eta_i) + \mathcal{O}\Big(\frac{1}{N}\Big)\right).
\end{equation}
\begin{remark}
    To have a well-defined distributional limit for $\Phi_{L}$, the $\sigma_{\ell}$ must be chosen in a way that makes the above quantity $\mathcal{O}(1)$ for all $\eta_i \in \bT_{L,k_i}$. 
This is equivalent to requiring that
$\sigma_{\eta} N^{\tfrac{1}{2}({|E(\eta)| - |\mathcal{L}_0(\eta)}|)} = \mathcal{O}(1)$ for all $L,k \geq 1$ and $\eta \in \mathbb{T}_{L,k}$, which is in turn equivalent to 
\begin{equation}\label{eq:GPlimitinit}
     \sigma_0=\mathcal{O}(1) \text{ and } \sigma_{\ell} = \mathcal{O}(N_{\ell}^{-1/2}) \,\,  \forall \ell \geq 1.
\end{equation}
The canonical choice would then be the one for which $\sigma_{\eta} N^{\tfrac{1}{2}({|E(\eta)| - |\mathcal{L}_0(\eta)}|)} = 1$, which explains our Assumption \ref{assumption:GP_lim} and is known as the ``GP limit" parametrisation (see \cite{Lee2017DeepNN}). 
\end{remark}

\noindent The proof of Theorem \ref{thm:joint_gaussian_limit} (replacing $|\mathcal{P}_{\mathrm{B}}(\sqcup_i\eta_i)|$ with $\alpha(\sqcup_i\eta_i)$) then gives the following limit.
\begin{proposition}\label{lemma:joint_limit_trees}
    Assume that Assumption \ref{assumption:GP_lim} holds. Then for any $M,L \geq 1$  and fixed sequences $(\bx_i \in \bR^{N_0} ~|~ i \in [M])$ , $(k_i \in [N_{L+1}] ~|~ i \in [M])$ and $(\eta_i \in \bT_{L,k_i}(\bx_i) ~|~ i \in [M])$ we have 
    \begin{equation}
        \left(\bW_{\eta_i} ~\bigg|~ i \in [M] \right) 
        \xrightarrow[N \to \infty]{d} 
        \left(Z_{\eta_i} ~\bigg|~ i \in [M] \right) 
        \overset{d}{=}\, \mathcal{GP}(0,C)
    \end{equation}
    where the right-hand side denotes a Gaussian process indexed on $(\eta_i \in \bT_{L,k_i}(\bx_i) ~|~ i \in [M])$ with covariance $C(\eta,\nu) = \alpha(\eta \sqcup \nu)$, $\overset{d}{=}$ denotes equality in distribution and $\alpha$ is defined as in the statement of Lemma \ref{lemma:GP_scaling}.
\end{proposition}
\noindent Lastly, we will need the following decomposition of the covariances $\alpha(\eta\sqcup \nu)$.
\begin{lemma}\label{lemma:alpha_trees}
Let $\alpha$ be as in Lemma \ref{lemma:GP_scaling} and fix $\ell \geq 1$. For any $\tau = [\tau_1 \cdots \tau_n]_{\ell} \in \bT_{\ell}(\bx_1)$ and $\eta = [\eta_1 \cdots \eta_m]_{\ell} \in \bT_{\ell}(\bx_2)$, let $\tau^{(i)}$ denote the tree obtained from $\tau$ by fixing its unique out-vertex to $\mathbf{e}_i$, and similarly for $\eta^{(j)}$. Then we have that
    \[
    \alpha(\tau^{(i)} \sqcup \eta^{(j)}) = \delta_{ij} \sum_{\pi} \prod_{\{\pi_1, \pi_2\} \in \pi} \alpha(\pi_1 \sqcup \pi_2)
    \]
where the sum is over all pairings $\pi$ of the elements of $(\tau_1^{(i)}, \dots, \tau_n^{(i)}, \eta_1^{(j)}, \dots, \eta_m^{(j)})$.
\end{lemma}
\begin{proof}
    See Section \ref{sec:proofs_lem2&3}.
\end{proof}

We are now ready to prove Theorem \ref{thm:GPlimit}.
\begin{proof}[Proof of Theorem \ref{thm:GPlimit}]
    It follows directly from Proposition \ref{lemma:joint_limit_trees} and Equation \eqref{eqn:tree_exp_coords} that 
    \[
        ([\Phi_L]_1,...,[\Phi_L]_{M})\overset{d}{\longrightarrow}\mathcal{GP}(0, K_L \otimes \mathrm{Id}_M)
    \]
    where for any $\ell\geq 0$, $K_{\ell}$ is defined by
    \begin{equation}\label{eqn:GP_lim_tree}
        K_{\ell}(\bx,\by) = 
        \sum_{\eta_{\bx} \in \bT_{\ell,1}(\bx) } \sum_{\eta_{\by} \in \bT_{\ell,1}(\by) }
        \frac{\varphi_{\eta_{\bx}}\varphi_{\eta_{\by}}}{s(\eta_{\bx}) s(\eta_{\by})}\alpha(\eta_{\bx} \sqcup \eta_{\by}) \in \bR.
    \end{equation}
    What's left is to show that this definition of $K_{\ell}$ coincides with the one in the theorem statement. 
    We proceed by induction on $\ell$.
    The base case $\ell=0$ follows from $\bT_{0,i}(\bx) = \{\text{\begin{tikzpicture}[x=0.75pt,y=0.75pt,yscale=-.5,xscale=.5]

\draw    (80.64,65) -- (121.14,65) ;
\draw [shift={(103.29,65)}, rotate = 180] [color={rgb, 255:red, 0; green, 0; blue, 0 }  ][line width=0.75]    (4.37,-1.32) .. controls (2.78,-0.56) and (1.32,-0.12) .. (0,0) .. controls (1.32,0.12) and (2.78,0.56) .. (4.37,1.32)   ;
\draw  [color={rgb, 255:red, 0; green, 0; blue, 0 }  ,draw opacity=1 ][fill={rgb, 255:red, 0; green, 0; blue, 0 }  ,fill opacity=1 ][line width=1.5]  (124.11,68.18) .. controls (125.87,66.54) and (125.97,63.79) .. (124.33,62.03) .. controls (122.68,60.27) and (119.93,60.18) .. (118.17,61.82) .. controls (116.41,63.46) and (116.32,66.21) .. (117.96,67.97) .. controls (119.6,69.73) and (122.36,69.82) .. (124.11,68.18) -- cycle ;
\draw  [color={rgb, 255:red, 0; green, 0; blue, 0 }  ,draw opacity=1 ][fill={rgb, 255:red, 0; green, 0; blue, 0 }  ,fill opacity=1 ][line width=1.5]  (83.62,68.18) .. controls (85.37,66.54) and (85.47,63.79) .. (83.83,62.03) .. controls (82.19,60.27) and (79.43,60.18) .. (77.67,61.82) .. controls (75.92,63.46) and (75.82,66.21) .. (77.46,67.97) .. controls (79.1,69.73) and (81.86,69.82) .. (83.62,68.18) -- cycle ;

\draw (83,44) node [anchor=north west][inner sep=0.75pt]  [font=\tiny,rotate=-0.89]  {$W_{0}$};
\draw (50,57.97) node [anchor=north west][inner sep=0.75pt]  [font=\tiny,rotate=-0.89]  {$\mathbf{e}_{i}$};
\draw (128.79,58.54) node [anchor=north west][inner sep=0.75pt]  [font=\tiny,rotate=-0.89]  {$\mathbf{x}$};

\end{tikzpicture}}\}$ and $\alpha(\text{\begin{tikzpicture}[x=0.75pt,y=0.75pt,yscale=-.5,xscale=.5]

\draw    (80.64,65) -- (121.14,65) ;
\draw [shift={(103.29,65)}, rotate = 180] [color={rgb, 255:red, 0; green, 0; blue, 0 }  ][line width=0.75]    (4.37,-1.32) .. controls (2.78,-0.56) and (1.32,-0.12) .. (0,0) .. controls (1.32,0.12) and (2.78,0.56) .. (4.37,1.32)   ;
\draw  [color={rgb, 255:red, 0; green, 0; blue, 0 }  ,draw opacity=1 ][fill={rgb, 255:red, 0; green, 0; blue, 0 }  ,fill opacity=1 ][line width=1.5]  (124.11,68.18) .. controls (125.87,66.54) and (125.97,63.79) .. (124.33,62.03) .. controls (122.68,60.27) and (119.93,60.18) .. (118.17,61.82) .. controls (116.41,63.46) and (116.32,66.21) .. (117.96,67.97) .. controls (119.6,69.73) and (122.36,69.82) .. (124.11,68.18) -- cycle ;
\draw  [color={rgb, 255:red, 0; green, 0; blue, 0 }  ,draw opacity=1 ][fill={rgb, 255:red, 0; green, 0; blue, 0 }  ,fill opacity=1 ][line width=1.5]  (83.62,68.18) .. controls (85.37,66.54) and (85.47,63.79) .. (83.83,62.03) .. controls (82.19,60.27) and (79.43,60.18) .. (77.67,61.82) .. controls (75.92,63.46) and (75.82,66.21) .. (77.46,67.97) .. controls (79.1,69.73) and (81.86,69.82) .. (83.62,68.18) -- cycle ;

\draw (83,44) node [anchor=north west][inner sep=0.75pt]  [font=\tiny,rotate=-0.89]  {$W_{0}$};
\draw (50,57.97) node [anchor=north west][inner sep=0.75pt]  [font=\tiny,rotate=-0.89]  {$\mathbf{e}_{i}$};
\draw (128.79,58.54) node [anchor=north west][inner sep=0.75pt]  [font=\tiny,rotate=-0.89]  {$\mathbf{x}$};

\end{tikzpicture}} \sqcup \text{\begin{tikzpicture}[x=0.75pt,y=0.75pt,yscale=-.5,xscale=.5]

\draw    (80.64,65) -- (121.14,65) ;
\draw [shift={(103.29,65)}, rotate = 180] [color={rgb, 255:red, 0; green, 0; blue, 0 }  ][line width=0.75]    (4.37,-1.32) .. controls (2.78,-0.56) and (1.32,-0.12) .. (0,0) .. controls (1.32,0.12) and (2.78,0.56) .. (4.37,1.32)   ;
\draw  [color={rgb, 255:red, 0; green, 0; blue, 0 }  ,draw opacity=1 ][fill={rgb, 255:red, 0; green, 0; blue, 0 }  ,fill opacity=1 ][line width=1.5]  (124.11,68.18) .. controls (125.87,66.54) and (125.97,63.79) .. (124.33,62.03) .. controls (122.68,60.27) and (119.93,60.18) .. (118.17,61.82) .. controls (116.41,63.46) and (116.32,66.21) .. (117.96,67.97) .. controls (119.6,69.73) and (122.36,69.82) .. (124.11,68.18) -- cycle ;
\draw  [color={rgb, 255:red, 0; green, 0; blue, 0 }  ,draw opacity=1 ][fill={rgb, 255:red, 0; green, 0; blue, 0 }  ,fill opacity=1 ][line width=1.5]  (83.62,68.18) .. controls (85.37,66.54) and (85.47,63.79) .. (83.83,62.03) .. controls (82.19,60.27) and (79.43,60.18) .. (77.67,61.82) .. controls (75.92,63.46) and (75.82,66.21) .. (77.46,67.97) .. controls (79.1,69.73) and (81.86,69.82) .. (83.62,68.18) -- cycle ;

\draw (83,44) node [anchor=north west][inner sep=0.75pt]  [font=\tiny,rotate=-0.89]  {$W_{0}$};
\draw (50,57.97) node [anchor=north west][inner sep=0.75pt]  [font=\tiny,rotate=-0.89]  {$\mathbf{e}_{i}$};
\draw (128.79,58.54) node [anchor=north west][inner sep=0.75pt]  [font=\tiny,rotate=-0.89]  {$\mathbf{y}$};

\end{tikzpicture}}) = \sprod{\bx}{\by}_{\bR^{N_0}}$.
    
    For the inductive step, we assume that the claim holds for $\ell$ and apply Proposition \ref{app:prop:expectation_tree_expansion} with $\cA_1 = \bT_{\ell,1}(\bx)$, $\cA_2 = \bT_{\ell,1}(\by)$ and $$\lambda(\eta,\nu) = \frac{\varphi_{\eta}\varphi_{\nu}}{s(\eta) s(\nu)}\alpha(\eta \sqcup \nu).$$
    Note that $\bT_{\ell+1,1}(\bx)$ and $\bX_{\cA_1}$ are in bijection via $\eta := [\eta_1 \cdots \eta_m]_{\ell+1} \mapsto \bar\eta := \llbracket \eta_1 \cdots \eta_m \rrbracket$ and under this identification we have 
    \[
    s(\eta) = {\fs(\bar\eta)} \prod_i s(\eta_i),\quad  \varphi_\eta = 
    \varphi_{\ell+1}^{(m)} \prod_i \varphi_{\eta_i}.
    \]
    Lemma \ref{lemma:alpha_trees} thus gives 
    \[
    \alpha(\eta \sqcup \nu) = \sum_{\pi \in \cP(\bar\eta \sqcup \bar\nu)} \prod_{\{\pi_1, \pi_2\} \in \pi} \alpha(\pi_1 \sqcup \pi_2),
    \]
    and this completes the proof since
    \[
    \frac{\varphi_{\eta}\varphi_{\nu}}{s(\eta) s(\nu)} \alpha(\eta \sqcup \nu)  = 
    \frac{\varphi_{\bar\eta}\varphi_{\bar\nu}}{\fs(\bar\eta) \fs(\bar\nu)} 
    \sum_{\pi \in \cP(\bar\eta \sqcup \bar\nu)} \prod_{\{\pi_1, \pi_2\} \in \pi} \lambda(\pi_1, \pi_2).
    \]
\end{proof}

\subsection{Convergence of the Neural Tangent Kernel}\label{sec:NTK}
Recall that $\Theta_{L}(\bx,\by)\in \mathbb{R}^{N_{L+1}\times N_{L+1}}$ is the matrix representation of
\begin{equation}
    \sum_{\ell=0}^{L} \lambda_{\ell}^2 (\mathrm{d}\Phi_{L}(\bx))_{W_{\ell}}\circ(\mathrm{d}\Phi_{L}(\by))_{W_{\ell}}^{\top},
\end{equation}
for $(\lambda_\ell^2)_{\ell\geq 0}$. We prove convergence of $\Theta_L(\mathbf{x},\mathbf{y})$ to its limit under the following conditions. 
\begin{assumption}\label{assumption:NTK}
     Assume that the input and output dimensions $N_0$ and $N_{L+1}$ of $\Phi_L$ are fixed, and that $N_{\ell}=N$ for all other layers $0<\ell \leq L$.
     Assume also that $\sigma_0=1$ and $\sigma_\ell = N^{-\frac{1}{2}}$ for $\ell>0$ (the ``GP limit" initialization).
\end{assumption}

\begin{theorem}\label{thm:ntk}
    Under Assumption \ref{assumption:NTK} and if $\lambda_\ell^2=\sigma_\ell^2\mathbf{1}(\ell>0)+\mathbf{1}(\ell=0)$, we have 
\begin{equation} 
    \mathbb{E}\bigg\{\Big\Vert\Theta_{L}(\bx,\by)  - \Theta_{L}^{\infty}(\bx,\by) \otimes \mathrm{Id}_{N_{L+1}}\Big\Vert^2\bigg\}=\mathcal{O}\bigg(\frac{N_{L+1}^2}{N}\bigg),
\end{equation}
where
\begin{equation}
    \Theta_0^{\infty}(\bx,\by) = \sprod{\bx}{\by}_{\bR^{N_0}}, \quad \Theta_{L}^{\infty}(\bx,\by) = K_{L}(\bx,\by) + \dot K_{L}(\bx,\by)\Theta_{L - 1}^{\infty}(\bx,\by) 
\end{equation}
and $\dot K_{\ell}$ is defined in the same way as $K_{\ell}$ but substituting $\varphi_{\ell}$ for $\varphi'_{\ell}$ in (\ref{def:kernel}).
\end{theorem}

 Recall that the $i,j$-th entry of the NTK can be expressed by the following graph expansion (see Equation (\ref{eqn:NTK_matrix_components}))
\begin{equation}\label{eq:ijNTK}
    \sum_{\ell=0}^L\lambda_\ell^2[(\mathrm{d}\Phi_{L}(\mathbf{x}))_{W_{\ell}}\circ(\mathrm{d}\Phi_{L}(\mathbf{y}))_{W_{\ell}}^{\top}]_{ij}
    = \sum_{\ell=0}^L
    \sum_{\substack{\tau \in \partial_{\ell}\bT_{L,i}(\mathbf{x})\\\eta \in \partial_{\ell}\bT_{L,j}(\mathbf{y})}}
    \frac{\varphi_{\tau}\varphi_{\eta}}{s(\tau)s(\eta)}
    \lambda_\ell^2\mathbf{W}_{\tau\circ\eta^{\top}}.
\end{equation}
We once again use Proposition \ref{prop:mixedmoments_} to study the joint distribution of the product graphs on the right-hand side, from which the following lemma follows.
\begin{lemma}\label{lemma:L2}
    Under Assumption \ref{assumption:NTK}, for any $\tau \in \partial_{\ell}\bT_{L,i_1}(\bx_1)$ and $\eta \in \partial_{\ell}\bT_{L,i_2}(\bx_2)$, we have
    \begin{equation}
        \E\bigg\{
     \Big( 
     \sigma_{\ell}^2 \bW_{\tau \circ \eta^{\top}} - 
    a(\tau \circ \eta^{\top})
     \Big)^2
     \bigg\} = 
    \mathcal{O}\Big(\frac{1}{N}\Big),
    \end{equation}
    where 
    \begin{equation}
        a(\tau \circ \eta^{\top}):= 
         \sum_{\psi \in \cP_{\mathrm{A}}(\tau \circ \eta^{\top})} \delta_{i_1i_2}
        \sprod{\bx}{\bx}_{\psi}.
    \end{equation}
\end{lemma}
\begin{proof}
    See Section \ref{sec:proofs_lem4&5}.
\end{proof}
\begin{remark}
This implies that we need to take $\lambda_\ell^2 = \sigma_{\ell}^2$ to get a well-defined limit, as assumed in Theorem \ref{thm:ntk}. This is known as the {``NTK parametrization"} (see \cite{jacot2018neural}). 
\end{remark}
As with Lemma \ref{lemma:alpha_trees} in the previous section, we will need a combinatorial decomposition of the limiting covariances $a(\tau\circ\eta^\top)$ in order to prove Theorem \ref{thm:ntk} by induction.
\begin{lemma}\label{lemma:NTK_bijection}
For $L\geq 1$, there is a canonical bijection $$\partial_\ell \bT_{L,i}(\bx) \to \bT_{L,i}(\bx) \times \partial_{\ell} \bT_{L-1,1}(\bx), \quad \eta \mapsto ( \eta^+, \eta^-)$$ under which 
\[
    a(\tau \circ \eta^{\top}) = \alpha(\tau^+ \sqcup \eta^+) a(\tau^- \circ {\eta^-}^{\top}), \quad {s(\tau^+)s(\tau^-) = s(\tau), \quad \varphi_{\tau} = \dot\varphi_{\tau^+}\varphi_{\tau^-}}
\]
where $\dot \varphi_{[\tau_1 \cdots \tau_n]_{L}} := \varphi_{L}^{(n+1)} \prod_i \varphi_{\tau_i}$ and $a$ is defined as in the statement of Lemma \ref{lemma:L2}.
Further, if $\ell = L$, we have $\partial_{\ell} \bT_{L-1,1}(\bx) := \{\bullet\}$ so that $\partial_\ell \bT_{L,i}(\bx) \simeq \bT_{L,i}(\bx)$ and $ a(\tau \circ \eta^{\top}) = \alpha(\tau^+ \sqcup \eta^+)$.
\end{lemma}
\begin{proof}
        See Section \ref{sec:proofs_lem4&5}.
\end{proof}
We can now derive the rate of convergence of the NTK to its limit.
 \begin{proof}[Proof of Theorem \ref{thm:ntk}]
    Let $i,j\in \{1,...,N_{L+1}\}$ be arbitrary, and recall from Equation \eqref{eq:ijNTK} that
\begin{align*}
    \lambda_\ell^2 [ (\mathrm{d}\Phi_{L}(\bx))_{W_{\ell}}\circ(\mathrm{d}\Phi_{L}(\by))_{W_{\ell}}^{\top} ]_{ij} 
    =& \sum_{\tau \in \partial_{\ell}\bT_{L,i}(\mathbf{x})}
    \sum_{\eta \in \partial_{\ell}\bT_{L,j}(\mathbf{y})}
    \frac{\varphi_{\tau}\varphi_{\eta}}{s(\tau)s(\eta)}
  \lambda_\ell^2 \mathbf{W}_{\tau\circ\eta^{\top}}
\end{align*}
for any $0\leq \ell \leq L$. Note that the number of terms in this sum is finite, and crucially does not depend on $N$ nor $N_0,N_{L+1}$. The same is true for the factors $\varphi_\tau,\varphi_\eta$ and $s(\tau), s(\eta)$. Lemma \ref{lemma:L2} and repeated applications of the elementary inequality $(a+b)^2\leq2(a^2+b^2)$ thus yield
\begin{align}\label{eq:L2NTKquant}
\mathbb{E}\bigg\{\Big(\sum_{\tau, \eta} \frac{\varphi_{\tau}\varphi_{\eta}}{s(\tau)s(\eta)} a(\tau \circ \eta^{\top})-\sum_{\tau,\eta}\frac{\varphi_{\tau}\varphi_{\eta}}{s(\tau)s(\eta)}\lambda_\ell^2\mathbf{W}_{\tau\circ\eta^\top}\Big)^2\bigg\}=\mathcal{O}\bigg(\frac{1}{N}\bigg),
\end{align}
where the implicit constant depends on $L$. 

For $\ell = L$, Lemma \ref{lemma:NTK_bijection} then gives
\[
\sum_{\tau, \eta} \frac{\varphi_{\tau}\varphi_{\eta}}{s(\tau)s(\eta)} a(\tau \circ \eta^{\top})=\sum_{\tau^+, \eta^+} \frac{\varphi_{\tau^+}\varphi_{\eta^+}}{s(\tau^+)s(\eta^+)} \alpha(\tau^+ \sqcup \eta^+)
\]
which is in turn equal to $\delta_{ij} K_{L}(\bx,\by)$ by the proof of Theorem \ref{thm:GPlimit}.
Otherwise, noting that $s(\tau)=s(\tau^+)s(\tau^-) $ and $\varphi_{\tau} = \dot\varphi_{\tau^+}\varphi_{\tau^-}$, we instead have 
\begin{align*}
&\sum_{\tau, \eta} \frac{\varphi_{\tau}\varphi_{\eta}}{s(\tau)s(\eta)} a(\tau \circ \eta^{\top})\\
&=
    \sum_{\tau^+, \eta^+}\sum_{\tau^-, \tilde\eta} \frac{\dot\varphi_{\tau^+}\dot\varphi_{\eta^+}}{s(\tau^+)s(\eta^+)}
\frac{\varphi_{\tau^-}\varphi_{\eta^-}}{s(\tau^-)s(\eta^-)} 
    \alpha(\tau^+ \sqcup \eta^+) a(\tau^- \circ {\eta^-}^{\top})
\end{align*}
which is equal to 
\[
\delta_{ij} \dot K_{L}(\bx,\by) \lim_{N \to \infty}
    {\sigma_\ell^2}[(\mathrm{d}\Phi_{L - 1}(\bx))_{W_{\ell}}(\mathrm{d}\Phi_{L - 1}(\by))_{W_{\ell}}^{\top}]_{1,1},
\]
(once again by the proof of Theorem \ref{thm:GPlimit}).

It follows that $[\Theta_{L}(\mathbf{x},\mathbf{y})]_{ij}$ converges in $L^2$ to $[\Theta_{L}^\infty(\mathbf{x},\mathbf{y})\otimes\mathrm{Id}_{N_{L+1}}]_{ij}$ at a rate of $\mathcal{O}(1/N)$ for each $i,j\in \{1,...,N_{L+1}\}$. The theorem then follows since $N_{L+1}$ is finite.
\end{proof}

\subsection{Distribution of the squared singular values of the Jacobian}\label{sec:jacobian}
Consider the input-output Jacobian $\mathbf{J}_{L,\bx}$, which we defined in Equation (\ref{def:Jacobian}) as the matrix representation of
\begin{equation}\label{eq:differential_sec5.3}
         \mathrm{d}(\varphi_{L}\circ\Phi_{L-1})_\bx \in \bR^{N_{L}\times N_0}.
\end{equation} 
Letting $\{\xi_1,...,\xi_{N_L}\}$ be the eigenvalues of $\mathbf{J}_{L,\bx}\mathbf{J}_{L,\bx}^\top$, recall that the empirical spectral distribution (ESD) of this matrix is defined as the random measure
\begin{equation}\label{def:JacobianESD}
    \rho_L := \frac{1}{N_L}\sum_{i=1}^{N_L} \delta_{\xi_i}
\end{equation}
where $\delta_{\xi_i}$ denotes a Dirac mass at $\xi_i$. Fix $x=(\lim_{N\to\infty}\frac{1}{N}\langle\mathbf{x},\mathbf{x}\rangle)^{1/2}$ and let $m_{k,L}^{(N)}(x)$ denote the $k$-th moment of $\rho_L$, noting that
\begin{equation}
        m_{k,L}^{(N)}(x) =       
        \frac{1}{N}
        \mathrm{Tr}\left((\mathbf{J}_{L,\bx}\mathbf{J}_{L,\bx}^{\top})^k\right).
\end{equation}

We will assume the following for the remainder of this section. 
\begin{assumption}\label{assumption:Jacobian} For all $\ell \geq 0$, $N_\ell=N$ and $\sigma_\ell^2=1/{N}$. Assume also that the $\bx \in \bR^N$ in \eqref{eq:differential_sec5.3} satisfies $\frac{1}{N}\sprod{\bx}{\bx}_{\bR^N} \to x^2 \in \bR$ as $N \to \infty$. 
\end{assumption}
Our goal is to show that under this assumption, $m_{k,L}^{(N)}(x)\to m_{k,L}(x)$, where $m_{k,L}$ is defined by the recursion in Equation \eqref{eq:recursion}. This is the content of Theorem \ref{thm:jacobianmoments}. In order to state the latter, we must first introduce the notion of a non-crossing partition and its Kreweras complement.
\begin{definition}
    A partition $\pi$ of $\{1,...,n\}$ is said to be \textit{non-crossing} if, whenever $a,b$ belong to a block (disjoint subset) of $\pi$ and $c,d$ to another, we cannot have $a<c<b<d$. We denote the set of all such partitions by  $\mathrm{NC}_n$. 
\end{definition}

To visualize the non-crossing property, plot the sequence $(1,...,n)$ in clockwise order as evenly spaced points on a circle. For every block $\{a_1,...,a_j\}$ of $\pi\in \mathrm{NC}_n$, connect the points corresponding to each $a_i$ to form polygons with these points as vertices, as depicted in Figure \ref{fig:dual} (if $j=2$, simply form a line segment). If $\pi$ is non-crossing, the resulting polygons and line segments will never intersect. 

\begin{figure}[ht]\label{fig:dual}
    \centering
    \scalebox{.99}{\input{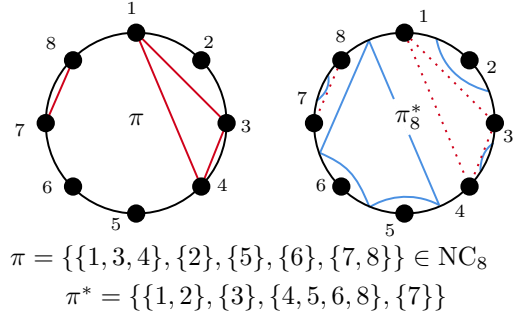}}
    \caption{A non-crossing partition $\pi$ of eight elements (red), and its Kreweras complement $\pi^*$ (blue).}
\end{figure}

Further, interpreting the circle as a cycle graph whose vertices are the same points $1,...,n$ and edges are the arcs connecting these points, such a $\pi\in \mathrm{NC}_8$ induces a partition on the edges as follows. 
Label by $i$ the edge connecting $v_i$ and $v_{i+1}$ and by $n$ the edge connecting $v_n$ and $v_1$. 
Then  two edges are put in the same block if we can draw a line inside the circle connecting them without intersecting a polygon (or line segment) corresponding to a block of $\pi$. This resulting edge partition is known as the \textit{Kreweras complement of $\pi$} (see \cite{nicaspeicher}), and denote it by $\pi^*$. Note that this is also an element of $\mathrm{NC}_n$.

\begin{theorem}\label{thm:jacobianmoments}  Define  
    \begin{equation}
        \mu_{k,\ell+1}(x) := \E[\varphi'_{\ell+1}(X_{\ell})^{2k}] \text{ for } \quad X_{\ell} \sim \mathcal{N}(0,K_{\ell}({x})),
    \end{equation}
    where $K_\ell$ is defined recursively by
    \[
    K_0(x) = x^2, \quad K_{\ell+1}(x) = \E[\varphi_{\ell+1}(X_{\ell})^{2}].
    \]
    Then under Assumption \ref{assumption:Jacobian},
    \[
        m_{k,L}^{(N)}(x)\underset{N\to\infty}{\longrightarrow} m_{k,L}(x),
    \]
    where $m_{k,L}$ is defined recursively as follows:
    \[
        m_{k,0}(x) = 1, \quad m_{1,L}(x) = \mu_{1,L}(x)m_{1,L-1}(x),
    \]
    \begin{equation}\label{eq:recursion}
        m_{k,L}(x)=\sum_{\pi\in \mathrm{NC}_k} \mu_{\pi,L}(x)m_{\pi^*,L-1}(x),
    \end{equation}
    where $\mu_{\pi,L}(x):=\prod_{B\in \pi} \mu_{|B|, L}(x)$ and similarly for $m_{\pi^*,L-1}(x)$.
\end{theorem}
\begin{remark}\label{eq:fusscatalan}
The moments ${m}_{k,L}$ can be seen as a non--linear (activation--dependent) generalisation of the Fuss-Catalan numbers (see, e.g., \cite{nicaspeicher}) \[
    \mathrm{FC}_L^{(k)}=\frac{1}{kL+1}\binom{(L+1)k}{k}, \quad \mathrm{FC}_L^{(k)} = \sum_{\pi\in \mathrm{NC}_k}\prod_{B\in \pi}\mathrm{FC}_{L-1}^{(|B|)}.
    \] Indeed, if we consider the linear case with $\varphi=\mathrm{Id}$ in the above, then $\mu_{\pi,\ell}=1$ for any $\pi,\ell$ and we recover $m_{k,L}=\sum_{\pi\in \mathrm{NC}_k}m_{\pi,L-1}$ (noting that $\pi\mapsto\pi^*$ is a bijection).
\end{remark}
Since $x$ is fixed a priori by Assumption \ref{assumption:Jacobian}, we omit it from the notation in what follows.

\subsubsection{Proof of Theorem \ref{thm:jacobianmoments}}
Recall that have derived a graph expansion for $m_{k,L}^{(N)}$ in Section \ref{sec:trees} (Equations \eqref{eq:k=1trace} and \eqref{eq:k=2trace}). For general $k$, the latter takes the form
\begin{equation}\label{def:jacobian_firstexp}    \frac{1}{N}\sum_{\eta_j^{(i)}}\left(\prod_{i=1}^2\prod_{j=1}^k\frac{\varphi_{\eta^{(i)}_j}}{s(\eta^{(i)}_j)}
    \right)\bW_{\mathrm{Tr}\big(\eta^{(1)}_1\circ (\eta^{(1)}_2)^{\top} \circ \cdots \circ \eta^{(k)}_1\circ (\eta^{(k)}_2)^{\top}\big)},
\end{equation}
where the sum if taken over all trees $\eta^{(i)}_j$ in $(\partial_\mathbf{x}\mathbb{T}_L(\mathbf{x}))^*$, for $i\in [2]$ and $j\in [k]$.

We call the graphs
\[
    \mathrm{Tr}\left(\eta^{(1)}_1\circ (\eta^{(1)}_2)^{\top} \circ \cdots \circ \eta^{(k)}_1\circ (\eta^{(k)}_2)^{\top}\right),
\]
arising in this expansion \textit{decorated cycles}, owing to their particular form: they each consist of a unique undirected cycle (meaning that one obtains a cycle after ignoring edge directions in the graph) which in this case is of length $2kL$, and trees rooted at vertices belonging to this cycle. An example for $L=2$ is depicted in Figure \ref{fig:decoratedcycleconcrete}.

\begin{figure}[t]
    \centering
    \input{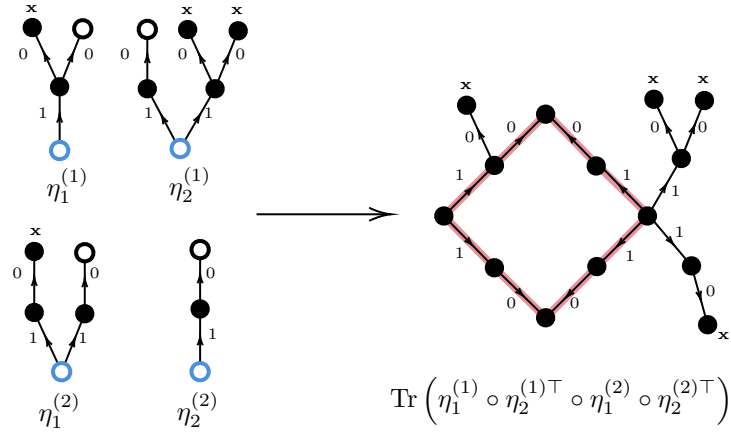}
    \caption{A decorated cycle (right) formed by four trees (left), with the latter's unique (undirected) cycle highlighted in red.}
    \label{fig:decoratedcycleconcrete}
\end{figure}

It will turn out to be convenient to work directly with the trees rooted at at the cycle instead of the initial $\eta_i^{(j)}$. To that end, note that in any $\eta_i^{(j)}\in(\partial_\mathbf{x}\mathbb{T}_L(\mathbf{x}))^*$, we have a unique path of length $L$ from the root to the (unique) free leaf, which we call the \textit{trunk}. The tree $\eta_i^{(j)}$ then consists of this trunk, together with $L$ subtrees  $\eta_{i,1}^{(j)},...,\eta_{i,L}^{(j)}$ belonging to $(\mathbb{T}_{1}(\mathbf{x}))^*,...,(\mathbb{T}_{L}(\mathbf{x}))^*$, respectively which come out of the trunk. This bijection is depicted in Figure \ref{fig:trunk_bijection} below. 

\begin{figure}[h]
    \centering
    {\input{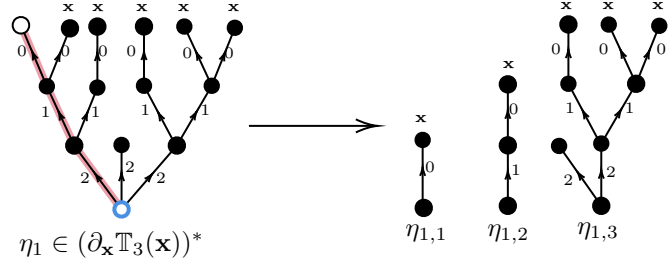}}
    \caption{A tree in $(\partial_\mathbf{x}\mathbb{T}(\mathbf{x}))^*$  is in bijection with the sequence of trees rooted at its trunk (highlighted in red).}
    \label{fig:trunk_bijection}
\end{figure}

Using this bijection, any graph of the form
\begin{equation}\label{def:decoratedcycle}
    \mathrm{Tr}\left(\eta^{(1)}_1\circ (\eta^{(1)}_2)^{\top} \circ \cdots \circ \eta^{(k)}_1\circ (\eta^{(k)}_2)^{\top}\right),
\end{equation}
thus corresponds to a unique sequence of trees 
\begin{align}\label{eq:treeseq}
\boldsymbol{\eta} &=(\eta_{i,l}^{(j)})_{j\in [k],i\in [2], l\in [L]} \in \prod_{i=1}^k ((\bT_{0}(\mathbf{x}))^*\times\cdots\times (\bT_{L}(\mathbf{x}))^*)^2,
\end{align}
(which are product graphs) and so we adopt the notation $G(\boldsymbol{\eta})$ to denote the graph, as well as $C(\boldsymbol{\eta})$ to denote its cycle. Lastly, we use $\mathrm{DC}(k,L)$ to denote the space of all such graphs. With this notation, we can rewrite the expansion in Equation \eqref{def:jacobian_firstexp} as follows.
\begin{lemma} 
\label{lemma:jacobian_moms_expansion}
\[      \mathrm{Tr}\big((\mathbf{J}_{L,\mathbf{x}}\mathbf{J}_{L,\mathbf{x}}^\top)^k\big)
=\sum_{G(\boldsymbol{\eta})} \frac{\dot \varphi_{\boldsymbol{\eta}}}{s(\boldsymbol{\eta})}\bW_{G(\boldsymbol{\eta})},
\]
where the sum is taken over all $G(\boldsymbol{\eta})\in \mathrm{DC}(k,L)$ (or equivalently, all $\boldsymbol{\eta}=(\eta_{i,l}^{(j)})_{i,j,l}$, \textit{c.f.} Equation (\ref{eq:treeseq})), and
\[
\frac{\dot \varphi_{\boldsymbol{\eta}}}{s(\boldsymbol{\eta})}:=\prod_{i,j,l}\frac{\dot \varphi_{\eta_{i,l}^{(j)}}}{s(\eta_{i,l}^{(j)})}.
\]
\end{lemma}
\begin{proof}
    This follows from the aforementioned observation, and the fact that
    \[
\left(\prod_{i=1}^2\prod_{j=1}^k\frac{\varphi_{\eta^{(j)}_i}}{s(\eta^{(j)}_i)}
    \right)=\prod_{i,j,l}\frac{\dot \varphi_{\eta_{i,l}^{(j)}}}{s(\eta_{i,l}^{(j)})}
    \]
    where $\dot\varphi$ is defined as in Lemma \ref{lemma:NTK_bijection}.
\end{proof}
We thus need to determine the limits of (values of) decorated cycles, which is achieved in the following proposition.
\begin{proposition}\label{prop:decorated_cycle_scaling}
    Let $k,L\geq 0$ and $G=G(\boldsymbol{\eta})\in \mathrm{DC}(k,L)$  be a decorated cycle. Then under Assumption \ref{assumption:Jacobian}, we have that
    \begin{equation}
        \E\left\{\left(
        \lambda_G \bW_G - a(G)
        \right)^2\right\} = \mathcal{O}\bigg(\frac{1}{N}\bigg),
    \end{equation}
    where
    \begin{equation}
        \lambda_G = \frac{1}{\sigma_G N^{\check{e}(G)+1}},
        \text{ and }
        a(G) := \sum_{\psi \in \cP_{\mathrm{A}}(G)}
        \prod_{\{\bullet_{\bx}, \bullet_{\bx}\} \in (\mathcal{L}_0(G)/\sim_{\psi})}
        \frac{\sprod{\bx}{\bx}}{N}.
    \end{equation}
\end{proposition}
\begin{proof}
    See Section \ref{proofs_prop12}.
\end{proof}

\begin{remark}
    Note that by Assumption \ref{assumption:Jacobian}, $\lambda_G = \frac{1}{N}$ for all decorated cycles $G$ since $\sigma_\ell = \frac{1}{\sqrt{N}}$ for all $\ell \geq 0$. 
\end{remark}

    Since $\sigma_\ell = \frac{1}{\sqrt{N}}$, we then have that
    \[
    \frac{1}{N}\mathrm{Tr}\big((\mathbf{J}_{L,\mathbf{x}}\mathbf{J}_{L,\mathbf{x}}^\top)^k\big)
    =\sum_{\substack{G(\boldsymbol{\eta})\in \mathrm{DC}(k,L)}} \frac{\dot \varphi_{\boldsymbol{\eta}}}{s(\boldsymbol{\eta})} \lambda_{G(\boldsymbol{\eta})} \bW_{G(\boldsymbol{\eta})}
    \xrightarrow{L^2} 
    \sum_{\substack{G(\boldsymbol{\eta})\in \mathrm{DC}(k,L)}} \frac{\dot \varphi_{\boldsymbol{\eta}}}{s(\boldsymbol{\eta})} a(G(\boldsymbol{\eta}))
    \]
    with rate $\mathcal{O}\big(\frac{1}{N}\big)$, since these are finite sums.

Now letting 
\[
     \tilde{m}_{k,L}=\sum_{\substack{G(\boldsymbol{\eta})\in \mathrm{DC}(k,L)}} \frac{\dot \varphi_{\boldsymbol{\eta}}}{s(\boldsymbol{\eta})} a(G(\boldsymbol{\eta})),
\]
it suffices to show that $\tilde{m}_{k,L}$ satisfies the recursion in \eqref{eq:recursion}.

To derive a recursive formula for the coefficients $m_{k,L}$, we analyze the structure of atomic pairings for decorated cycles, beginning with the following result.
\begin{lemma}\label{lem:firstdecomp} 
Let $G(\boldsymbol{\eta})\in \mathrm{DC}(k,L)$, $C(\boldsymbol{\eta})$ be this graph's cycle (\emph{i.e.} the subgraph which is a cycle after forgetting about edge orientations), and $\phi \in\mathcal{P}_\mathrm{A}(G(\boldsymbol{\eta}))$. Then $\phi=\phi_C \sqcup \phi_{G\setminus C}$ for some $\phi_C\in\mathcal{P}(C(\boldsymbol{\eta}))$. In particular, all of the $L$-edges of $C(\boldsymbol{\eta})$ are paired between themselves by some pairing $\phi_L\subseteq \phi_C$.
\end{lemma}
\begin{proof}
For $\phi$ to be atomic, it must pair edges in a way that collapses the undirected cycle $C(\boldsymbol{\eta})$ without creating any new cycles in the process. But this is only possible if edges of $C(\boldsymbol{\eta})$ are paired between themselves, as any edge that isn't paired this way would end up in a cycle as depicted in Figure \ref{fig:cyclicpairing} below.
\begin{figure}[ht]
    \centering
    \input{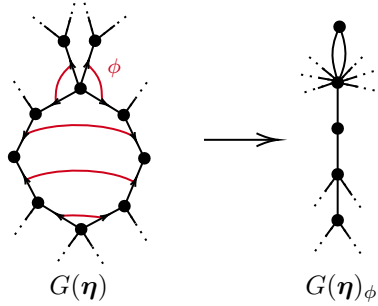}
    \caption{A cycle formed by pairing edges that belong to $C(\boldsymbol{\eta})$ with edges that do not.}
    \label{fig:cyclicpairing}
\end{figure}

Further, since edges can only be paired with other edges with the same input, $\phi$ must contain a pairing $\phi_{L}$ of the $L$-edges in $C(\boldsymbol{\eta})$.
\end{proof}

Now let $\phi\in \mathcal{P}_{\mathrm{A}}(G(\boldsymbol{\eta}))$, and $\phi_L$ be as in the lemma above. Note that in $G(\boldsymbol{\eta})$, the trees $\eta_{1,L}^{(j)},\eta_{2,L}^{(j)}$ share a root for each $j$, which belongs to $C(\boldsymbol{\eta})$. For each $j\in[k]$, we let $v_j:=\text{root}(\eta_{1,L}^{(j)}\wedge\eta_{2,L}^{(j)})$ denote said roots, and note that each $v_j$ is the head of exactly two $L$-edges in $C(\boldsymbol{\eta})$. All other vertices in the latter are heads of edges with inputs $W_i$ with $i<L$.

Said differently, the $k$ pairs of $L$-edges of $C(\boldsymbol{\eta})$ are in bijection with $v_1,...,v_k$, and through this bijection, $\phi_L\subseteq \phi$ induces a unique partition $\pi$ of $\{1,...,k\}$ defined by letting $a$ and $b$ be in the same block of $\pi$ if they label endpoints of edges which have been paired by $\phi_{L}$. This is depicted in Figure \ref{fig:noncrossingpartition}, and we write $\phi_L\sim \pi$ to denote the fact that $\phi_L$ induces $\pi$.

 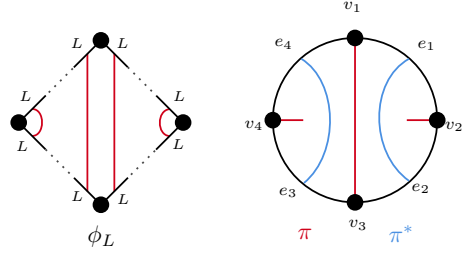
\begin{figure}[!ht]
    \centering
    \scalebox{.90}{\tikzset{every picture/.style={line width=0.75pt}} 

\begin{tikzpicture}[x=0.75pt,y=0.75pt,yscale=-1,xscale=1]

\draw [color={rgb, 255:red, 208; green, 2; blue, 27 }  ,draw opacity=1 ]   (294.71,1153.91) -- (294.71,1245.6) ;
\draw [color={rgb, 255:red, 208; green, 2; blue, 27 }  ,draw opacity=1 ]   (191.37,1193.6) .. controls (184,1194.18) and (184,1208.18) .. (191.31,1208.63) ;
\draw [color={rgb, 255:red, 208; green, 2; blue, 27 }  ,draw opacity=1 ]   (114.71,1193.54) .. controls (122.05,1193.54) and (121.99,1208.57) .. (114.66,1208.57) ;
\draw [color={rgb, 255:red, 208; green, 2; blue, 27 }  ,draw opacity=1 ]   (145.47,1162.79) -- (145.47,1239.44) ;
\draw [color={rgb, 255:red, 208; green, 2; blue, 27 }  ,draw opacity=1 ]   (160.5,1162.73) -- (160.5,1239.38) ;
\draw   (248.86,1199.75) .. controls (248.86,1174.43) and (269.39,1153.91) .. (294.71,1153.91) .. controls (320.03,1153.91) and (340.55,1174.43) .. (340.55,1199.75) .. controls (340.55,1225.07) and (320.03,1245.6) .. (294.71,1245.6) .. controls (269.39,1245.6) and (248.86,1225.07) .. (248.86,1199.75) -- cycle ;
\draw  [color={rgb, 255:red, 0; green, 0; blue, 0 }  ,draw opacity=1 ][fill={rgb, 255:red, 0; green, 0; blue, 0 }  ,fill opacity=1 ] (290.58,1153.91) .. controls (290.58,1151.63) and (292.43,1149.78) .. (294.71,1149.78) .. controls (296.99,1149.78) and (298.83,1151.63) .. (298.83,1153.91) .. controls (298.83,1156.19) and (296.99,1158.04) .. (294.71,1158.04) .. controls (292.43,1158.04) and (290.58,1156.19) .. (290.58,1153.91) -- cycle ;
\draw  [color={rgb, 255:red, 0; green, 0; blue, 0 }  ,draw opacity=1 ][fill={rgb, 255:red, 0; green, 0; blue, 0 }  ,fill opacity=1 ] (194.73,1201.09) .. controls (194.73,1198.81) and (196.58,1196.96) .. (198.86,1196.96) .. controls (201.13,1196.96) and (202.98,1198.81) .. (202.98,1201.09) .. controls (202.98,1203.37) and (201.13,1205.21) .. (198.86,1205.21) .. controls (196.58,1205.21) and (194.73,1203.37) .. (194.73,1201.09) -- cycle ;
\draw  [color={rgb, 255:red, 0; green, 0; blue, 0 }  ,draw opacity=1 ][fill={rgb, 255:red, 0; green, 0; blue, 0 }  ,fill opacity=1 ] (103.04,1201.09) .. controls (103.04,1198.81) and (104.89,1196.96) .. (107.17,1196.96) .. controls (109.45,1196.96) and (111.29,1198.81) .. (111.29,1201.09) .. controls (111.29,1203.37) and (109.45,1205.21) .. (107.17,1205.21) .. controls (104.89,1205.21) and (103.04,1203.37) .. (103.04,1201.09) -- cycle ;
\draw  [color={rgb, 255:red, 0; green, 0; blue, 0 }  ,draw opacity=1 ][fill={rgb, 255:red, 0; green, 0; blue, 0 }  ,fill opacity=1 ] (290.58,1245.6) .. controls (290.58,1243.32) and (292.43,1241.47) .. (294.71,1241.47) .. controls (296.99,1241.47) and (298.83,1243.32) .. (298.83,1245.6) .. controls (298.83,1247.88) and (296.99,1249.72) .. (294.71,1249.72) .. controls (292.43,1249.72) and (290.58,1247.88) .. (290.58,1245.6) -- cycle ;
\draw  [color={rgb, 255:red, 0; green, 0; blue, 0 }  ,draw opacity=1 ][fill={rgb, 255:red, 0; green, 0; blue, 0 }  ,fill opacity=1 ] (148.89,1246.93) .. controls (148.89,1244.65) and (150.73,1242.8) .. (153.01,1242.8) .. controls (155.29,1242.8) and (157.14,1244.65) .. (157.14,1246.93) .. controls (157.14,1249.21) and (155.29,1251.06) .. (153.01,1251.06) .. controls (150.73,1251.06) and (148.89,1249.21) .. (148.89,1246.93) -- cycle ;
\draw  [color={rgb, 255:red, 0; green, 0; blue, 0 }  ,draw opacity=1 ][fill={rgb, 255:red, 0; green, 0; blue, 0 }  ,fill opacity=1 ] (148.89,1155.24) .. controls (148.89,1152.96) and (150.73,1151.12) .. (153.01,1151.12) .. controls (155.29,1151.12) and (157.14,1152.96) .. (157.14,1155.24) .. controls (157.14,1157.52) and (155.29,1159.37) .. (153.01,1159.37) .. controls (150.73,1159.37) and (148.89,1157.52) .. (148.89,1155.24) -- cycle ;
\draw    (107.17,1201.09) -- (122.14,1216.06) ;
\draw    (138.04,1231.95) -- (153.01,1246.93) ;
\draw    (153.01,1155.24) -- (167.99,1170.22) ;
\draw    (183.88,1186.11) -- (198.86,1201.09) ;
\draw    (168.1,1231.84) -- (153.01,1246.93) ;
\draw    (198.86,1201.09) -- (183.77,1216.18) ;
\draw    (153.01,1155.24) -- (137.92,1170.33) ;
\draw    (122.26,1186) -- (107.17,1201.09) ;
\draw [color={rgb, 255:red, 74; green, 74; blue, 74 }  ,draw opacity=1 ] [dash pattern={on 0.84pt off 2.51pt}]  (137.35,1170.91) -- (122.26,1186) ;
\draw [color={rgb, 255:red, 74; green, 74; blue, 74 }  ,draw opacity=1 ] [dash pattern={on 0.84pt off 2.51pt}]  (183.19,1216.75) -- (168.1,1231.84) ;
\draw [color={rgb, 255:red, 74; green, 74; blue, 74 }  ,draw opacity=1 ] [dash pattern={on 0.84pt off 2.51pt}]  (122.14,1216.06) -- (138.04,1231.95) ;
\draw [color={rgb, 255:red, 74; green, 74; blue, 74 }  ,draw opacity=1 ] [dash pattern={on 0.84pt off 2.51pt}]  (167.99,1170.22) -- (183.88,1186.11) ;
\draw [color={rgb, 255:red, 74; green, 144; blue, 226 }  ,draw opacity=1 ]   (264.78,1165.57) .. controls (285.33,1178.18) and (286.33,1218.18) .. (266.24,1234.78) ;
\draw [color={rgb, 255:red, 74; green, 144; blue, 226 }  ,draw opacity=1 ]   (324.78,1165.57) .. controls (302.67,1177.84) and (302.67,1217.84) .. (324.78,1233.57) ;
\draw [color={rgb, 255:red, 208; green, 2; blue, 27 }  ,draw opacity=1 ]   (248.86,1199.75) -- (265.83,1199.75) ;
\draw [color={rgb, 255:red, 208; green, 2; blue, 27 }  ,draw opacity=1 ]   (323.58,1199.75) -- (340.55,1199.75) ;
\draw  [color={rgb, 255:red, 0; green, 0; blue, 0 }  ,draw opacity=1 ][fill={rgb, 255:red, 0; green, 0; blue, 0 }  ,fill opacity=1 ] (244.74,1199.75) .. controls (244.74,1197.47) and (246.58,1195.63) .. (248.86,1195.63) .. controls (251.14,1195.63) and (252.99,1197.47) .. (252.99,1199.75) .. controls (252.99,1202.03) and (251.14,1203.88) .. (248.86,1203.88) .. controls (246.58,1203.88) and (244.74,1202.03) .. (244.74,1199.75) -- cycle ;
\draw  [color={rgb, 255:red, 0; green, 0; blue, 0 }  ,draw opacity=1 ][fill={rgb, 255:red, 0; green, 0; blue, 0 }  ,fill opacity=1 ] (336.42,1199.75) .. controls (336.42,1197.47) and (338.27,1195.63) .. (340.55,1195.63) .. controls (342.83,1195.63) and (344.68,1197.47) .. (344.68,1199.75) .. controls (344.68,1202.03) and (342.83,1203.88) .. (340.55,1203.88) .. controls (338.27,1203.88) and (336.42,1202.03) .. (336.42,1199.75) -- cycle ;

\draw (261.33,1259) node [anchor=north west][inner sep=0.75pt]  [color={rgb, 255:red, 208; green, 2; blue, 27 }  ,opacity=1 ]  {$\pi$};
\draw (312,1257) node [anchor=north west][inner sep=0.75pt]  [color={rgb, 255:red, 74; green, 144; blue, 226 }  ,opacity=1 ]  {$\pi^{*}$};
\draw (248.33,1153.07) node [anchor=north west][inner sep=0.75pt]  [font=\scriptsize]  {$e_{4}$};
\draw (250.78,1235.91) node [anchor=north west][inner sep=0.75pt]  [font=\scriptsize]  {$e_{3}$};
\draw (324.78,1233.57) node [anchor=north west][inner sep=0.75pt]  [font=\scriptsize]  {$e_{2}$};
\draw (328,1153.74) node [anchor=north west][inner sep=0.75pt]  [font=\scriptsize]  {$e_{1}$};
\draw (231.58,1198.6) node [anchor=north west][inner sep=0.75pt]  [font=\scriptsize]  {$v_{4}$};
\draw (289.58,1253.93) node [anchor=north west][inner sep=0.75pt]  [font=\scriptsize]  {$v_{3}$};
\draw (343.33,1199.07) node [anchor=north west][inner sep=0.75pt]  [font=\scriptsize]  {$v_{2}$};
\draw (286.67,1132.41) node [anchor=north west][inner sep=0.75pt]  [font=\scriptsize]  {$v_{1}$};
\draw (144,1257) node [anchor=north west][inner sep=0.75pt]    {$\phi _{L}$};
\draw (160.5,1239.38) node [anchor=north west][inner sep=0.75pt]  [font=\tiny]  {$L$};
\draw (134.97,1237.93) node [anchor=north west][inner sep=0.75pt]  [font=\tiny]  {$L$};
\draw (191.31,1208.63) node [anchor=north west][inner sep=0.75pt]  [font=\tiny]  {$L$};
\draw (191.77,1181.92) node [anchor=north west][inner sep=0.75pt]  [font=\tiny]  {$L$};
\draw (160.27,1152.58) node [anchor=north west][inner sep=0.75pt]  [font=\tiny]  {$L$};
\draw (134.76,1152.58) node [anchor=north west][inner sep=0.75pt]  [font=\tiny]  {$L$};
\draw (104.25,1210.26) node [anchor=north west][inner sep=0.75pt]  [font=\tiny]  {$L$};
\draw (106.09,1181.92) node [anchor=north west][inner sep=0.75pt]  [font=\tiny]  {$L$};

\end{tikzpicture}}
    \caption{A pairing $\phi_L$ of the $L$-edges of a cycle $C(\boldsymbol{\eta})$ (left), the partition $\pi\in \mathrm{NC}_4$ that is induced by $\phi_L$ (right, in red), and its complement $\pi^*$ (in blue).}
    \label{fig:noncrossingpartition}
\end{figure}

\bigskip
For atomic $\phi$'s, we notice that the only $\pi$ that arise this way must be non-crossing. 
Indeed, if $\phi_L\sim \pi$ and $\pi$ admits a crossing, then a cycle is formed as depicted in Figure \ref{fig:crossing} below.


\begin{figure}[!ht]
    \centering
    \tikzset{every picture/.style={line width=0.75pt}} 

\begin{tikzpicture}[x=0.75pt,y=0.75pt,yscale=-1,xscale=1]

\draw [color={rgb, 255:red, 0; green, 0; blue, 0 }  ,draw opacity=1 ]   (338.96,3501.87) .. controls (343.63,3515.03) and (375.96,3510.87) .. (377.96,3501.72) ;
\draw [color={rgb, 255:red, 0; green, 0; blue, 0 }  ,draw opacity=1 ]   (338.96,3501.87) .. controls (342.63,3491.03) and (374.96,3490.87) .. (377.96,3501.72) ;
\draw [color={rgb, 255:red, 74; green, 74; blue, 74 }  ,draw opacity=1 ] [dash pattern={on 0.84pt off 2.51pt}]  (194.82,3520.03) -- (194.42,3490.42) ;
\draw [color={rgb, 255:red, 208; green, 2; blue, 27 }  ,draw opacity=1 ]   (206.36,3479.57) .. controls (225.82,3491.16) and (231.82,3520.16) .. (232.06,3545.87) ;
\draw [color={rgb, 255:red, 208; green, 2; blue, 27 }  ,draw opacity=1 ]   (206.56,3531.31) .. controls (218.82,3519.16) and (236.96,3489.88) .. (231.56,3463.44) ;
\draw [color={rgb, 255:red, 74; green, 74; blue, 74 }  ,draw opacity=1 ] [dash pattern={on 0.84pt off 2.51pt}]  (271.82,3547.16) -- (245.82,3549.16) ;
\draw  [color={rgb, 255:red, 0; green, 0; blue, 0 }  ,draw opacity=1 ][fill={rgb, 255:red, 0; green, 0; blue, 0 }  ,fill opacity=1 ] (214.17,3468.72) .. controls (214.17,3466.44) and (216.02,3464.59) .. (218.3,3464.59) .. controls (220.58,3464.59) and (222.42,3466.44) .. (222.42,3468.72) .. controls (222.42,3471) and (220.58,3472.84) .. (218.3,3472.84) .. controls (216.02,3472.84) and (214.17,3471) .. (214.17,3468.72) -- cycle ;
\draw    (218.3,3468.72) -- (244.82,3458.16) ;
\draw [shift={(233.79,3462.55)}, rotate = 158.3] [color={rgb, 255:red, 0; green, 0; blue, 0 }  ][line width=0.75]    (4.37,-1.96) .. controls (2.78,-0.92) and (1.32,-0.27) .. (0,0) .. controls (1.32,0.27) and (2.78,0.92) .. (4.37,1.96)   ;
\draw  [color={rgb, 255:red, 0; green, 0; blue, 0 }  ,draw opacity=1 ][fill={rgb, 255:red, 0; green, 0; blue, 0 }  ,fill opacity=1 ] (214.17,3542.58) .. controls (214.17,3540.3) and (216.02,3538.46) .. (218.3,3538.46) .. controls (220.58,3538.46) and (222.42,3540.3) .. (222.42,3542.58) .. controls (222.42,3544.86) and (220.58,3546.71) .. (218.3,3546.71) .. controls (216.02,3546.71) and (214.17,3544.86) .. (214.17,3542.58) -- cycle ;
\draw    (218.3,3542.58) -- (245.82,3549.16) ;
\draw [shift={(234.39,3546.43)}, rotate = 193.44] [color={rgb, 255:red, 0; green, 0; blue, 0 }  ][line width=0.75]    (4.37,-1.96) .. controls (2.78,-0.92) and (1.32,-0.27) .. (0,0) .. controls (1.32,0.27) and (2.78,0.92) .. (4.37,1.96)   ;
\draw    (218.3,3468.72) -- (194.42,3490.42) ;
\draw [shift={(208.88,3477.28)}, rotate = 137.73] [color={rgb, 255:red, 0; green, 0; blue, 0 }  ][line width=0.75]    (4.37,-1.96) .. controls (2.78,-0.92) and (1.32,-0.27) .. (0,0) .. controls (1.32,0.27) and (2.78,0.92) .. (4.37,1.96)   ;
\draw    (194.82,3520.03) -- (218.3,3542.58) ;
\draw [shift={(208.29,3532.97)}, rotate = 223.85] [color={rgb, 255:red, 0; green, 0; blue, 0 }  ][line width=0.75]    (4.37,-1.96) .. controls (2.78,-0.92) and (1.32,-0.27) .. (0,0) .. controls (1.32,0.27) and (2.78,0.92) .. (4.37,1.96)   ;
\draw  [color={rgb, 255:red, 0; green, 0; blue, 0 }  ,draw opacity=1 ][fill={rgb, 255:red, 0; green, 0; blue, 0 }  ,fill opacity=1 ] (190.33,3490.95) .. controls (190.04,3488.69) and (191.63,3486.62) .. (193.89,3486.33) .. controls (196.15,3486.03) and (198.22,3487.63) .. (198.51,3489.89) .. controls (198.81,3492.15) and (197.21,3494.22) .. (194.96,3494.51) .. controls (192.7,3494.8) and (190.63,3493.21) .. (190.33,3490.95) -- cycle ;
\draw  [color={rgb, 255:red, 0; green, 0; blue, 0 }  ,draw opacity=1 ][fill={rgb, 255:red, 0; green, 0; blue, 0 }  ,fill opacity=1 ] (190.7,3520.03) .. controls (190.7,3517.75) and (192.54,3515.91) .. (194.82,3515.91) .. controls (197.1,3515.91) and (198.95,3517.75) .. (198.95,3520.03) .. controls (198.95,3522.31) and (197.1,3524.16) .. (194.82,3524.16) .. controls (192.54,3524.16) and (190.7,3522.31) .. (190.7,3520.03) -- cycle ;
\draw  [color={rgb, 255:red, 0; green, 0; blue, 0 }  ,draw opacity=1 ][fill={rgb, 255:red, 0; green, 0; blue, 0 }  ,fill opacity=1 ] (240.7,3458.16) .. controls (240.7,3455.88) and (242.54,3454.03) .. (244.82,3454.03) .. controls (247.1,3454.03) and (248.95,3455.88) .. (248.95,3458.16) .. controls (248.95,3460.44) and (247.1,3462.28) .. (244.82,3462.28) .. controls (242.54,3462.28) and (240.7,3460.44) .. (240.7,3458.16) -- cycle ;
\draw  [color={rgb, 255:red, 0; green, 0; blue, 0 }  ,draw opacity=1 ][fill={rgb, 255:red, 0; green, 0; blue, 0 }  ,fill opacity=1 ] (241.7,3549.16) .. controls (241.7,3546.88) and (243.54,3545.03) .. (245.82,3545.03) .. controls (248.1,3545.03) and (249.95,3546.88) .. (249.95,3549.16) .. controls (249.95,3551.44) and (248.1,3553.28) .. (245.82,3553.28) .. controls (243.54,3553.28) and (241.7,3551.44) .. (241.7,3549.16) -- cycle ;
\draw    (245.82,3549.16) -- (258.82,3548.16) ;
\draw [color={rgb, 255:red, 74; green, 74; blue, 74 }  ,draw opacity=1 ] [dash pattern={on 0.84pt off 2.51pt}]  (273.82,3462.16) -- (244.82,3458.16) ;
\draw    (244.82,3458.16) -- (259.82,3460.16) ;
\draw    (194.62,3499.23) -- (194.42,3490.42) ;
\draw    (194.82,3520.03) -- (194.62,3512.23) ;
\draw [color={rgb, 255:red, 0; green, 0; blue, 0 }  ,draw opacity=1 ][fill={rgb, 255:red, 0; green, 0; blue, 0 }  ,fill opacity=1 ]   (264,3502.42) -- (299.29,3502.42) ;
\draw [shift={(301.29,3502.42)}, rotate = 180] [color={rgb, 255:red, 0; green, 0; blue, 0 }  ,draw opacity=1 ][line width=0.75]    (10.93,-3.29) .. controls (6.95,-1.4) and (3.31,-0.3) .. (0,0) .. controls (3.31,0.3) and (6.95,1.4) .. (10.93,3.29)   ;
\draw  [color={rgb, 255:red, 0; green, 0; blue, 0 }  ,draw opacity=1 ][fill={rgb, 255:red, 0; green, 0; blue, 0 }  ,fill opacity=1 ] (373.84,3501.72) .. controls (373.84,3499.44) and (375.68,3497.59) .. (377.96,3497.59) .. controls (380.24,3497.59) and (382.09,3499.44) .. (382.09,3501.72) .. controls (382.09,3504) and (380.24,3505.84) .. (377.96,3505.84) .. controls (375.68,3505.84) and (373.84,3504) .. (373.84,3501.72) -- cycle ;
\draw  [color={rgb, 255:red, 0; green, 0; blue, 0 }  ,draw opacity=1 ][fill={rgb, 255:red, 0; green, 0; blue, 0 }  ,fill opacity=1 ] (334.84,3501.87) .. controls (334.84,3499.59) and (336.68,3497.75) .. (338.96,3497.75) .. controls (341.24,3497.75) and (343.09,3499.59) .. (343.09,3501.87) .. controls (343.09,3504.15) and (341.24,3506) .. (338.96,3506) .. controls (336.68,3506) and (334.84,3504.15) .. (334.84,3501.87) -- cycle ;
\draw    (328.63,3494.03) -- (338.96,3501.87) ;
\draw    (338.96,3501.87) -- (327.96,3510.87) ;
\draw    (388.96,3509.87) -- (377.96,3501.72) ;
\draw    (377.96,3501.72) -- (387.96,3492.87) ;
\draw  [dash pattern={on 0.84pt off 2.51pt}]  (328.63,3494.03) .. controls (308.06,3479.39) and (301.06,3520.39) .. (327.96,3510.87) ;
\draw  [dash pattern={on 0.84pt off 2.51pt}]  (387.96,3492.87) .. controls (407.91,3476.39) and (407.91,3524.39) .. (388.96,3509.87) ;

\end{tikzpicture}
    \caption{If $\phi_L\sim \pi$ and $\pi$ has a crossing, then identifying edges according to $\phi_L$ creates a cycle.}
    \label{fig:crossing}
\end{figure}
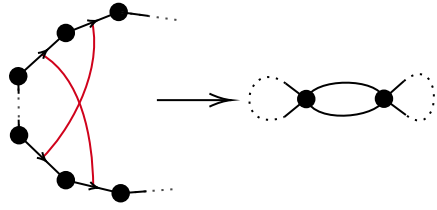


Next, recall that by definition, $v_{j}$ is the root of $\eta_{1,L}^{(j)}\wedge \eta_{2,L}^{(j)}$ for every $j\in [k]$. Upon identifying $L$-edges of $C(\boldsymbol{\eta})$ by $\phi_L$, such trees get joined at their roots. More precisely, if $\phi_L\sim \pi$, then for every block $B=\{j_1,...,j_l\}$ of $\pi$ we obtain the tree $$(\eta^{(j_1)}_{1,L}\wedge\eta^{(j_1)}_{2,L}\wedge \cdots \wedge  \eta^{(j_l)}_{1,L}\wedge\eta^{(j_l)}_{2,L}):=\wedge_{i=1}^l(\eta^{(j_i)}_{1,L}\wedge\eta^{(j_i)}_{2,L}).$$
Furthermore, for $\phi$ to be acyclic, it must at the very least pair edges of such trees internally within each tree (otherwise, a cycle would be formed at their roots). We therefore get, for each block $B=\{j_1,..,j_l\}$ of $\pi$, a pairing $\phi_B\in \mathcal{P}_{\mathrm{A}}(\wedge_{i=1}^l(\eta_{1,L}^{(j_i)}\wedge\eta_{2,L}^{(j_i)}))$

Having now paired $L$-edges belonging to $C(\boldsymbol{\eta})$ (using $\phi_L$), and those belonging to trees of height $L$ decorating this cycle (using the $\phi_B$'s), we must describe how $\phi$ pairs the remaining edges. Crucially, observe that the unpaired edges have inputs $W_i$ for $i\in \{1,...,L-1\}$ and form disjoint decorated cycles themselves. To be precise, we will need the following definition.

\begin{definition}
    Let $G(\boldsymbol{\eta})\in \mathrm{DC}(k,L), \phi_L$ be as above and $\phi_L\sim \pi$ for some $\pi\in \mathrm{NC}_k$. 
    Let $\tilde{B}$ be a block of $\pi^*\in\mathrm{NC}_k$.
    Then the \textup{restriction of $G(\boldsymbol{\eta})_{\phi_L}$ to $\tilde{B}$} is the decorated cycle in $\mathrm{DC}(|\tilde{B}|,L-1)$ determined by the trees
    \begin{align*}
        \boldsymbol{\eta}|_{\tilde{B}}=(\eta_{i,l}^{(j)})_{i,j,l},\quad i\in [2], l\in [L-1], j\in B.
    \end{align*}
\end{definition}

Returning to our previous point, the leftover edges (that were not paired by $\phi_L$ or $\phi_B$) form one decorated cycle per block of $\pi^*$, which are given by
\[
    G(\boldsymbol{\eta}|_{\tilde{B}}),\quad \tilde{B}\in \pi^*.
\]
For $\phi$ to be atomic, we claim that the edges of each of these $G(\boldsymbol{\eta}|_{\tilde{B}})$ must once again be paired internally. This is because the $G(\boldsymbol{\eta}|_{\tilde{B}})$ are connected to one another in $G(\boldsymbol{\eta})_{\phi_L\sqcup\{\phi_B\}_{B\in \pi}}$ by paths whose edges that have already been paired by $\phi_B$ or $\phi_L$; pairing some $e\in G(\boldsymbol{\eta}|_{\tilde{B}})$ to another $e'\in G(\boldsymbol{\eta}|_{\tilde{B'}})$ would then create a cycle which we cannot collapse (as in Figure \ref{fig:crossing}). Put differently, this implies that any pairing $\phi'$ of the remaining edges can be decomposed into a disjoint union as follows:
\[ 
    {\phi'}=\bigsqcup_{\tilde{B}\in \pi^*} \phi'_{\tilde{B}}.
\]
Summarizing everything we have said so far, any  pairing $\phi\in \mathcal{P}_{\mathrm{A}}(G(\boldsymbol{\eta}))$ consists of the following:
\begin{enumerate}
    \item A pairing $\phi_L$ of the $L$-edges of $C(\boldsymbol{\eta})$, such that $\phi_L\sim \pi\in \mathrm{NC}_k$.
    \item For every block $B=\{j_1,...,j_l\}$ of $\pi$, a pairing $\phi_B \in \mathcal{P}_{\mathrm{A}}(\wedge_{i=1}^l(\eta^{(j_i)}_{1,L}\wedge\eta^{(j_i)}_{2,L}))$.
    \item For every block $\tilde{B}$ of $\pi^*_k$, a pairing $\phi_{\tilde{B}}\in \mathcal{P}_{\mathrm{A}}(G(\boldsymbol{\eta})|_{\tilde{B}})$.
\end{enumerate}
Noting how the quantity 
\[
\prod_{\{\bullet_{\bx}, \bullet_{\bx}\} \in (\mathcal{L}_0(G(\mathbf{\eta}))/\sim_{\phi})}
        \frac{\sprod{\bx}{\bx}}{N}
\]
then factors into
\[
\left(
\prod_{B \in \pi}
\prod_{\{\bullet_{\bx}, \bullet_{\bx}\} \in (\mathcal{L}_0(G(\mathbf{\eta}))/\sim_{\phi_B})}
        \frac{\sprod{\bx}{\bx}}{N}
\right)
\left(
\prod_{\tilde{B} \in \pi^*} 
\prod_{\{\bullet_{\bx}, \bullet_{\bx}\} \in (\mathcal{L}_0(G(\mathbf{\eta}))/\sim_{\phi_{\tilde{B}}})}
\frac{\sprod{\bx}{\bx}}{N}
\right),
\]
the decomposition of $\phi \in \mathcal{P}_{\mathrm{A}}(G(\boldsymbol{\eta}))$ given by (1),(2),(3) yields the following result.
\begin{lemma}\label{lemma:a_decomp}
\begin{equation}
    a(G(\boldsymbol{\eta}))=\sum_{\pi \in \mathrm{NC}_k} \left(\prod_{B=\{j_i\}_{i=1}^l}a({\wedge_{i=1}^l(\eta^{(j_i)}_{1,L}\wedge\eta^{(j_i)}_{2,L})})\right)\left(\prod_{\tilde{B}\in \pi^*_k}a({G(\boldsymbol{\eta}|_{\tilde{B}})})\right).
\end{equation}
\end{lemma}



Before finishing the proof of Theorem \ref{thm:jacobianmoments}, an example illustrating this lemma and the heavy notation that we have introduced for it is well in order. Consider the case when we're computing $m_{4,2}$, for which we end up studying decorated cycles $G(\boldsymbol{\eta})\in \mathrm{DC}(4,2)$. Then the decomposition of $a(G(\boldsymbol{\eta}))$ given in Lemma \ref{lemma:a_decomp} is depicted in full detail in the  figure \ref{fig:fullexample} below.
\begin{figure}[ht]
    \centering
    \input{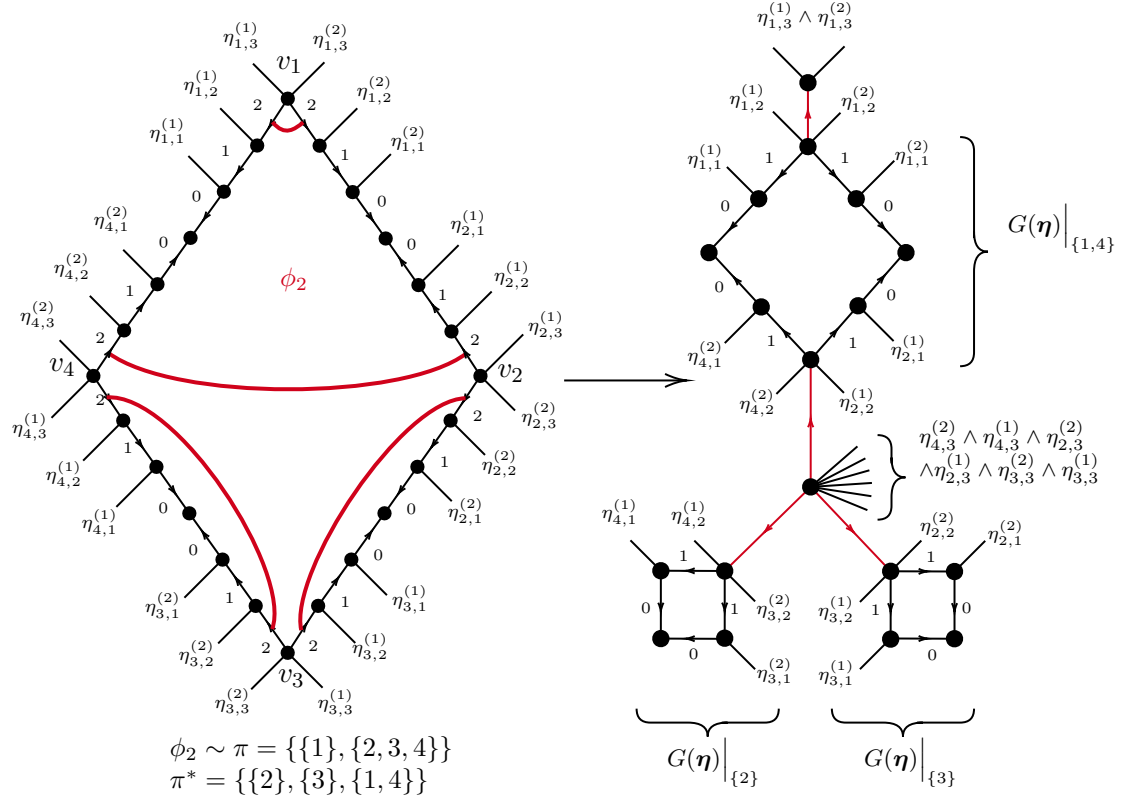}
    \caption{Let $\phi\in \mathcal{P}_{\mathrm{A}}(G(\boldsymbol{\eta}))$ and $\phi_2\subset \phi$ be as in Lemma \ref{lem:firstdecomp} (depicted in red on the left in this figure). Then $\phi_2$ corresponds to a non-crossing partition $\pi$: here, $\pi$ only has two blocks, which gives rise to the trees $\eta_{1,3}^{(1)}\wedge \eta_{1,3}^{(2)}$ and $\eta_{4,3}^{(1)}\wedge\eta_{4,3}^{(2)}\wedge\eta_{2,3}^{(1)}\wedge\eta_{2,3}^{(2)}\wedge\eta_{3,3}^{(1)}\wedge\eta_{3,3}^{(2)}$. Its complement has three blocks, giving rise to three decorated cycles: $G(\boldsymbol{\eta})|_{\{1,4\}}\in \mathrm{DC}(2,1)$, $G(\boldsymbol{\eta})|_{\{2\}}\in \mathrm{DC}(1,1)$ and $G(\boldsymbol{\eta})|_{\{3\}}\in \mathrm{DC}(1,1)$.}
    \label{fig:fullexample}
\end{figure}

We now have the tools to prove Theorem \ref{thm:jacobianmoments}.
\begin{proof}[Proof of Theorem \ref{thm:jacobianmoments}]
Recall that
\[
    \tilde{m}_{k,L}=\sum_{\substack{G(\boldsymbol{\eta})\in \mathrm{DC}(k,L)}} \frac{\dot \varphi_{\boldsymbol{\eta}}}{s(\boldsymbol{\eta})} a(G(\boldsymbol{\eta}))
\]
by Proposition \ref{prop:decorated_cycle_scaling}. We can then use the decomposition given in Lemma \ref{lemma:a_decomp} for each $a(G(\boldsymbol{\eta}))$ in this sum to get
\[
    \sum_{\substack{G(\boldsymbol{\eta})\in \mathrm{DC}(k,L)}}\frac{\dot \varphi_{\boldsymbol{\eta}}}{s(\boldsymbol{\eta})} \sum_{\pi \in \mathrm{NC}_k} \left(\prod_{B=\{j_i\}_{i=1}^l\in \pi}a({\wedge_{i=1}^l(\eta^{(j_i)}_{1,L}\wedge\eta^{(j_i)}_{2,L})})\right)\left(\prod_{\tilde{B}\in \pi^*_k}a({G(\boldsymbol{\eta}|_{\tilde{B}})})\right).
\]
The next step will be to interchange the order of summation. 

To that end, we must first distribute the factor $\frac{\dot \varphi_{\boldsymbol{\eta}}}{s(\boldsymbol{\eta})}$. The same arguments of Proposition \ref{app:prop:expectation_tree_expansion} used to prove Theorem \ref{thm:GPlimit} show that
    \[
    \mu_{k,L} = \sum_{(\eta^{(j)}_{i,L})_{j\in[k], i\in [2]}\in ((\bT_{L}(\bx))^*)^{2k}} 
    \left(\prod_{j=1}^k\frac{\dot\varphi_{\eta^{(j)}_{2,L}}\dot\varphi_{\eta^{(j)}_{1,L}}}{s(\eta^{(j)}_{1,L})s(\eta^{(j)}_{2,L})}
    \right)
    a({\wedge_{j=1}^k (\eta^{(j)}_{1,L}\wedge\eta^{(j)}_{2,L})}).
    \]
Notice that the sum over $\mathrm{NC}_k$ does not depend on a specific $G(\boldsymbol{\eta})$. 
Splitting the sum over $G(\eta)$ into two sums, the first over $(\eta^{(j)}_{i,L})_{j\in[k], i\in [2]}\in ((\bT_{L}(\bx))^*)^{2k}$ and the second over the remaining trees in $\boldsymbol{\eta}$, interchanging the sums then gives
\[
    \sum_{\pi\in \mathrm{NC}_k} \bigg(\prod_{B\in \pi} \mu_{|B|,L}\bigg)\bigg(\prod_{\tilde B\in \pi^*} \sum_{G(\boldsymbol{\eta}|_{\tilde B})} \frac{\dot \varphi_{\boldsymbol{\eta}|_{\tilde B}}}{s(\boldsymbol{\eta}|_{\tilde B})} a(G(\boldsymbol{\eta}|_{\tilde B}))\bigg).
\]
Lastly, we note that
\[
    \sum_{G(\boldsymbol{\eta}|_{\tilde B})} \frac{\dot \varphi_{\boldsymbol{\eta}|_{\tilde B}}}{s(\boldsymbol{\eta}|_{\tilde B})} a(G(\boldsymbol{\eta}|_{\tilde B})) = \sum_{G'\in \mathrm{DC}(|\tilde{B}|,L-1)} \frac{\dot \varphi_{G'}}{s(G')} a(G') = \tilde{m}_{|\tilde{B}|,L-1}.
\]
It follows that
\[
 \tilde{m}_{k,L}=\sum_{\pi\in \mathrm{NC}_k} \mu_{\pi,L} \tilde{m}_{\pi^*,L-1},
\]
and we conclude that $\tilde{m}_{k,L}=m_{k,L}$ for every $k,L$, where $m_{k,L}$ is as in the Proposition's statement.

\end{proof}

\section{Extensions to non-Gaussian, sparse and complex weights.}\label{sec:extensions}

\subsection{Non-Gaussian weights}
\label{sec:non_gaussian}
All of the results obtained in the previous section can be extended to neural networks $\Phi$ with non-Gaussian random weights, under moment assumptions. Consider the following more general sequence $\mathcal{W}$.
\begin{definition}\label{def:non-gaussian-real}
Let $N>0$ and $\mathcal{W}:=(W_i ~|~ i\in \N)$, where the $W_i\in\mathbb{R}^{N\times N}$ are independent matrices with entries  $\sim \sigma_i Z_i$, where $\{Z_i\}_{i\geq 0}$ is a family of i.i.d. random variables satisfying
\[
\E\{Z_i^{2}\} = 1, \quad
\E\{Z_i^{2k+1}\} = 0, \quad 
\E\{Z_i^{2k}\} < \infty\quad \forall k\geq 0.
\] 
\end{definition}
\noindent Note that we take the matrices in this sequence to be square for simplicity only, the results in this section carrying over to rectangular matrices as in Section \ref{sec:genus} by straightforward modifications. The graphs considered in this section are all assumed to satisfy the following condition.
\begin{assumption}\label{assumption:non-gaussian}
 $G$ is a product graph with inputs which are either elements of $\mathcal{W}$, or deterministic vector/matrices with entries uniformly bounded in $N$.
\end{assumption}

We define $E_\mathcal{W}$ and the set of admissible pairings of  $E_\mathcal{W}$ in exactly the same way as in the Gaussian case.


Recall also that for any $G$, $\check{e}(G):=|E(G_{\phi_0})|$ for any fixed admissible pairing $\phi_0$ of its edges (the choice of $\phi_0$ is irrelevant, as the value of $|E(G_{\phi_0})|$ is constant over pairings). 

\begin{theorem}[Approximate Wick expansion]\label{thm:non_gauss_wickexpansion}
 For any (possibly disconnected) product graph $G$
    \begin{equation}\label{eq:non_gauss_wickexpansion}
        \E \left\{\bW_G \right\} =  \sum_{\phi \in \cP(G)} \bW_{G_{\phi}} + \sigma_G \mathcal{O}(N^{\check{e}(G)+c(G)-1})
    \end{equation}
where $\phi$ runs over all of $G$'s \emph{admissible} pairings.
\end{theorem}

\begin{proof}
    Consider the set of indexations $\mathcal{I}^G := \prod_{v \in V(G)} \{1,...,N_{v}\}$. Recall that for each $e=(u,v)\in E$, $i_e=i_ui_v$. 
    Then
    \[
    \bW_G = \sum_{\mathbf{i} \in \mathcal{I}^G} \prod_{c \in C} [\mathbf{X}_c]_{i_c},
    \quad
    \E\{\bW_G\} = \sum_{\mathbf{i}} 
    \left( 
    \prod_{c \in C\setminus E_{\mathcal{W}}} [\mathbf{X}_c]_{i_c} 
    \right)
    \E\left\{ 
    \prod_{e \in E_{\mathcal{W}}} [\mathbf{X}_e]_{i_e} 
    \right\} 
    \]

    Following \cite{Dubach2021} we write $\mathbf{i} \rightsquigarrow \phi$ whenever 
    \[
    \forall \{e,e'\} \in \phi, ~ i_e = i_{e'}
    \]
    and say that an indexation $\mathbf{i}$ is \emph{balanced} whenever there exists a $\phi\in \mathcal{P}(G)$ for which $\mathbf{i} \rightsquigarrow \phi$.

    Note how for \emph{unbalanced} $\mathbf{i}$ one has 
    $\E\left\{ \prod_{e \in E_{\mathcal{W}}} [\mathbf{X}_e]_{i_e} \right\} = 0$ so that we can write the expectation as 
    \begin{align*}
    \E\{\bW_G\} &= \sum_{\phi \in \cP(G)}
    \sum_{\mathbf{i}\rightsquigarrow \phi} 
    \frac{1}{|\{\phi \in \cP(G) ~|~ \mathbf{i}\rightsquigarrow \phi\}|}
    \left( \prod_{c \in C\setminus E_{\mathcal{W}}} [\mathbf{X}_c]_{i_c} 
    \right)
    \E\left\{ \prod_{e \in E_{\mathcal{W}}} [\mathbf{X}_e]_{i_e} \right\} 
    \end{align*}
    
    For simplicity write
    $\Gamma_\phi := \sum_{\mathbf{i}\rightsquigarrow \phi}  \frac{1}{|\{\phi \in \cP(G) ~|~ \mathbf{i}\rightsquigarrow \phi\}|}
    \left( \prod_{c \in C\setminus E_{\mathcal{W}}} [\mathbf{X}_c]_{i_c} 
    \right)
    \E\left\{ \prod_{e \in E_{\mathcal{W}}} [\mathbf{X}_e]_{i_e} \right\} $
    so that 
    \[
    \E\{\bW_G\} = \sum_{\phi \in \cP(G)} \Gamma_\phi.
    \]

    Note that for all $\phi \in \cP(G)$ it holds, for constants depending on $G$ but not on $N$, that 
    \begin{align*}
        \Gamma_\phi 
        = \sum_{\mathbf{i}\rightsquigarrow \phi}  \sigma_G \mathcal{O}(1) 
        = \sigma_G \mathcal{O}(|\{\mathbf{i} ~|~ \mathbf{i}\rightsquigarrow \phi\}|) 
        = \sigma_G \mathcal{O}(N^{|V(G_\phi)|})
    \end{align*}
    Notice that the maximum value of $|V(G_\phi)|$ is reached only when $\phi$ is {fully atomic}, in which case one has 
    \[
    |V(G_\phi)| = \check{e}(G) + c(G)
    \]
    where we recall that $\check{e}=|E(G_\phi)|$
    (otherwise, this is a strict inequality).

    Furthermore, for such $\phi$, we have
    \begin{equation}
        \Gamma_\phi = \sum_{\mathcal{I}^G_{(1)} \ni \mathbf{i}\rightsquigarrow \phi} 
        \left( \prod_{c \in C\setminus E_{\mathcal{W}}} [\mathbf{X}_c]_{i_c} 
        \right)
        \E\left\{ \prod_{e \in E_{\mathcal{W}}} [\mathbf{X}_e]_{i_e} \right\} 
        + \sigma_G \mathcal{O}(N^{\check{e}(G) + c(G)-1})
    \end{equation}
    where 
    \begin{equation*}
        \mathcal{I}^G_{(1)} := 
        \{\mathbf{i}\in \mathcal{I}^G ~|~ \forall e,e'\in E_{\mathcal{W}}, i_e=i_{e'}\implies \mathbf{X}_e=\mathbf{X}_{e'} \text{ and } i_e \neq i_{e_0} \hspace{3pt} \forall e_0 \in E_\mathcal{W}  \setminus \{e,e'\}\},
    \end{equation*}
    noting that if $\mathbf{i}\in \mathcal{I}_{(1)}^G$, there exists exactly one $\phi\in \mathcal{P}(G)$ for which $\mathbf{i}\rightsquigarrow \phi$.

    But in this case, we also have
    \[
    \bW_{G_\phi} = \sum_{\mathcal{I}^G_{(1)} \ni \mathbf{i}\rightsquigarrow \phi} 
        \left( \prod_{c \in C\setminus E_{\mathcal{W}}} [\mathbf{X}_c]_{i_c} 
        \right)
        \E\left\{ \prod_{e \in E_{\mathcal{W}}} [\mathbf{X}_e]_{i_e} \right\} 
        + \sigma_G \mathcal{O}(N^{\check{e}(G) + c(G)-1}),
    \]
    so that we can write 
    \begin{align*}
        \E\{\bW_G\} &= \sum_{\phi \in \cP_{\mathrm{A}}(G)} \Gamma_\phi + \sum_{\phi \in \cP(G)\setminus\cP_{\mathrm{A}}(G)} \Gamma_\phi
        \\
        &= \sum_{\phi \in \cP_{\mathrm{A}}(G)} \Gamma_\phi 
        + \sigma_G \mathcal{O}(N^{\check{e}(G) + c(G)-1})
        \\
        &= \sum_{\phi \in \cP_{\mathrm{A}}(G)} 
        \left( \bW_{G_\phi} + \sigma_G \mathcal{O}(N^{\check{e}(G) + c(G)-1}) \right)
        + \sigma_G \mathcal{O}(N^{\check{e}(G) + c(G)-1})
        \\
        &= \sum_{\phi \in \cP_{\mathrm{A}}(G)} \bW_{G_\phi}
        + \sigma_G \mathcal{O}(N^{\check{e}(G) + c(G)-1})
        \\
        &= \sum_{\phi \in \cP(G)} \bW_{G_\phi}
        + \sigma_G \mathcal{O}(N^{\check{e}(G) + c(G)-1})
    \end{align*}
    where in the last equality we have used the fact that
    \[
    \forall \phi \in \cP(G)\setminus\cP_{\mathrm{A}}(G), ~ 
    \bW_{G_\phi} = \sigma_G \mathcal{O}(N^{\check{e}(G) + c(G)-1}).
    \]
\end{proof}

\begin{proposition}\label{prop:mixedmoments_non_gauss}
Let $G=\sqcup_{i=1}^{c(G)}G_i$ be a product graph with disjoint connected components $G_i$. Then
\begin{align*}
     \E\bigg\{
     \prod_{i=1}^{c(G)} \Big( 
     &\bW_{G_i} - 
    \sum_{\psi \in \mathcal{P}_{\mathrm{A}}(G_{i})}
      \bW_{(G_i)_\psi}
     + \sigma_{G_i}\mathcal{O}(N^{\check{e}(G_i)})\Big)
     \bigg\} \\
     &= 
     \sum_{\phi \in \mathcal{P}_{\mathrm{B}}(G)} \bW_{G_\phi} +  \sigma_G\mathcal{O}(N^{\check{e}(G) +{c(G)}/{2}-1}).
\end{align*}
\end{proposition}

\begin{proof}
    {For any $\phi\in \mathcal{P}_\mathrm{A}(G)$, let $\Gamma_\phi=\mathbf{W}_{G_\phi}+\sigma_{G}\mathcal{O}(N^{\check{e}(G)})$ and define $\Gamma_{\phi_i}$ similarly for each $\phi_i\in \mathcal{P}_\mathrm{A}(G_i)$}. Proceeding as in Proposition \ref{prop:mixedmoments_} we obtain 
    \begin{align*}
        \E\bigg\{
     \prod_{i=1}^{c(G)} \Big( 
     \bW_{G_i} - 
    \sum_{\phi \in \mathcal{P}(G_i)}
      \Gamma_\phi \Big)
     \bigg\} &= \sum_{T \subseteq [c(G)]} 
        (-1)^{|T^c|}
        \sum_{\phi_T \times_{i \in T^c} \phi_i}
        \Big(  \Gamma_{\phi_T}\prod_{i\in T^c} \Gamma_{\phi_i} \Big)
    \end{align*}
    where the last sum is over $\cP(G_T) \times (\prod_{i \in T^c} \mathcal{P}_{\mathrm{A}}(G_i))$.

    Arguing as in the proof of the previous lemma, we find that
    \begin{equation}
        \Gamma_{\phi_T}\prod_{i\in T^c} \Gamma_{\phi_i} =
    \Gamma_{\phi_T \times_{i \in T^c} \phi_i}
    + \mathcal{O}(N^{|V(G_{\phi_T \times_{i \in T^c} \phi_i})|-1})
    \end{equation}
    where the implicit constants do not depend from $N$.

    But then using inclusion exclusion, we obtain on the one hand
    \[
    \sum_{T \subseteq [c(G)]} 
        (-1)^{|T^c|}
        \sum_{\phi_T \times_{i \in T^c} \phi_i}
        \Gamma_{\phi_T \times_{i \in T^c} \phi_i}
    =
    \sum_{\phi \in \tilde{\cP}(G)} \Gamma_\phi
    \]
    where $\tilde{\cP}(G) \subseteq \cP(G)$ is the set of pairings of $G$ such that none of its connected components are paired within themselves, and on the other hand,
    \begin{equation}\label{eq:gammaestimate}
    \sum_{T \subseteq [c(G)]} 
        (-1)^{|T^c|}
        \sum_{\phi_T \times_{i \in T^c} \phi_i} \mathcal{O}(N^{|V(G_{\phi_T \times_{i \in T^c} \phi_i})|})
    = \mathcal{O}\Big(\sum_{\phi \in \tilde{\cP}(G)} N^{|V(G_{\phi})|-1}\Big).
    \end{equation}
    We conclude by arguing as before that the maximal order of the $\Gamma_\phi$ is reached when the partition is {bi-atomic}, which is the only case when $V(G_\phi) = \check{e}(G) + \frac{c(G)}{2}$ and when 
    \[
    \Gamma_\phi = \bW_\phi + \mathcal{O}\Big(N^{\check{e}(G) + {c(G)}/{2}-1}\Big).
    \]
\end{proof}

Reiterating the arguments in Section \ref{sec:application_NNs} then yields the following.
\begin{corollary}\label{cor:nongaussiantheorem}
    Let $\Phi$ be a neural network with weights $W_\ell$ with entries as in definition \ref{def:non-gaussian-real}. Then the conclusions of Theorems \ref{thm:GPlimit}, \ref{thm:ntk} and \ref{thm:jacobianmoments} still hold .
\end{corollary}

\subsection{Sparse weights}
\label{sec:sparse}
Following \cite{Dubach2021}, we can also go further and consider sparse matrices with i.i.d. entries satisfying the same moment assumptions as in \ref{def:non-gaussian-real}.

Consider the sequence $\tilde{\mathcal{W}}$ of matrices $\tilde W_i := W_i \odot B_i$, where $W_i\in \mathcal{W}$ as defined above and the $B_i$ are independent matrices with i.i.d., Bernoulli distributed entries with parameter $p_N$ satisfying $N p_N \to \infty$. Note that
\[
\E\{[\tilde W_i]_{\alpha,\beta}^{2k}\} = p_N\sigma_i^{2k}\E\{Z_i^{2k}\} , \quad
\E\{[\tilde W_i]_{\alpha,\beta}^{2k+1}\} = 0, \quad \forall k\geq 0.
\]
If $G$ is a product graph whose random inputs are from $\tilde{\mathcal{W}}$, we will use $\tilde{\mathbf{X}}_e=\mathbf{X}_e\odot B_i$ to denote such inputs.

\begin{assumption}\label{assumption:sparse}
 $G$ is a product graph with inputs which are either elements of $\tilde{\mathcal{W}}$, or deterministic vector/matrices with entries uniformly bounded in $N$.
\end{assumption}

\begin{proposition}
\label{prop:sparse_wick}
    Let $G$ be a product graph satisfying Assumption \ref{assumption:sparse}. Then
    \begin{equation}\label{eq:sparse_wickexpansion}
        \E \left\{\bW_G \right\} =  \sum_{\phi \in \cP_{\mathrm{A}}(G)} \bW_{G_{\phi}} + \sigma_G p_N^{{|E_{\mathcal{W}}|}/{2}-1} \mathcal{O}(N^{\check{e}(G)+c(G)-1})
    \end{equation}
\end{proposition}

\begin{proof} Consider the set of indexations $\mathcal{I}^G := \prod_{v \in V(G)} \{1,...,N_{v}\}$, and define $\mathcal{I}^G_{(1)}$, balanced and unbalanced indexations as in the proof of theorem \ref{thm:non_gauss_wickexpansion}. 

If $\mathbf{i}\in \mathcal{I}^G$ is {unbalanced}, one has 
    $\E\left\{ \prod_{e \in E_{\mathcal{W}}} [\tilde{\mathbf{X}}_e]_{i_e} \right\} = 0$
    while if $\mathbf{i} \in \mathcal{I}^G_{(1)}$, 
    \[
    \E\left\{ \prod_{e \in E_{\mathcal{W}}} [\tilde{\mathbf{X}}_e]_{i_e} \right\} = \sigma_G p_N^{{|E_{\mathcal{W}}|/2}}.
    \]
    If $\mathbf{i} \notin  \mathcal{I}^G_{(1)}$ is balanced, then
     \[
    \E\left\{ \prod_{e \in E_{\mathcal{W}}} [\tilde{\mathbf{X}}_e]_{i_e} \right\} = \sigma_G p_N^{l} \mathcal{O}(1)
    \]
    for some $1 \leq l < {|E_{\mathcal{W}}|}/{2}$.  
    We can thus partition the set of balanced indexations as $\sqcup_{l=1}^{{|E_{\mathcal{W}}|/2}} \mathcal{J}^G_l$ where 
    $\mathcal{J}^G_l$ is the subset of indexations which result in a factor $p_N^l$.
    Noting that $\mathcal{J}^G_{|E_{\mathcal{W}}|/2} = \mathcal{I}^G_{(1)}$, we can write
    \begin{align*}
    \E\{{\bW}_G\} =&
        p_N^{|E_{\mathcal{W}}|/2}  \sum_{\mathbf{i} \in \mathcal{I}^G_{(1)}}
        \left( 
        \prod_{c \in C\setminus E_{\mathcal{W}}} [\mathbf{X}_c]_{i_c} 
        \right)
        \sigma_G 
        +
        \sum_{l=1}^{|E_{\mathcal{W}}|/2-1} 
        \sum_{\mathbf{i} \in \mathcal{J}^G_l} 
        \left( 
        \prod_{c \in C\setminus E_{\mathcal{W}}} [\mathbf{X}_c]_{i_c} 
        \right)
        \sigma_G p_N^{l} \mathcal{O}(1)
        \\
        =&
        p_N^{|E_{\mathcal{W}}|/2}  \sum_{\mathbf{i} \in \mathcal{I}^G_{(1)}}
        \left( 
        \prod_{c \in C\setminus E_{\mathcal{W}}} [\mathbf{X}_c]_{i_c} 
        \right)
        \sigma_G 
        +
        \sum_{l=1}^{|E_{\mathcal{W}}|/2-1} 
        \sigma_G p_N^{l} \mathcal{O}\Big( |\mathcal{J}^G_l|  \Big)
    \end{align*}
    Note also that $|\mathcal{J}^G_l| = \mathcal{O}\Big( N^{|E\setminus E_{\mathcal{W}}| + c(G) + l }  \Big)$.
    
    In fact, each $\mathbf{i}$ is compatible to some $\phi \in \cP(G)$ by virtue of being balanced, thus any $\mathbf{i}$ takes at most ${c(G) + |E\setminus E_{\mathcal{W}}| + \frac{|E_{\mathcal{W}}|}{2}}$ different values, this being the maximal number of vertices in any $G_\phi$. 
    Note though that, by definition of $\mathcal{J}^G_l$, the indices corresponding to the ${|E_{\mathcal{W}}|}/{2}$ pairs take exactly $l$ different values. Each $\mathbf{i} \in \mathcal{J}^G_l$ thus takes at most $c(G) + |E\setminus E_{\mathcal{W}}| + l$ values, ranging from $1$ to $N$.

    Hence we can write
    \begin{align*}
        \sum_{l=1}^{|E_{\mathcal{W}}|/2-1} 
        \sigma_G p_N^{l} \mathcal{O}\Big( |\mathcal{J}^G_l|  \Big)
        &=  \sigma_G \mathcal{O}\Big( N^{|E\setminus E_{\mathcal{W}}|+c(G)} \max\{Np_N, (Np_N)^{|E_{\mathcal{W}}|/2-1}\} \Big)
        \\ 
        &= \sigma_G p_N^{|E_{\mathcal{W}}|/2-1} \mathcal{O}(N^{\check{e}(G)+c(G)-1}).
    \end{align*}
    The claim then follows from the fact that
    \begin{align*}
        \sum_{\phi \in \cP_{\mathrm{A}}(G)} \bW_{G_{\phi}} 
        &= p_N^{|E_{\mathcal{W}}|/2} 
        \left(
        \sum_{\mathbf{i} \in \mathcal{I}^G_{(1)}}
        \left( 
        \prod_{c \in C\setminus E_{\mathcal{W}}} [\mathbf{X}_c]_{i_c} 
        \right)
        \sigma_G 
        + \mathcal{O}(N^{\check{e}(G)+c(G)-1})
        \right)
        \\ 
        &= p_N^{|E_{\mathcal{W}}|/2} 
        \sum_{\mathbf{i} \in \mathcal{I}^G_{(1)}}
        \left( 
        \prod_{c \in C\setminus E_{\mathcal{W}}} [\mathbf{X}_c]_{i_c} 
        \right)
        \sigma_G 
        + p_N^{|E_{\mathcal{W}}|/2} \mathcal{O}(N^{\check{e}(G)+c(G)-1}),
    \end{align*}
    and that 
    $p_N^{|E_{\mathcal{W}}|/2} \mathcal{O}(N^{\check{e}(G)+c(G)-1}) = p_N^{|E_{\mathcal{W}}|/2-1} \mathcal{O}(N^{\check{e}(G)+c(G)-1})$ (since $p_N \leq 1$).
\end{proof}

Similar arguments prove the centered version of this result, giving
\begin{align}
\label{eqn:sparse_wick_center}
\begin{split}
    \E\bigg\{
     \prod_{i=1}^{c(G)} \Big( 
     \bW_{G_i} - 
    & \sum_{\psi \in \mathcal{P}_{\mathrm{A}}(G_{i})}
      \bW_{(G_i)_\psi}
     + \sigma_{G_i}p_N^{{|E_{\mathcal{W}}(G_i)|}/{2} - 1}\mathcal{O}(N^{\check{e}(G_i)})\Big)
     \bigg\} \\ 
     & = 
     \sum_{\phi \in \mathcal{P}_{\mathrm{B}}(G)} \bW_{G_\phi} +  \sigma_G p_N^{|{E_{\mathcal{W}}(G)|/2} - 1} \mathcal{O}(N^{\check{e}(G) +{c(G)}/{2}-1}).
\end{split}
\end{align}
We may then conclude the following.
\begin{corollary}\label{cor:sparsetheorems}
    Let $\Phi$ be a neural network with weights $\tilde{W}_\ell$ (as defined above) and parameters.
    Then after rescaling standard deviations as $\sigma_\ell \mapsto \frac{1}{\sqrt{p_N}}\sigma_\ell$, the conclusions of Theorems \ref{thm:GPlimit}, \ref{thm:ntk} and \ref{thm:jacobianmoments} still hold.
\end{corollary}

\begin{figure}[h!]
    \centering
    \includegraphics[scale=0.25]{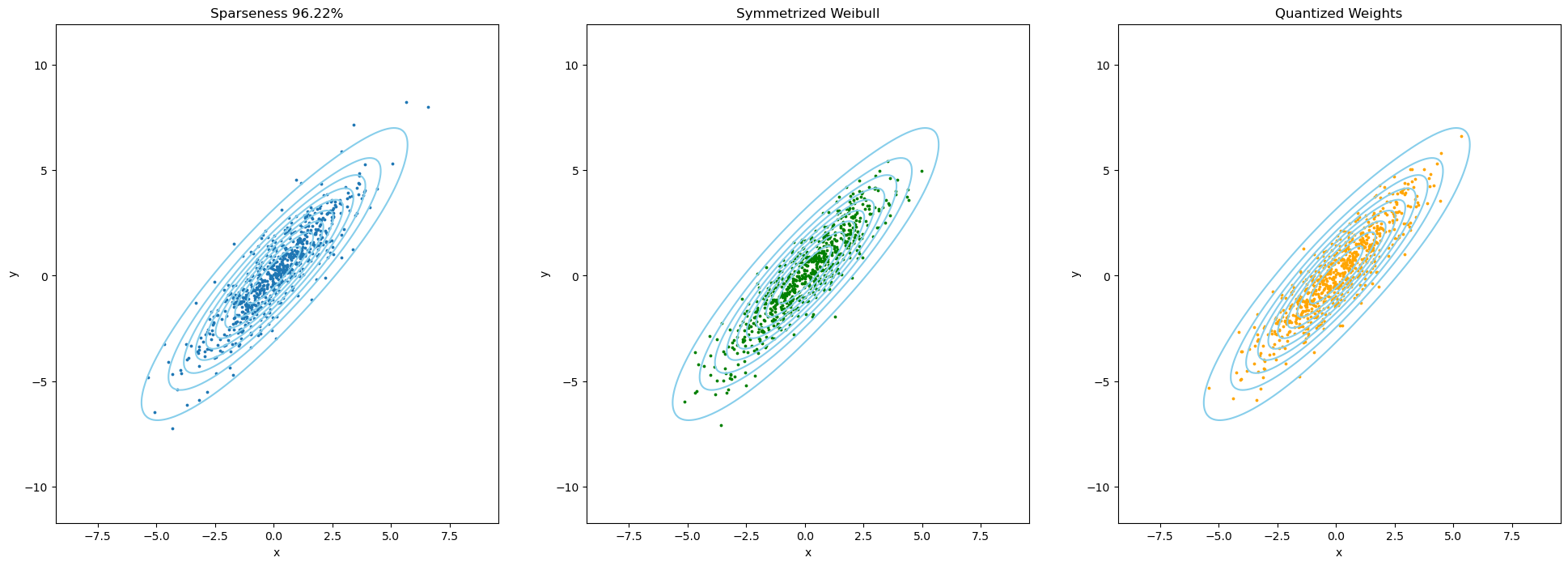}
    \caption{We consider a ReLU network $\Phi_2$ (meaning that the activations are all $x\mapsto \max(x,0)$), where $(N_0, N_1, N_2, N_3) = (2, 700, 700, 1)$. 
    We fix two inputs $\mathbf{x} = (\pi, 1), \mathbf{y}=(e,3) \in \mathbb{R}^2$, and sample the weight matrices $W_0,W_1$ and $W_2$ $700$ times from 4 different ensembles and plot $(\Phi_2(\mathbf{x}), \Phi_2(\mathbf{y}))$. 
    The contour plots in light blue in all three figures serve as a reference, and correspond to $W_0,W_1,W_2$ being Gaussian. 
    The leftmost scatter plot (dark blue) shows the result of multiplying $W_1$ and $W_2$ entry-wise by $p_N^{-1/2}\text{Ber}(p_N)$ for $p_N = \frac{1}{\sqrt{300}}$ (so that roughly $96.22\%$ of their entries are set to $0$). 
    In the middle (green), the $W_i$ have Weibull entries multiplied by a random sign, while on the right (orange) the entries are uniformly distributed in $\{-1,0,1\}$. }
    \label{fig:sparse_weibull_quantized}
\end{figure}

\subsection{The complex case}
\label{sec:complex_case}

We end this section by showing how all of our results (including the extensions to sparse matrices in the previous section) hold for complex-valued matrices. Noting that the definition of product graphs and the extensions in Section \ref{sec:trees} hold \textit{mutatis mutandis} in the complex case, we will once again only need to adapt the arguments of Section \ref{sec:genus}.

\begin{definition}
Let $\mathcal{W}_{\bC}:=(W_i ~|~ i\in \mathbb{Z}\setminus \{0\})$, where for $i>0$, the $W_i\in\mathbb{C}^{N_{r_i}\times N_{c_i}}$ are independent matrices with i.i.d. entries equal to $\sigma_i Z$ where $Z\sim \mathcal{N}_{\mathbb{C}}(0,1)$.
For $i<0$, we set $W_{-i}:=\overline{W}_i$.
\end{definition}

\begin{assumption}\label{assumption1_cmplx} 
$G$ is a product graph whose inputs are either deterministic, or matrices in $\mathcal{W}_\bC$.
\end{assumption}

Recall how a complex random variable $Z \sim \mathcal{N}_{\mathbb{C}}(0,1)$ satisfies the following moment assumptions 
\[
\E[|Z|]=1, \quad n \neq m \implies \E[Z^{n} \bar{Z}^{m}] = 0
\]
where $\bar{Z}$ denotes the conjugate of $Z$.

This means that, in view of Wick's theorem (which takes the same form as in the real case), the correct notion of admissible pairing is given as follows.
\begin{definition}[Complex admissible pairing] 
    Let $G$ be a product graph, and $E_\mathcal{W}=\{e\in E: W_e=W_\ell\text{ for some $\ell(e):=\ell \in \mathbb{Z}\setminus \{0\}$}\}$. 
    Then a pairing $\phi$ of $E_\mathcal{W}$ is said to be \emph{admissible} if any two paired edges $\{e,e'\} \in \phi$ satisfy $\ell(e)=-\ell(e')$. 
    We denote the set of all admissible pairings of $E_\mathcal{W}$ by $\mathcal{P}(G)$.
\end{definition}

At this point notice that the arguments of Sections \ref{sec:genus} follow verbatim, with the exception of Theorem \ref{thm:joint_gaussian_limit} which now takes the following form. 
\begin{theorem}\label{thm:joint_gaussian_limit_cmplx}
    Let $(G_i)_{i\geq 1}$ be a sequence of connected product graphs satisfying Assumptions \ref{assumption1_cmplx} and \ref{assumption:genus_graph} (for $\mathcal{W}_\bC$). Then for any $m\geq 1$,
    \[
        \big((\sigma_{G_i}N^{\check{e}(G_i) + \frac{1}{2}}\big)^{-1}\bW_{G_i} - |\mathcal{P}_\mathrm{A}(G_i)| \sqrt{N})_{i=1}^m \overset{d}{\longrightarrow} \mathbf{Z} = \big(Z_1,...,Z_m\big)^{\top}
    \]
    where the limit is a centered complex Gaussian vector with covariance function  $C := \E[\mathbf{Z}\mathbf{Z}^{\top}]$
    having entries $[C]_{i,j} = |\mathcal{P}_\mathrm{A}(G_i \sqcup G_j)|$ and pseudo covariance $\Gamma := \E[\mathbf{Z}\mathbf{Z}^{H}]$ with entries $[\Gamma]_{i,j} = |\mathcal{P}_\mathrm{A}(G_i \sqcup \bar{G}_j)|$ where $\bar{G}$ is obtained from $G$ by conjugating all the random matrices..
\end{theorem}
As for non-Gaussian complex-valued matrices, we replace the entries of $W_i$ (for $i>0$) with $\sigma_i Z_i$, where $\{Z_i\}_{i\geq 0}$ is now a family of i.i.d. random variables satisfying
\[
\E[|Z|]=1, \quad n \neq m \implies \E[Z^{n} \bar{Z}^{m}] = 0,
\quad \E[|Z|^k]<\infty.
\]
Complex analogues of corollaries \ref{cor:nongaussiantheorem} and \ref{cor:sparsetheorems} then follow straightforwardly.



\bibliographystyle{abbrvnat}
\bibliography{main}

\section{Proofs of auxiliary combinatorial results}
This section proves the combinatorial lemmas of Section \ref{sec:application_NNs}. 
\label{sec:aux}

\subsection{Atomic pairings of trees}
We begin with the following useful lemma, which highlights important features of atomic pairings of edges of trees. It will be used in the proof of Lemmas \ref{lemma:alpha_trees} and \ref{lemma:NTK_bijection} below.

In what follows, for any tree $\tau$, we refer to edges whose head is the root as \emph{root edges} and denote them by $r(\tau)$. Given any $e \in E(\tau)$, we also let $\tau|_e$ denote the sub-tree of $\tau$ rooted at $e$'s head. This new tree $\tau|_e$ then has a unique root edge (which is $e$), and a sub-tree rooted at $e$'s tail. Lastly, recall that admissible pairings only pair edges in $E_\mathcal{W}$ (namely edges with random matrices as inputs), and that an $\ell-$edge in $E_\mathcal{W}$ is one that has been fixed to the random matrix $W_\ell\in \mathcal{W}$.
    
\begin{lemma}\label{prop:tree_movements}
    Let $\tau$ be a connected, directed graph with all vertices having in-degree $1$, except for the root $r_\tau \in V(\tau)$ which has in-degree $0$ (so that $\tau$ is a directed tree). 
    Then for any $\phi \in \cP_{\mathrm{A}}(\tau)$, the following holds:
    \begin{enumerate}
        \item $\phi$ only pairs edges which are equidistant from the root.
        \item Let $e,e'$ be arbitrary edges in $\tau$. If $\{e,e'\}\in \phi$, then $\phi$ is an atomic pairing of the edges in $\tau|_e \wedge \tau|_{e'}$ 
        (the formed by identifying the roots of $\tau|_e $ and $\tau|_{e'}$).
        \item  For any $k\geq 0$, let $e_0,...,e_k$ and $e_0',...,e_k'$ be two sequences of edges forming paths of length $k$ which start at $r_\tau$ (so that $e_i, e_i'$ are at distance $i$ from the root and $e_0, e_0'$ are root edges). Then $$\{e_k,e_k'\}\in \phi \implies \{e_i,e_i'\}\in \phi \text{ or } e_i=e_i' \text{ for all } 0\leq i\leq k.$$
    \end{enumerate}
\end{lemma}
\begin{proof}
    We begin by proving that root edges have to be paired with other root edges: consider a root edge $e = (r_\tau, v) \in E_{\mathcal{W}}(\tau)$ and assume that it is paired to a non-root edge $e' = (u',v')$ by $\phi$.

    Assume first that $e'\in E_{\mathcal{W}}(\tau|_e)$ with $u' \neq r_\tau$. Then consider the image in $(\tau)_\phi$ of the path from $v$ to $v'$ in $\tau$: it starts and ends at the same vertex of  $(\tau)_\phi$, while only traversing the (undirected) edge corresponding to $\{e,e'\}$ once (\emph{cf.} Figure \ref{fig:UpDownandTouchtheGround} left).
    This is a cycle, which contradicts the fact that $\phi\in \mathcal{P}_A(\tau)$. 
    Now assume that $e'\in E_{\mathcal{W}}(\tau)\setminus E_{\mathcal{W}}(\tau|_e)$ with $u'\neq r_\tau$. 
    By what we have just shown, there must be another root edge $\tilde e = (r_\tau, \tilde v)$ such that $e' \in E_{\mathcal{W}}(\tau|_{\tilde e})$, and which is not paired to any edge in $E_{\mathcal{W}}(\tau|_{\tilde e})$. 
    Arguing as above, the image of the path from $r_\tau$ to $u'$ is a cycle in $(\tau)_\phi$ traversing the image of $\tilde e$ only once (\emph{cf.} Figure \ref{fig:UpDownandTouchtheGround} (middle)) , giving the desired contradiction.


    We now prove (2). Assume that $\{e,e'\}\in \phi$, and for contradiction that $\{\varepsilon,\varepsilon'\}\in \phi$ for some $\varepsilon \in E_{\mathcal{W}}(\tau|_e)$ (which is $\neq e$) and $\varepsilon' \in E_{\mathcal{W}}(\tau)\setminus E(\tau|_e \sqcup \tau|_{e'})$. 
    Then consider the path which starts at the head of $\varepsilon'$, goes through $r_\tau$ without using edges in $\tau|_e \sqcup \tau|_{e'}$ and ends at the head of $\varepsilon$ using the edge $e$ (\emph{cf.} Figure \ref{fig:UpDownandTouchtheGround} right). 
    This path's image in $(\tau)|_{\phi}$ is then once again a cycle since it traverses $\{e,e'\}$ only once, we conclude that $\phi|_{\tau|_e \sqcup \tau|_{e'}} \in \cP_{{A}}(\tau|_e \wedge \tau|_{e'})$

    To prove (3) we proceed by induction. Assume that $\{e_k,e_k'\}\in \phi$. If $e_{k-1}=e_{k-1}'$, there is nothing to prove. Otherwise, assume that $e_{k-1}  \notin E_{\mathcal{W}}(\tau)$, so that it is not paired by $\phi$. Then the path which starts at the head of $e_k$, goes to $r_\tau$ through $e_{k-1}$ and ends at the head of $e_k'$ via $e_{k-1}'$ is a cycle in $(\tau)_{\phi}$. We must therefore have  $e_{k-1}, e_{k-1}' \in E_{\mathcal{W}}(\tau)$.
    
    Now if these edges were not paired together by $\phi$, then $\{e_k,e_k'\}\notin \phi$ by (2), giving us a contradiction. This proves (3) by induction on $k$.

    Finally (1) follows from (3) by induction on the distance from the root, using the fact that root edges have to be paired with other root edges.
\end{proof}

 \begin{figure}[ht]
    \centering
    \scalebox{.85}{\input{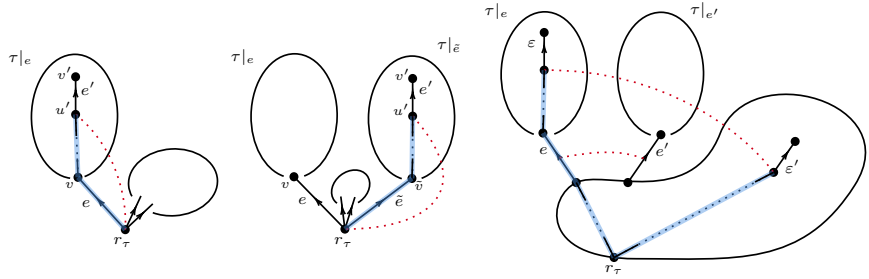}}
    \vspace{-5pt}
    \caption{Paths in $\tau$ which are cycles in $(\tau)_\phi$ in the proof of Proposition \ref{prop:tree_movements}. Red dotted lines denote vertices/edges which are the same in $(\tau)_\phi$, and paths highlighted in blue are cycles in the latter. }
    \label{fig:UpDownandTouchtheGround}
\end{figure}

\subsection{Results in Section \ref{sec:GP}}\label{sec:proofs_lem2&3}
\begin{proof}[Proof of Lemma \ref{lemma:GP_scaling}]
    For the first assertion, we just have to prove that given any $\eta \in \bT_{\ell,k}(\bx)$ one has $\mathcal{P}_{\mathrm{A}}(\eta) = \emptyset$. This is clear, since the tree $\eta$ only has one $\ell-$edge (namely, the one stemming from its root), which then cannot be paired with any other edge. 
    
    For the second assertion, recall that the elements of $\mathcal{L}_0(\sqcup_i \eta_i)$ are the only cells in $C(\sqcup_i\eta_i)$ which are tails of $0$-labeled edges.
    This means that any $\phi \in \cP(\sqcup_i \eta_i)$ identifies these vertices {in pairs} as shown in Figure \ref{fig:leaf_pairing_GP}. It follows that for any pair of vertices $\bullet_{\bx_i}, \bullet_{\bx_j}\in \mathcal{L}_0(\sqcup_i \eta_i)$, if we let ${G}$ denote the product graph obtained from $\sqcup_i \eta_i$ by deleting these vertices along with their adjacent edges (see Figure \ref{fig:leaf_pairing_GP}), then 
        \begin{align}\label{pairupleaves}
        \bW_{(\sqcup_i \eta_i)_\phi} &= 
        \sum_{\alpha,\beta} 
        [\bW_{{G}_{\phi}}]_{\alpha} 
        [\sigma_0^2\mathbf{I}]_{\alpha,\beta} 
        [\bx_i \odot \bx_j]_{\beta} 
        \nonumber \\ &= 
        \sigma_0^2 \sum_{\alpha,\beta} 
        [\bW_{{G}_{\phi}}]_{\alpha} 
        [\bx_i]_{\beta} [\bx_j]_{\beta} 
        = \sigma_0^2 \sprod{\bx_i}{\bx_j}_{\bR^{N_0}} \bW_{{G}_{\phi}}.
    \end{align}

        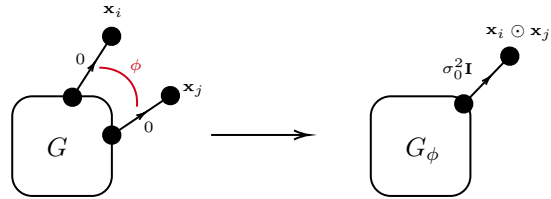
\begin{figure}[ht]
        \centering
        \tikzset{every picture/.style={line width=0.75pt}} 

\begin{tikzpicture}[x=0.75pt,y=0.75pt,yscale=-1,xscale=1]

\draw   (210.62,130.73) .. controls (210.62,125.21) and (215.09,120.73) .. (220.62,120.73) -- (250.62,120.73) .. controls (256.14,120.73) and (260.62,125.21) .. (260.62,130.73) -- (260.62,160.73) .. controls (260.62,166.25) and (256.14,170.73) .. (250.62,170.73) -- (220.62,170.73) .. controls (215.09,170.73) and (210.62,166.25) .. (210.62,160.73) -- cycle ;
\draw  [color={rgb, 255:red, 0; green, 0; blue, 0 }  ,draw opacity=1 ][fill={rgb, 255:red, 0; green, 0; blue, 0 }  ,fill opacity=1 ] (243.79,124.39) .. controls (245.55,122.75) and (245.64,119.99) .. (244,118.23) .. controls (242.36,116.47) and (239.61,116.38) .. (237.85,118.02) .. controls (236.09,119.66) and (235.99,122.42) .. (237.64,124.17) .. controls (239.28,125.93) and (242.03,126.03) .. (243.79,124.39) -- cycle ;
\draw    (240.82,121.2) -- (260.5,90.63) ;
\draw [shift={(251.96,103.9)}, rotate = 122.77] [color={rgb, 255:red, 0; green, 0; blue, 0 }  ][line width=0.75]    (4.37,-1.32) .. controls (2.78,-0.56) and (1.32,-0.12) .. (0,0) .. controls (1.32,0.12) and (2.78,0.56) .. (4.37,1.32)   ;
\draw  [color={rgb, 255:red, 0; green, 0; blue, 0 }  ,draw opacity=1 ][fill={rgb, 255:red, 0; green, 0; blue, 0 }  ,fill opacity=1 ] (263.47,93.81) .. controls (265.23,92.17) and (265.32,89.42) .. (263.68,87.66) .. controls (262.04,85.9) and (259.29,85.81) .. (257.53,87.45) .. controls (255.77,89.09) and (255.68,91.84) .. (257.32,93.6) .. controls (258.96,95.36) and (261.71,95.45) .. (263.47,93.81) -- cycle ;
\draw  [color={rgb, 255:red, 0; green, 0; blue, 0 }  ,draw opacity=1 ][fill={rgb, 255:red, 0; green, 0; blue, 0 }  ,fill opacity=1 ] (263.54,143.39) .. controls (265.3,141.75) and (265.39,138.99) .. (263.75,137.23) .. controls (262.11,135.47) and (259.36,135.38) .. (257.6,137.02) .. controls (255.84,138.66) and (255.74,141.42) .. (257.39,143.17) .. controls (259.03,144.93) and (261.78,145.03) .. (263.54,143.39) -- cycle ;
\draw    (260.57,140.2) -- (290,120.38) ;
\draw [shift={(277.27,128.95)}, rotate = 146.04] [color={rgb, 255:red, 0; green, 0; blue, 0 }  ][line width=0.75]    (4.37,-1.32) .. controls (2.78,-0.56) and (1.32,-0.12) .. (0,0) .. controls (1.32,0.12) and (2.78,0.56) .. (4.37,1.32)   ;
\draw  [color={rgb, 255:red, 0; green, 0; blue, 0 }  ,draw opacity=1 ][fill={rgb, 255:red, 0; green, 0; blue, 0 }  ,fill opacity=1 ] (292.97,123.56) .. controls (294.73,121.92) and (294.82,119.17) .. (293.18,117.41) .. controls (291.54,115.65) and (288.79,115.56) .. (287.03,117.2) .. controls (285.27,118.84) and (285.18,121.59) .. (286.82,123.35) .. controls (288.46,125.11) and (291.21,125.2) .. (292.97,123.56) -- cycle ;
\draw [color={rgb, 255:red, 208; green, 2; blue, 27 }  ,draw opacity=1 ]   (254.33,106.35) .. controls (266.11,106.35) and (274.33,114.58) .. (273.17,125.91) ;
\draw    (310.44,140.13) -- (357.67,140.13) ;
\draw [shift={(359.67,140.13)}, rotate = 180] [color={rgb, 255:red, 0; green, 0; blue, 0 }  ][line width=0.75]    (7.65,-2.3) .. controls (4.86,-0.97) and (2.31,-0.21) .. (0,0) .. controls (2.31,0.21) and (4.86,0.98) .. (7.65,2.3)   ;
\draw   (390.56,130.73) .. controls (390.56,125.21) and (395.03,120.73) .. (400.56,120.73) -- (430.56,120.73) .. controls (436.08,120.73) and (440.56,125.21) .. (440.56,130.73) -- (440.56,160.73) .. controls (440.56,166.25) and (436.08,170.73) .. (430.56,170.73) -- (400.56,170.73) .. controls (395.03,170.73) and (390.56,166.25) .. (390.56,160.73) -- cycle ;
\draw  [color={rgb, 255:red, 0; green, 0; blue, 0 }  ,draw opacity=1 ][fill={rgb, 255:red, 0; green, 0; blue, 0 }  ,fill opacity=1 ] (440.17,127.72) .. controls (441.93,126.08) and (442.03,123.32) .. (440.39,121.57) .. controls (438.74,119.81) and (435.99,119.71) .. (434.23,121.35) .. controls (432.47,122.99) and (432.38,125.75) .. (434.02,127.51) .. controls (435.66,129.27) and (438.42,129.36) .. (440.17,127.72) -- cycle ;
\draw    (437.2,124.54) -- (460.33,100.58) ;
\draw [shift={(450.43,110.83)}, rotate = 133.99] [color={rgb, 255:red, 0; green, 0; blue, 0 }  ][line width=0.75]    (4.37,-1.32) .. controls (2.78,-0.56) and (1.32,-0.12) .. (0,0) .. controls (1.32,0.12) and (2.78,0.56) .. (4.37,1.32)   ;
\draw  [color={rgb, 255:red, 0; green, 0; blue, 0 }  ,draw opacity=1 ][fill={rgb, 255:red, 0; green, 0; blue, 0 }  ,fill opacity=1 ] (463.3,103.76) .. controls (465.06,102.12) and (465.16,99.36) .. (463.52,97.61) .. controls (461.88,95.85) and (459.12,95.75) .. (457.36,97.39) .. controls (455.6,99.03) and (455.51,101.79) .. (457.15,103.55) .. controls (458.79,105.31) and (461.55,105.4) .. (463.3,103.76) -- cycle ;

\draw (226.51,140) node [anchor=north west][inner sep=0.75pt]    {${G}$};
\draw (254.89,74.01) node [anchor=north west][inner sep=0.75pt]  [font=\tiny,rotate=-0.89]  {$\mathbf{x}_{i}{}{}$};
\draw (241,98) node [anchor=north west][inner sep=0.75pt]  [font=\tiny,rotate=-0.89]  {$0$};
\draw (275.25,131.93) node [anchor=north west][inner sep=0.75pt]  [font=\tiny,rotate=-0.89]  {$0$};
\draw (294.89,112.76) node [anchor=north west][inner sep=0.75pt]  [font=\tiny,rotate=-0.89]  {$\mathbf{x}_{j}{}{}$};
\draw (268.17,100.23) node [anchor=north west][inner sep=0.75pt]  [font=\tiny,color={rgb, 255:red, 208; green, 2; blue, 27 }  ,opacity=1 ]  {$\phi $};
\draw (406.44,140) node [anchor=north west][inner sep=0.75pt]    {${G}_{\phi }$};
\draw (446.83,82.89) node [anchor=north west][inner sep=0.75pt]  [font=\tiny,rotate=-0.89]  {$\mathbf{x}_{i} \odot \mathbf{x}_{j}$};

\draw (425,100) node [anchor=north west][inner sep=0.75pt]  [font=\tiny,rotate=-0.89]  {$\sigma_0^2\mathbf{I}$};

\end{tikzpicture}
        \caption{Any admissible edge pairing induces a pairing of leaves.}
        \label{fig:leaf_pairing_GP}
    \end{figure}

    A similar argument holds for every pair of roots (heads of $\ell-$edges) in $\sqcup_i \eta_i$, of which there are $c(\sqcup_i \eta_i)$ many. Repeating \eqref{pairupleaves} for each pair of leaves and roots then yields the following equality for any $\phi \in \cP(\sqcup_i \eta_i)$:
    \begin{equation}
        \bW_{(\sqcup_i \eta_i)_\phi} = \sigma_0^{|\mathcal{L}_0(\sqcup_i \eta_i)|} \sigma_{\ell}^{c(\sqcup_i \eta_i)}
        \sprod{\bx}{\bx}_{\phi} \delta_\phi
        \bW_{G^1_{\phi}}
    \end{equation}
    where $G^1$ is obtained from $\sqcup_i \eta_i$ by removing all leaves, roots and their adjacent edges, and 
    \[
    \sprod{\bx}{\bx}_{\phi} := \prod_{\{\bullet_{\bx_i}, \bullet_{\bx_j}\} \in (\mathcal{L}_0(\sqcup_i\eta_i)/\sim_{\phi})} \sprod{\bx_i}{\bx_j}_{\bR^{N_0}}.
    \]
    Note that here, $\delta_\phi$ arises from the roots in the same way as $\sprod{\bx}{\bx}_{\phi}$ did from the leaves, except that since the roots all have inputs $\mathbf{e}_i$ for some $i$, the resulting inner product must either equal zero or one.

    Further, note that under Assumption \ref{assumption:GP_lim}, $G^1$ is a product graph satisfying Assumption \ref{assumption:genus_graph}, in particular a product graph with all vertices having dimension $\fd = N$, fixed to $\mathbf{1}_N$ and with all edges fixed to random Gaussian matrices ($E\setminus E_{\mathcal{W}}(G^1) = \emptyset$).
    We can thus leverage the genus expansion results of Section \ref{sec:genus_actual} to compute the values $\bW_{G_\phi}$ explicitly. 
    
    Note that any $\phi \in \cP_{\mathrm{AF}}(\sqcup_i \eta_i)$ must induce an atom-free partition of the edges of $G^1$, {since each of the removed edges shares a vertex with $G^1$}. By Lemma \ref{lemma:mixedmoments_genus} we know that $\bW_{G^1_\phi}$ is maximized when this induced partition is bi-atomic, in which case 
    \[
    \bW_{G^1_\phi} = \sigma_{G{\color{blue}^1}} N^{(|E(G^1)|+c(G^1))/2}.
    \]
    This yields the claim, since $\sigma_{\sqcup_i\eta_i} = \sigma_0^{|\mathcal{L}_0(\sqcup_i \eta_i)|} \sigma_{\ell}^{c(\sqcup_i \eta_i)} \sigma_{G^1}$,
    \[
    |E(G^1)| = |E(\sqcup_i\eta_i)| - |\mathcal{L}_0(\sqcup_i\eta_i)| - c(\sqcup_i\eta_i), 
    \quad
    c(G^1) = c(\sqcup_i\eta_i),    
    \]
    and  
    \[
        \{\psi\in \mathcal{P}_{\mathrm{B}}(G{^1}): \text{$\psi$ is induced by some $\phi = \mathcal{P}_{\mathrm{B}}(\sqcup_i \eta_i)$}\} = \mathcal{P}_{\mathrm{B}}(G{\color{blue}^1}).
    \]
\end{proof}

\begin{proof}[Proof of Lemma \ref{lemma:alpha_trees}]
    Let $\phi\in \mathcal{P}_\mathrm{B}(\tau^{(i)}\sqcup \eta^{(j)})$ be arbitrary. Denote by $r(\tau)$ (resp. $r(\eta)$) denote the unique edge whose head is the root of $\tau$ (resp. $\eta$). Then $r(\tau)$ and $r(\eta)$ are the only $\ell-$edges in their respective trees, and must thus be paired by $\phi$. 
    Since $\phi$ must then pair $\ell-1$-labeled edges with each other (which are precisely the edges $\{r(\tau_{k}^{(i)})\}_{k=1}^n$, $\{r(\eta_{k}^{(j)})\}_{k=1}^m$) this induces a pairing $\pi$ of the trees in
    \[
        (\tau_1^{(i)},...,\tau_n^{(i)},\eta_1^{(j)},\cdots,\eta_m^{(j)})
    \]
    via the bijection $\tau^{(i)}_{k}\mapsto r(\tau_k^{(i)})$ (and similarly for the $\eta^{(j)}_k$).

By Lemma \ref{prop:tree_movements}, (1) and (2), for each $\{\pi_1,\pi_2\}\in \pi$, $\phi$ contains a bi-atomic pairing of the edges of $\pi_1\sqcup \pi_2$, which we denote by $\phi\big|_{\pi_1\sqcup \pi_2}$ ($\in \mathcal{P}_{\mathrm{B}}(\pi_1\sqcup \pi_2)$). In other words, we can decompose $\phi$ as
   \begin{equation}\label{eq:split}
        \phi = \big\{\{r(\tau), r(\eta)\}\big\} \bigsqcup_{\{\pi_1,  \pi_2\} \in \pi} \phi|_{\pi_1 \sqcup \pi_2},    
    \end{equation}
    and the lemma then follows since we can write  
        \[
    \sprod{\bx}{\bx}_{\phi} = \prod_{\{\pi_1,  \pi_2\} \in \pi} \sprod{\bx}{\bx}_{\phi|_{\pi_1 \sqcup \pi_2}}.
    \] 
\end{proof}

\subsection{Results in Section \ref{sec:NTK}}\label{sec:proofs_lem4&5}

\begin{proof}[Proof of Lemma \ref{lemma:L2}]
    By Proposition \ref{prop:mixedmoments_} we know, setting $G := \tau \circ \eta^{\top}$, that
    \begin{equation}
        \E\bigg\{
     \Big( 
     \bW_{G} - 
    \sum_{\psi \in \mathcal{P}_{\mathrm{A}}(G)}
      \bW_{G_\psi}
     \Big)^2
     \bigg\} = 
     \sum_{\phi \in \mathcal{P}_{\mathrm{AF}}(G\sqcup G)} \bW_{(G \sqcup G)_\phi}.
    \end{equation}
    
    We begin by proving the claim for $0 < \ell < L\neq 0$.  
    We will be reasoning as we did to prove Lemma \ref{lemma:GP_scaling}, namely by analyzing the structure of a pairing in $\mathcal{P}_{\mathrm{A}}(G)$. To that end, note that the elements of $\mathcal{L}_0(G)$ are the only cells of $G$ which are tails of $0$-labeled edges, while the roots of $\tau$ and $\eta$ are the only heads of $L$-labeled edges. 

    It follows that for any $\psi \in \mathcal{P}_{\mathrm{A}}(G)$, these vertices will necessarily all get paired, yielding
    \[
    \bW_{G_{\psi}} = \sigma_0^{|\mathcal{L}_0(G)|} \sigma_L^{2} \sprod{\bx}{\bx}_{\psi} \delta_{i_1,i_2} \bW_{G^1_{\psi}}
    \]
    where the product graph $G^1$ is obtained from $G$ by removing all leaves, roots and their adjacent vertices.
    Under Assumption \ref{assumption:NTK}, $G^1$ is then a product graph satisfying Assumption \ref{assumption:genus_graph}: all of its vertices have dimension $\fd = N$ and fixed to $\mathbf{1}_N$, all edges except one are fixed to random Gaussian matrices, the remaining one is fixed to $\mathbf{I}$. We thus apply the genus expansion formula (Corollary \ref{thm:wickexpansion_ones}) with 
    \[
    |E_{\mathcal{W}}(G^1)| = |E(G^1)|- 1 = |E(G)| - ( |\mathcal{L}_0(G)| + 2) - 1,
    \]
    \[
    |E\setminus E_{\mathcal{W}}(G^1)| = 1,
    \]
    and since any atomic partition of $G$ induces an atomic one on $G^1$, we conclude that
    \begin{align*}
    &\bW_{G^1_{\psi}} = \sigma_{G^1} N^{\frac{1}{2}{(|E(G)| - |\mathcal{L}_0(G)| - 3)} + 2},\\
    &\sum_{\psi \in \mathcal{P}_{\mathrm{A}}(G)}
      \bW_{G_\psi} = \sigma_G N^{\frac{1}{2}({|E(G)| - |\mathcal{L}_0(G)|+1})} a(G).
    \end{align*}
    The same reasoning applied to $\cP_{\mathrm{AF}}(G\sqcup G)$ gives  

    \begin{align*}
        &\bW_{(G\sqcup G)_{\phi}} = \sigma_{G}^2 N^{|E(G)| - |\mathcal{L}_0(G)| } \sprod{\bx}{\bx}_{\phi} \delta_\phi \text{\quad\quad(for $\phi \in \cP_{\mathrm{B}}(G\sqcup G)$),}\\
        &\sum_{\phi \in \mathcal{P}_{\mathrm{AF}}(G\sqcup G)} \bW_{(G \sqcup G)_\phi}
        =
        \sigma_{G}^2 N^{|E(G)| - |\mathcal{L}_0(G)| }
        \left(
        \alpha(G\sqcup G) + \mathcal{O}\Big(\frac{1}{N}\Big) 
        \right).
    \end{align*}

    We now turn to the cases when $\ell=0$ or $\ell=L$, noting that the pruned graphs do not contain any $\mathbf{I}$ this time. 
    In the $\ell=L$ case, we get
    \begin{align*}
    &\bW_{G_\psi}\asymp \sigma_G N^{\frac{1}{2}{(|E(G)| - (|\mathcal{L}_0(G)|+1))}+1} = \sigma_G N^{\frac{1}{2}{(|E(G)| - |\mathcal{L}_0(G)|+1))}}, \\
    &\bW_{(G\sqcup G)_\phi}\asymp \sigma_G^2 N^{\frac{1}{2}(2|E(G)| - 2(|\mathcal{L}_0(G)| + 1)) + 1}=   \sigma_G^2 N^{|E(G)| - |\mathcal{L}_0(G)| },
    \end{align*}
    respectively, while for $\ell=0$ we get
    \begin{align*}
    &\bW_{G_\psi}\asymp \sigma_G N^{\frac{1}{2}({|E(G)| - (|\mathcal{L}_0(G)|+2)})+1} = \sigma_G N^{\frac{1}{2}({|E(G)| - |\mathcal{L}_0(G)|})}, \\
    &\bW_{(G\sqcup G)_\phi} \asymp \sigma_G^2 N^{\frac{1}{2}({2|E(G)| - 2(|\mathcal{L}_0(G)| + 2)}) + 1} =   \sigma_G^2 N^{|E(G)| - |\mathcal{L}_0(G)|  - 1}.
    \end{align*}

    Putting everything together, for any $\ell$ and under Assumption \ref{assumption:GP_lim} (in which case $ \sigma_G = N^{-\frac{1}{2}{(|E(G)| - |\mathcal{L}_0(G)|-\mathbf{1}(\ell>0)})}$), we always obtain
    \begin{equation}
        \E\bigg\{
     \Big( 
     \sigma_{\ell}^2 \bW_{G} - 
    a(G)
     \Big)^2
     \bigg\} = 
    \frac{1}{N}\left( \alpha(G\sqcup G) + \mathcal{O}\Big(\frac{1}{N}\Big)\right).
    \end{equation}
\end{proof}

\begin{proof}[Proof of Lemma \ref{lemma:NTK_bijection}] Each tree in $\bT_{L,i}(\bx)$ has exactly one $L-$edge, namely the unique edge adjacent to its root. If $\ell=L$, the canonical bijection in the lemma's statement is the map which assigns to $\eta \in \partial_L \bT_{L,i}(\bx)$ the tree $\eta^+$ obtained by fixing $\eta$'s root edge to $W_L$.
Since each bi-atomic pairing of $\tau^+ \sqcup \eta^+$ has to necessarily pair the root edges, 
it follows that $\cP_{\mathrm{B}}(\tau^+ \sqcup \eta^+)$ and $\cP_{\mathrm{A}}(\tau \circ \eta^{\top})$ are in canonical bijection, and in particular that $a(\tau \circ \eta^{\top}) = \alpha(\tau^+ \sqcup \eta^+)$. 
\begin{figure}[!ht]
    \centering
    \scalebox{0.85}{\input{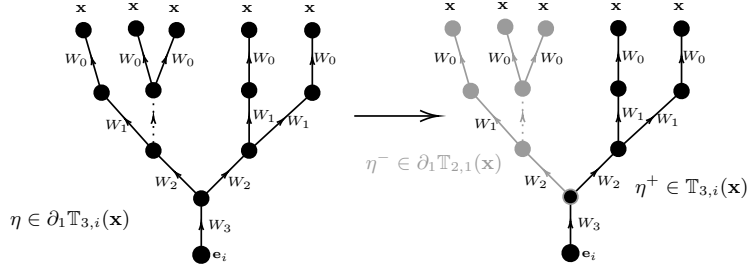}}
    \vspace{-10pt}
    \caption{Decomposition of a tree in $\partial_1\bT_{3,i}$}
    \label{fig:tree_decomposition_NTK}
\end{figure}

If $\ell < L$, then any $\eta \in \partial_{\ell}\bT_{L,i}(\bx)$ can be viewed as a tree $\tau^- \in \partial_{\ell} \bT_{L-1,1}(\bx)$, which is attached by its root to a tree $\tau^+ \in \bT_{L,i}(\bx)$ as in the example of Figure \ref{fig:tree_decomposition_NTK}. Note that this decomposition is unique. It thus remains to show that $\mathcal{P}_\mathrm{A}(\tau\circ \eta^\top)$ is in bijection with $ \cP_{\mathrm{B}}(\tau^+ \sqcup \eta^+) \times \cP_{\mathrm{A}}(\tau^- \circ (\eta^-)^{\top})$.

Let $M\in \mathbb{R}^{N\times N}$ be a random matrix which is independent from the $\{W_\ell\}_\ell$, and is populated with i.i.d. standard Gaussian random variables. Consider the trees $\tau[M]$, $\eta[M]$, which we recall are obtained from $\tau$ and $\eta$ by inserting $M$ into their unique free in-edge, respectively. 

Clearly, $\mathcal{P}_{\mathrm{A}}(\tau\circ \eta^\top)$ is then in bijection with $\mathcal{P}_{\mathrm{A}}(\tau[M]\wedge \eta[M])$, since any pairing in the latter must necessarily pair the edges given $M$ as input, as well as the root edges $r(\tau[M])$ and $r(\eta[M])$ (since they are the only $L$-labeled edges).

We can now use Lemma \ref{prop:tree_movements}, (3). Fix any pairing $\phi\in \mathcal{P}_\mathrm{A}(\tau[M]\wedge\eta[M])$. In $\tau[M]$, the path starting at the root and ending at the head of the unique $M$-labeled edge must be paired by $\phi$ with the analogous path in $\eta[M]$, meaning that edges which are equidistant from the respective roots are paired with each other. In particular, this will end up pairing $r(\tau^-[M])$ with $r(\eta^-[M])$. 

Furthermore, by Lemma \ref{prop:tree_movements}, (2), $\phi$ will only pair edges of $\tau^-[M]$ with edges of $\eta^-[M]$ (and similarly for $\tau^+, \eta^+$ as a consequence). This implies that $\phi$ can be decomposed into the union of some $\phi^-\in \mathcal{P}_\mathrm{A}(\tau^-[M]\wedge \eta^-[M])$ with some $\phi^+\in \mathcal{P}_{\mathrm{A}}(\tau^+\wedge \eta^+)$. 

Repeating this argument for the unique $\tilde{\phi}\in\mathcal{P}_\mathrm{A}(\tau\circ\eta^\top)$ corresponding to $\phi$, we conclude that $\mathcal{P}_\mathrm{A}(\tau\circ\eta^\top)$ is in bijection with
\[
    \mathcal{P}_\mathrm{A}(\tau^+\sqcup \eta^+)\times\mathcal{P}_\mathrm{A}(\tau^-\circ (\eta^-)^\top),
\]
and it follows that 
$$a(\tau \circ \eta^{\top}) = \alpha(\tau^+ \sqcup \eta^+) a(\tau^- \circ {\eta^-}^{\top}).$$
\end{proof}

\subsection{Results in Section \ref{sec:jacobian}}\label{proofs_prop12}
\begin{proof}[Proof of Proposition \ref{prop:decorated_cycle_scaling}]
     Let $G=G(\boldsymbol{\eta})$ be a decorated cycle. We first show that for any admissible pairing $\phi\in \mathcal{P}(G)$, $\mathbf{W}_{G_\phi}=\mathcal{O}(\sigma_G N^{|V(G_\phi)|})$.
     Indeed, for any such $\phi$, we have
    \[
        \bW_{G_\phi} = \sigma_G \prod_{\pi \in V(G)/\sim_{\phi}}{\sum_{\alpha=1}^N [ \odot_{\bullet_\by \in \pi} \by]_{\alpha}}
        = \sigma_G N^{|V(G_\phi)|}\prod_{\pi \in V(G)/\sim_{\phi}} \frac{\sum_{\alpha=1}^N [ \odot_{\bullet_\by \in \pi} \by]_{\alpha}}{N},
    \]
    and as argued in the proof of Lemma \ref{lemma:GP_scaling} (see Figure \ref{fig:leaf_pairing_GP}), the only vertices with input $\bx$ in $G$ are of in-degree $1$. All other vertices have input  $\mathbf{1}$. 
    
    For any $\pi \in V(G)/\sim_\phi$, we therefore have
    \[
    \odot_{\bullet_\by \in \pi} \by \in \{ \mathbf{1}, \bx, \bx\odot\bx\},
    \]
    so that 
    \[
    \frac{\sum_{\alpha=1}^N [ \odot_{\bullet_\by \in \pi} \by]_{\alpha}}{N} \in \bigg\{1, \frac{\sum_{\alpha=1}^N [\bx]_\alpha}{N}, \frac{\sprod{\bx}{\bx}}{N}\bigg\}.
    \]
    Since ${\sprod{\bx}{\bx}}/{N} \to_{N\to\infty} x$ and $$\frac{\sum_{\alpha=1}^N [\bx]_\alpha}{N} \leq \frac{\norm{\bx}_1}{N} \leq \frac{\sqrt{N}\norm{\bx}_2}{N} = \frac{\norm{\bx}_2}{\sqrt{N}} \to_{N\to\infty} \sqrt{x},$$ we can conclude that 
    \[
        \bW_{G_\phi}\leq C_x\sigma_G N^{|V(G_\phi)|}
    \]
    for some constant $C_x$ depending on $x$. The order of $\mathbf{W}_{G_\phi}$ is therefore determined by the number of vertices in $G_\phi$.

    Following Section \ref{sec:genus_actual}, it suffices to prove that
    \[
    \sum_{\psi \in \cP_{\mathrm{A}}(G)} \bW_{G_\psi} = \sigma_G N^{|E(G_\psi)|+1}a(G),
    \]
    which is true in the event that
    \[
    \prod_{\pi \in V(G)/\sim_{\psi}} \frac{\sum_{\alpha=1}^N [ \odot_{\bullet_\by \in \pi} \by]_{\alpha}}{N} = 
        \prod_{\{\bullet_{\bx}, \bullet_{\bx}\} \in (\mathcal{L}(G)/\sim_{\psi})}
        \frac{\sprod{\bx}{\bx}}{N}
    \]
    holds for every $\psi\in \mathcal{P}_{\mathrm{A}}(G)$.

    The latter is equivalent to proving that the edges whose input is equal to $\bx$ are paired between themselves.
    
    Note that on the one hand, these are $0$-labeled edges. On the other, $0$-labeled edges arise in one of two ways. They are either:
    \begin{enumerate}
        \item Edges whose tails are leaves of $G$ (meaning vertices for which the sum of the in and out-degree equals one), and which therefore have input $\mathbf{x}$, or
        \item Edges belonging to the cycle $C(\boldsymbol{\eta})$.
    \end{enumerate} 
    Since $\psi$ is atomic, the result then follows by Lemma \ref{lem:firstdecomp}, as it prohibits $\psi$ form pairing edges inside of $C(\boldsymbol{\eta})$ with edges outside of it in $G$.
\end{proof}
\section{Technical Lemmas}
\label{app:sect:TreeStuff}

The goal of this section is to prove Proposition \ref{app:prop:expectation_tree_expansion}, which is a key ingredient in the proof of Theorem \ref{thm:tree_exp} (in Section \ref{thm:tree_exp}) and Theorem \ref{thm:GPlimit}  (Section \ref{sec:application_NNs}). We begin with the following technical lemma. 

\begin{lemma}\label{lem:summability}
    Consider a sequence of sets $(\cA_i ~|~ i \in \N_{>0})$ and absolutely summable sequences  of vectors $(v_{a_i} ~|~ a_i\in \cA_i) \subseteq \bR^d$ indexed by them. Let $v_{(a_1,\dots,a_M)} := \otimes_{i=1}^M v_{a_i}$  with the convention that empty products have value $1$. 
    Then for $M \in \N$, 
    \begin{equation}
        (\bR^d)^{\otimes M} \ni \otimes_{i=1}^M \left( \sum_{a_i \in \cA_i} v_{a_i} \right)
        =
       \sum_{\boldsymbol{a}_M \in \times_{i=1}^M\cA_i} v_{\boldsymbol{a}_M}.
    \end{equation}
\end{lemma}
\begin{proof}
We proceed by induction on $M$.
If $M=0$ then the result holds by convention.
For the inductive case assume the lemma true up to $M$, then
\begin{align*}
    \otimes_{i=1}^{M+1} \sum_{a_i \in \cA_i} v_{a_i}
    =&
    \left( \otimes_{i=1}^M \sum_{a_i \in \cA_i} v_{a_i} \right)
    \otimes \left( \sum_{a_{M+1} \in \cA_{M+1}} v_{a_{M+1}} \right)
    \\
    =&
    \left(  \sum_{\boldsymbol{a}_M \in \times_{i=1}^M \cA_i} v_{\boldsymbol{a}_M} \right)
    \otimes
    \left( \sum_{a_{M+1} \in \cA_{M+1}} v_{a_{M+1}} \right)
    \\
    =&
   \sum_{a_{M+1} \in \cA_{M+1}}  
   \sum_{\boldsymbol{a}_M \in \times_{i=1}^M \cA_i}  
   v_{\boldsymbol{a}_M} \otimes v_{a_{M+1}}
\end{align*}
Letting $\boldsymbol{a}_M = (a_1,\dots,a_M)$, we have 
$v_{\boldsymbol{a}_M} \otimes v_a = v_{(a_1,\dots,a_{M+1})}$
and the claim follows.
\end{proof}
 
The next step will be to prove an analogous result with Hadamard products $\odot$ instead of tensor products $\otimes$. 
Given that the Hadamard product is commutative, it will be convenient to write the results in terms of multisets (meaning sets where elements can appear more than once) since they are unordered.

Given a set $\cA$, let $\mathbb{X}_{\cA}^M$ be the set of multisets of elements of $\mathcal{A}$ of cardinality $M$, and $\bX_{\mathcal{A}} = \bigcup_{M=0}^{\infty} \bX_{\mathcal{A}}^M$. 
    For any multiset $\tau = \llbracket(a_1)^{k_1} \cdots (a_N)^{k_N}\rrbracket \in \bX_{\mathcal{A}}$, where $a_1,\dots,a_N$ are distinct elements of $\cA$, its \emph{symmetry factor} is defined recursively as
    \begin{equation}\label{eq:multiset_symm}
        \mathfrak{s}(\llbracket~\rrbracket)=1, \quad \mathfrak{s}(\tau) := \prod_{i=1}^N (k_i)!.
    \end{equation}

Using these definitions, we can express Hadamard products of sums of vectors in the following compact form.

\begin{proposition}\label{prop:symmetric_expansion}
    Consider a set $\cA$ and a summable sequence $(v_{a})_{a\in \cA} \subseteq \bR^d$ indexed by $\cA$. We have for any $M\in \mathbb{N}$
    \begin{equation}
        (\bR^d)^{\otimes M} \ni \left( \sum_{a \in \cA} v_a \right)^{\odot M}
        =
       \sum_{\tau \in  \bX_{\cA}^M} \frac{M!}{\mathfrak{s}(\tau)} v_{\tau}
    \end{equation}
    where for $\tau := \llbracket a_1 \cdots a_M \rrbracket$ we define $v_{\tau} := v_{a_{1}} \odot \cdots \odot v_{a_M}$.
\end{proposition}

\begin{proof}    
    Considering $M$ copies of $\cA$ as different sets and using Lemma \ref{lem:summability}, we have on the one hand that
    \[
    \otimes_{i=1}^M \left( \sum_{a \in \cA} v_a \right)
    = 
    \sum_{\boldsymbol{a} \in \times_{i=1}^M \cA}  \otimes_{i=1}^M 
    ( v_{a_i} ).
    \]
    On the other hand if we define linearly
    \[
    \bigodot : (\bR^d)^{\otimes M} \to \bR^d, \quad v_1 \otimes \cdots \otimes v_M \mapsto v_1 \odot \cdots \odot v_M
    \]
    we have
    \begin{align*}
    \left( \sum_{a \in \cA} v_a \right)^{\odot M}
    =&
    \bigodot \left( \sum_{a \in \cA} v_a \right)^{\otimes M}
    \\
    =&
    \bigodot \sum_{\boldsymbol{a} \in \cA^M} 
    \otimes_{i=1}^M 
    (v_{a_{i}} )
   =
    \sum_{\boldsymbol{a} \in \cA^M}
    \odot_{i=1}^M ( v_{a_i} )
    \end{align*}
    by definition. 
    Note that for any $\boldsymbol{a} = (a_1, \dots, a_M), \boldsymbol{a'} = (a_1', \dots, a_M') \in \cA^M$, 
    \[
        \odot_{i=1}^M ( v_{a_i} )=\odot_{i=1}^M ( v_{a'_i} )
    \]
    whenever
    \[
        (a_1, \dots, a_M) = (a_{\sigma(1)}',\dots,a_{\sigma(M)}').
    \]
    for some $\sigma\in S_M$ \emph{i.e.} $\odot$ is invariant to the action of $S_M$.
    Therefore we can write
    \[
    \left( \sum_{a \in \cA} v_a \right)^{\odot M}
    = 
    \sum_{\tau \in  \bX_{\cA}^M} \frac{M!}{\mathfrak{s}(\tau)} v_{\tau}.
    \]
    since the cardinality of orbit $\tau$ is exactly $\frac{M!}{\mathfrak{s}(\tau)}$.
\end{proof}


We are now equipped to prove Proposition \ref{app:prop:expectation_tree_expansion}, which we do using the previous proposition and Wick's theorem. This proposition gives an expansion for covariances of type $\mathbb{E}[\varphi(X)\psi(Y)]$, where $\varphi, \psi$ are polynomials and $(X,Y)$ is a Gaussian vector with covariance of the form \eqref{eq:weird_covariance}. 
\begin{proposition}\label{app:prop:expectation_tree_expansion}
    Let $\varphi,\psi : \bR \to \bR$ be {polynomial functions}. 
    Fix index sets $\cA_1, \cA_2$ and  $\lambda: (\cA_1\cup\cA_2)\times(\cA_1\cup\cA_2) \to \bR$ a symmetric and \emph{summable} function.
    Then if
    \begin{equation}\label{eq:weird_covariance}
    (X,Y) \sim \mathcal{N}\left(0,\begin{bmatrix}
        ~ \sum_{(a,a') \in \cA_1 \times \cA_1} \lambda(a,a') & \sum_{(a,b) \in \cA_1 \times \cA_2} \lambda(a,b) 
        \\ 
        \sum_{(b,a) \in \cA_2 \times \cA_1} \lambda(b,a) & \sum_{(b,b') \in \cA_2 \times \cA_2} \lambda(b,b') 
    \end{bmatrix}\right)
    \end{equation}
    the following expansion holds
    \begin{equation}
        \E\{\varphi(X)\psi(Y)\} = \sum_{\tau \in \bX_{\cA_1}} \sum_{\eta \in \bX_{\cA_2}} 
        \frac{\varphi_{\tau}\psi_{\eta}}{\fs(\tau)\fs(\eta)} 
       \sum_{\pi \in \cP_{\bX}(\tau \sqcup \eta)}  \prod_{\{\pi_1,\pi_2\} \in \pi} \lambda(\pi_1, \pi_2)
    \end{equation}
    where $\varphi_{\llbracket \alpha_1 \cdots \alpha_m \rrbracket} := \varphi^{(m)}(0)$ and $\cP_\bX(\tau \sqcup \eta)$ is the set of pairings of the entries of $\tau$ and $\eta$.
\end{proposition}

\begin{proof}
    Using Wick's theorem we have 
    \begin{align*}
        \E\{\varphi(X)\psi(Y)\} &=
        \sum_{M,N=0}^{\infty} \frac{\varphi^{(M)}(0)\psi^{(N)}(0)}{M!N!} \E\{X^M Y^N\}
        \\ &=
        \sum_{M,N=0}^{\infty} \frac{\varphi^{(M)}(0)\psi^{(N)}(0)}{M!N!} 
        \sum_{\pi \in \cP([M]\sqcup[N])} 
        \prod_{\substack{\{\pi_1,\pi_2\} \in \pi \\ \pi_1 < \pi_2}} 
        \E\{Z_{\pi_1} Z_{\pi_2}\}
        \\ &=
        \sum_{M,N=0}^{\infty} \frac{\varphi^{(M)}(0)\psi^{(N)}(0)}{M!N!} 
        \sum_{\pi \in \cP([M]\sqcup[N])} 
        \prod_{\substack{\{\pi_1,\pi_2\} \in \pi \\ \pi_1 < \pi_2}} 
        \sum_{\gamma_1,\gamma_2 \in \cA_1\sqcup\cA_2}
        \lambda_{\pi_1,\pi_2}(\gamma_1,\gamma_2).
    \end{align*}
    Here $[M]\sqcup[N] := ((1,1), \dots, (1,M), (2,1), \dots, (2,N))$ is ordered lexicographically, where with a slight abuse of notation $\cP([M]\sqcup[N])$ is the set of pairings of elements in $[M]\sqcup[N]$ and
    \[
    \lambda_{(i,*),(j,*)}(c,c') =  \lambda(c,c') \mathbf{1}(c \in \cA_i) \mathbf{1}(c' \in \cA_j).
    \]
    The claim would thus follow from the equality
    \begin{align*}
        & \frac{1}{M!N!} 
        \sum_{\pi \in \cP([M]\sqcup[N])} 
        \prod_{\substack{\{\pi_1,\pi_2\} \in \pi \\ \pi_1 < \pi_2}} 
        \sum_{\gamma_1,\gamma_2 \in \cA_1\sqcup\cA_2}
        \lambda_{\pi_1,\pi_2}(\gamma_1,\gamma_2)
        \\
        &= \sum_{\tau \in \bX_{\cA_1}^M} \sum_{\eta\in \bX_{\cA_2}^N} 
        \frac{1}{\fs(\tau)\fs(\eta)} 
        \sum_{\pi \in \cP_\bX(\tau \sqcup \eta)} \prod_{\{\gamma_1,\gamma_2\} \in \pi} 
        \lambda(\gamma_1,\gamma_2),
    \end{align*}
    which we now prove.
    Begin by noticing that
    \begin{align}\label{app:eqn:combinatorica_1}
        & \sum_{\pi \in \cP([M]\sqcup[N])} 
        \prod_{\substack{\{\pi_1,\pi_2\} \in \pi \\ \pi_1 < \pi_2}} 
        \sum_{\gamma_1,\gamma_2 \in \cA_1\sqcup\cA_2}
        \lambda_{\pi_1,\pi_2}(\gamma_1,\gamma_2)
        \\
        &= \sum_{\Gamma_1 : [M] \to \cA_1}\sum_{\Gamma_2 : [M] \to \cA_2}
        \sum_{\pi \in \cP([M]\sqcup[N])} 
        \prod_{\substack{\{\pi_1,\pi_2\} \in \pi \\ \pi_1 < \pi_2}} 
        \lambda(\Gamma(\pi_1),\Gamma(\pi_2))
    \end{align}
    where $\Gamma((i,k)) = \Gamma_i(k)$.
    We are simply double counting: the left-hand side in \eqref{app:eqn:combinatorica_1} is obtained by labeling the elements of $[M]\sqcup[N]$ after having paired them, while the right-hand side is obtained by reversing the order of the operations with the functions $\Gamma_1,\Gamma_2$ specifying the labels.

 Then for any fixed $\tau \in \bX_{\cA_1}^M$ and $\eta \in \bX_{\cA_2}^N$, we have
    \begin{align}\label{app:eqn:combinatorics_2}
        & \sum_{\pi \in \cP_\bX(\tau \sqcup \eta)} \prod_{\{\gamma_1,\gamma_2\} \in \pi} 
        \lambda(\gamma_1,\gamma_2)
        = 
        \sum_{\pi \in \cP([M]\sqcup[N])} 
        \prod_{\substack{\{\pi_1,\pi_2\} \in \pi \\ \pi_1 < \pi_2}} 
        \lambda(\Gamma(\pi_1),\Gamma(\pi_2))
    \end{align}
    whenever $\Gamma_1 : [M] \to \cA_1$ and $\Gamma_2 : [M] \to \cA_2$  satisfy
    \[
    \tau = \llbracket \Gamma_1(1) \cdots \Gamma_1(M) \rrbracket 
    \text{ and }
    \eta = \llbracket \Gamma_2(1) \cdots \Gamma_2(N) \rrbracket.
    \]
    Letting $\Omega_{\tau,\eta}$ denote the set of all such pairs $(\Gamma_1,\Gamma_2)$, we note that
    \[
        |\Omega_{\tau,\eta}|=\frac{M!}{\fs(\tau)}\frac{N!}{\fs(\eta)}
    \]
    and, in turn, that
\begin{equation}\label{app:eqn:combinatorics_3}
        \bigsqcup_{\tau \in \bX_{\cA_1}^M} \bigsqcup_{\eta \in \bX_{\cA_2}^N} \Omega_{\tau,\eta}  = \{ \Gamma_ 1: [M] \to \cA_1 \}\times\{ \Gamma_2 : [N] \to \cA_2 \}.
    \end{equation}
    
    This allows us to conclude that
    \begin{align*}
        & \sum_{\substack{\tau \in \bX_{\cA_1}^M\\ }} \sum_{\eta \in \bX_{\cA_2}^N} 
        \frac{1}{\fs(\tau)\fs(\eta)} 
        \sum_{\pi \in \cP(\tau \sqcup \eta)} \prod_{\{\gamma_1,\gamma_2\} \in \pi} 
        \lambda(\gamma_1,\gamma_2)
        \\
        &=
        \sum_{\tau \in \bX_{\cA_1}^M} \sum_{\eta \in \bX_{\cA_2}^N}   
        \frac{1}{\fs(\tau)\fs(\eta)} \frac{1}{|\Omega_{\tau,\eta}|}
        \sum_{(\Gamma_1,\Gamma_2) \in \Omega_{\tau,\eta}}
        \sum_{\pi \in \cP([M]\sqcup[N])} 
        \prod_{\substack{\{\pi_1,\pi_2\} \in \pi \\ \pi_1 < \pi_2}} 
        \lambda(\Gamma(\pi_1),\Gamma(\pi_2))
        \\
        &= 
        \frac{1}{M!N!}
        \sum_{\tau \in \bX_{\cA_1}^M} \sum_{\eta \in \bX_{\cA_2}^N}   
        \sum_{(\Gamma_1,\Gamma_2) \in \Omega_{\tau,\eta}}
        \sum_{\pi \in \cP([M]\sqcup[N])} 
        \prod_{\substack{\{\pi_1,\pi_2\} \in \pi \\ \pi_1 < \pi_2}} 
        \lambda(\Gamma(\pi_1),\Gamma(\pi_2))
        \\
        &= 
        \frac{1}{M!N!}
        \sum_{\Gamma_1 : [M] \to \cA_1}\sum_{\Gamma_2 : [M] \to \cA_2}
        \sum_{\pi \in \cP([M]\sqcup[N])} 
        \prod_{\substack{\{\pi_1,\pi_2\} \in \pi \\ \pi_1 < \pi_2}} 
        \lambda(\Gamma(\pi_1),\Gamma(\pi_2))
        \\
        &=
        \frac{1}{M!N!} 
        \sum_{\pi \in \cP([M]\sqcup[N])} 
        \prod_{\substack{\{\pi_1,\pi_2\} \in \pi \\ \pi_1 < \pi_2}} 
        \sum_{\gamma_1,\gamma_2 \in \cA_1\sqcup\cA_2}
        \lambda_{\pi_1,\pi_2}(\gamma_1,\gamma_2)
    \end{align*}
    where the first equality follows from (\ref{app:eqn:combinatorics_2}), the second from $|\Omega_{\tau,\eta}| = \frac{M!}{\fs(\tau)}\frac{N!}{\fs(\eta)}$, the third is (\ref{app:eqn:combinatorics_3}) and the last (\ref{app:eqn:combinatorica_1}). 
\end{proof}

 \section{Numerical study of the Jacobian spectral distribution}\label{app:numerics}

 In this section we present some evidence that the results presented in this work hold for non-polynomial activation functions, focusing on the Jacobian spectral distribution. 
 
We consider ReLU networks (\emph{i.e.} with all activations equal to $\varphi(x) = x\cdot\mathbf{1}(x>0)$).
Ignoring the fact that $\varphi$ is not a polynomial and using Proposition \ref{thm:jacobianmoments} (noting that $\mu_{k,\ell}(x) = \frac{1}{2}$ for all $k,\ell$ and $x$ in that case), we can solve the resulting recursion for $m_{4,L}$ and obtain the following closed-form formulae as a function of the network depth $L$:
\begin{gather*}
    m_{1,L} = \frac{1}{2^L} \quad
    m_{2,L} = \frac{1+2L}{4^L}\quad
    m_{3,L} = \frac{6L(L+1)-2L+1}{8^L} \\
    m_{4,L} = \frac{\frac{4}{3}L(16L^2 + 12L + 5) + 1}{16^L}.
\end{gather*}
We note in passing that this seems to suggest a phenomenon of vanishing gradients as $L$ grows, with the numerator in $m_{k,L}$ being a polynomial in $L$ and the denominator being exponential. 

We then compute these moments empirically, and compare the resulting numerical values with those obtained from these expressions. This is compiled in the following table, which shows very good agreement.


\begin{table}[h!]
\begingroup
\renewcommand{\arraystretch}{1.5} 
\centering
\begin{tabular}{cccccc}
\hline
Quantity & L=1 & L=2 & L=3 &                                           \\ \hline
$m_{1,L}$      &  0.501 (0.5)   &  0.251 (0.25)   &  0.126 (0.125)                                   \\
$m_{2,L}$       &  0.752 (0.75)   &  0.316 (0.313)  &  0.111 (0.109)                                 \\
$m_{3,L}$      &   1.384 (1.375)  &  0.529 (0.516)  &  0.135  (0.131)                       \\
$m_{4,L}$       &  2.840 (2.813)   &  1.012  (0.973) &  0.191 (0.181)   \\ \hline
\end{tabular}
\vspace{5pt}
\caption{The first four moments of the limiting spectral distribution of the Jacobian of a depth-$L$ {ReLU} network, for different values of $L$. In the columns, the first value is obtained numerically, by sampling $L$ weight matrices of dimension $500\times 500$ and averaging over $200$ tries, while the value in parentheses is the theoretical prediction (we only kept the first three decimal digits for both).}
\label{table:moments_ReLU}
\endgroup
\end{table}
Next, we independently initialize $M=200$ depth-$2$ networks with weight matrices of dimension $N\times N$ weights, and do so for each $N \in \{50, 100, 250, 500, 700, 1000\}$. 
For each of these networks, we compute the first few empirical spectral moments of $\mathbf{J}_{2,\mathbf{x}}\mathbf{J}_{2,\mathbf{x}}^\top$  with respect to the input $\bx = \mathbf{1}$. In Figure \ref{fig:around_limit}, we plot histograms of the resulting values for each $N$. The picture shows how the convergence to the limiting value is in $L^2$, with the variance around the limit shrinking as the dimension increases. 

 \begin{figure}[h]
    \centering
    \includegraphics[width=0.8\textwidth]{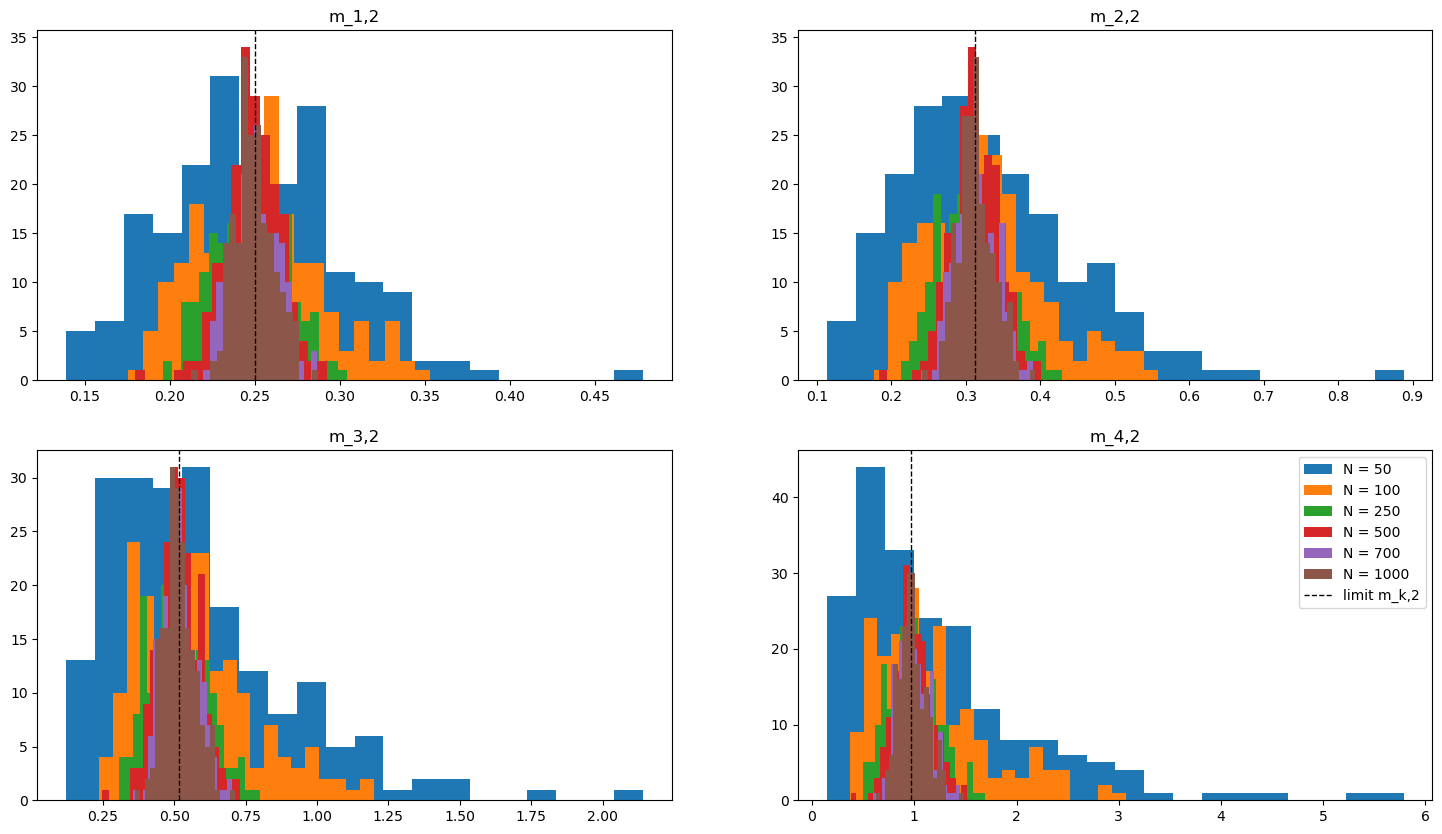}
    \caption{Distribution of empirical Jacobian moments around their limit as $N$ increases.}
    \label{fig:around_limit}
\end{figure}








\newpage
\section{Table of Symbols}
\label{app:sect:table}

\begingroup
\renewcommand{\arraystretch}{1.2} 

\hspace{-2cm}
\begin{tabular}{ |c||p{11.5cm}|  }
 \hline
 Symbol & \multicolumn{1}{| c |}{Meaning} \\
 \hline
 $\Phi_L$           & Neural network of depth $L$ (Eqn. \ref{def:NN}). \\
 $\mathbf{J}_{L,\bx}$ & Neural network Jacobian  $\mathbf{J}_{L,\bx} = \mathrm{d}(\varphi_{L}\circ \Phi_{L-1})_{\bx}$ (Eqn. \ref{def:Jacobian}).\\
 $\fd$              & Dimension function associated to cells of a graph $G$ (Sec. \ref{sec:graph_dictionary}). \\
 $\mathbf{X}_c \in \bR^{\fd(c)}$ & Input of cell $c$ in a graph $G$ (Sec. \ref{sec:graph_dictionary}). \\
 $\mathfrak{C}$ & Set of cell inputs of a product graph (Def. \ref{def:product-graph}). \\
 $\mathcal{F}=(\mathcal{F}_{\mathrm{in}},\mathcal{F}_{\mathrm{out}})$ & Free cells (Sec. \ref{subsec:operator_graph}) of an operator graph, comprised of a sequence of in-cells and a sequence of out-cells.\\
 $ \mathbf{W}_G$ & Operator associated to a graph $G$ with free cells (Def. \ref{def:operator-graph}). If $G$ has no free cells, it is a product graph and $\mathbf{W}_G$ is its value. \\
 $\bullet_{\mathbf{X}_v}$, $\xrightarrow[]{\mathbf{X}_e}$       & Cell fixed to $\mathbf{X}_c$ (Sec. \ref{subsec:operator_graph}). \\
 $\bullet$, $\rightarrow$       & Cell fixed to $\mathbf{1}$ or $\mathbf{I}$ (Sec. \ref{subsec:operator_graph}). \\
 $\text{\begin{tikzpicture}[x=0.75pt,y=0.75pt,yscale=-.5,xscale=.5]

\draw  [color={rgb, 255:red, 0; green, 0; blue, 0 }  ,draw opacity=1 ][fill={rgb, 255:red, 255; green, 255; blue, 255 }  ,fill opacity=1 ][line width=1]  (15.62,19.02) .. controls (17.37,17.38) and (17.47,14.63) .. (15.83,12.87) .. controls (14.19,11.11) and (11.43,11.01) .. (9.67,12.66) .. controls (7.92,14.3) and (7.82,17.05) .. (9.46,18.81) .. controls (11.1,20.57) and (13.86,20.66) .. (15.62,19.02) -- cycle ;

\end{tikzpicture}}\hspace{-3pt}$, $\dashrightarrow$      & In-cell (Sec. \ref{subsec:operator_graph}). \\
 $\text{\begin{tikzpicture}[x=0.75pt,y=0.75pt,yscale=-.5,xscale=.5]

\draw  [color={rgb, 255:red, 74; green, 144; blue, 226 }  ,draw opacity=1 ][fill={rgb, 255:red, 255; green, 255; blue, 255 }  ,fill opacity=1 ][line width=1]  (15.62,19.02) .. controls (17.37,17.38) and (17.47,14.63) .. (15.83,12.87) .. controls (14.19,11.11) and (11.43,11.01) .. (9.67,12.66) .. controls (7.92,14.3) and (7.82,17.05) .. (9.46,18.81) .. controls (11.1,20.57) and (13.86,20.66) .. (15.62,19.02) -- cycle ;

\end{tikzpicture}}\hspace{-3pt}$, {\color{rgb, 255:red, 74; green, 144; blue, 226 } $\dashrightarrow$}      & Out-cell (Sec. \ref{subsec:operator_graph}).\\
 $\bullet_\bx \mapsto \bullet_\by$       & Switching the input of a cell (Def. \ref{def:bT_L(x)}). \\
 $G_1 \wedge G_2$   & Out-vertex identification (corresponds to a Hadamard product) (Def. \ref{def:graph_wedge}). \\
 $\bX_{\mathcal{A}}^M$   & Space of multisets of elements of $\mathcal{A}$ of cardinality $M$ (Eq. \ref{eq:multiset_symm}). \\
  $\mathfrak{s}(\cdot)$   & Symmetric factor of a multiset (Eq. \ref{eq:multiset_symm}). \\
 $\bT_L$        &  Input-independent operator trees space for NN expansion (Sec. \ref{sec:tree_exp_defs}). \\
 $\bT_L(\bx)$        & Input-dependent operator trees space for NN expansion (Def. \ref{def:bT_L(x)}). \\
 $(\bT_L(\bx))*$        & Input-dependent operator trees with pruned root (Sec. \ref{sec:related_expansions}). \\
 $\partial_{\mathbf{x}}\bT_L(\bx)$   & Space of derivative operator trees with respect to input $\mathbf{x}$ (Sec. \ref{sec:related_expansions}). \\
  $\partial_{\ell}\bT_L(\bx)$   & Space of derivative operator trees with respect to matrix $W_\ell$ (Sec. \ref{sec:related_expansions}). \\
 $s(\cdot)$   & Symmetric factor of trees (Sec. \ref{sec:tree_exp_defs}).\\
$\cP(G)$ & Admissible pairings of edges of $G$ with random inputs (Def. \ref{def:admissible_pairings}).\\
$\mathcal{P}_{\mathrm{A}}(G), \mathcal{P}_{\mathrm{B}}(G), \mathcal{P}_{\mathrm{AF}}(G)$ & Atomic, bi-atomic and atom-free pairings of $G$ (Def. \ref{def:various_partitions}). \\
$G_\phi$ & Graph resulting from identification of $G$'s edges according to $\phi$ (Sec. \ref{sec:genus}).\\
 \hline
\end{tabular}

\endgroup


\vfill

\end{document}